\documentclass[12pt]{article}

\pdfoutput=1 

\usepackage{multirow}
\usepackage{verbatim}
\usepackage{makecell}
\usepackage{multicol}
\usepackage{multirow} 
\usepackage{diagbox}
\usepackage{amsmath}
\usepackage{mathrsfs}
\usepackage{amsfonts}
\usepackage{amssymb}
\usepackage{diagbox}
\usepackage{footnote}
\usepackage{graphicx}
\usepackage{float}
\usepackage{epstopdf}
\usepackage[OT1]{fontenc}
\usepackage{geometry}
\usepackage{setspace}
\usepackage{tikz}
\usepackage{pgfplots}
\usepackage[round]{natbib}
\usepackage{multirow}
\usepackage[toc,page]{appendix}
\usepackage{hyperref}
\usepackage{caption}
\usepackage{subcaption}
\usepackage{enumitem}
\usepackage{booktabs}
\usepackage{makecell}
\usepackage{multicol}
\usepackage{multirow} 
\usepackage{dsfont}
\usepackage[figuresright]{rotating}
\allowdisplaybreaks[4]
\providecommand{\U}[1]{\protect\rule{.1in}{.1in}}

\hypersetup{
	bookmarks=true,
	colorlinks=true,
	linkcolor=blue,
	citecolor=blue,
	filecolor=magenta,
	urlcolor=blue,
}

\newcommand{\bX}{\mathbf{X}}
\newcommand{\bY}{\mathbf{Y}}
\newcommand{\bT}{\mathbf{T}}
\newcommand{\bB}{\mathbf{B}}
\newcommand{\bL}{\mathbf{L}}
\newcommand{\bS}{\mathbf{S}}
\newcommand{\bR}{\mathbf{R}}
\newcommand{\bU}{\mathbf{U}}
\newcommand{\bbeta}{\boldsymbol{\beta}}
\newcommand{\mP}{\mathbb{P}}
\newcommand{\E}{\mathbb{E}}
\newcommand{\Var}{\mathrm{Var}}
\newcommand{\Cov}{\mathrm{Cov}}
\newcommand{\cH}{\mathcal{H}}
\newcommand{\cB}{\mathcal{B}}
\newcommand{\cC}{\mathcal{C}}
\newcommand{\brho}{\boldsymbol{\rho}}
\newcommand{\etab}{\boldsymbol{\eta}}
\newcommand{\bh}{\mathbf{h}}
\newcommand{\bZ}{\mathbf{Z}}
\newcommand{\bW}{\mathbf{W}}
\newcommand{\bx}{\mathbf{x}}
\newcommand{\by}{\mathbf{y}}
\newcommand{\bz}{\mathbf{z}}
\newcommand{\be}{\mathbf{e}}
\newcommand{\bO}{\mathcal{O}}
\newcommand{\cG}{\mathcal{G}}
\newcommand{\dto}{\stackrel{\mathcal{L}}{\rightsquigarrow}}
\newcommand{\pto}{\stackrel{p}{\to}}
\newcommand{\asto}{\stackrel{a.s.}{\to}}
\newcommand{\mR}{\mathbb{R}}
\newcommand{\ind}{\mathds{1}_{\left\{\bX_1^{\bh} \le x\right\}}}
\newcommand{\indi}{\mathds{1}_{\left\{\bX_i^{\bh} \le x\right\}}}
\newcommand{\indx}{\mathds{1}_{\left\{\bX^{\bh} \le x\right\}}}
\newcommand{\indu}{\mathds{1}_{\left\{\bU^{\bh} \le u\right\}}}
\newcommand{\induul}{\mathds{1}_{\left\{u_1<\bU^{\bh} \le u\right\}}}
\newcommand{\induur}{\mathds{1}_{\left\{u<\bU^{\bh} \le u_2\right\}}}
\newcommand{\cF}{\mathcal{F}}
\newcommand{\bE}{\mathbf{E}}
\newcommand{\cZ}{\mathcal{Z}}
\newcommand{\btheta}{\boldsymbol{\theta}}

\newtheorem{theorem}{Theorem}[section]

\newtheorem{corollary}{Corollary}[section]

\newtheorem{lemma}{Lemma}[section]

\newtheorem{remark}{Remark}[section]

\newenvironment{proof}[1][Proof]{\noindent\textbf{#1.} }{\ \rule{0.5em}{0.5em}}

\begin{document}
	
	\title{Testing linearity in semi-functional partially linear regression models}
	
	\author{Yongzhen Feng\footnotemark[1]\\ Tsinghua University \and Jie Li\footnotemark[2]\\ Renmin University of China\and Xiaojun Song\footnotemark[3]\\ Peking University
	}
	\date{}
	
	\renewcommand{\thefootnote}{\fnsymbol{footnote}}
	\footnotetext[1]{Center for Statistical Science and Department of Industrial Engineering, Tsinghua University, Beijing, 100084, China.}
		\footnotetext[2]{School of Statistics, Renmin University of China, Beijing, 100872, China. Email: lijie\_stat@ruc.edu.cn.}
	\footnotetext[3]{Department of Business Statistics and Econometrics, Guanghua School of Management and Center for Statistical Science, Peking University, Beijing, 100871, China. Email: sxj@gsm.pku.edu.cn.}
	\maketitle

	\begin{abstract}
	This paper proposes a Kolmogorov--Smirnov type statistic and a Cram\'{e}r--von Mises type statistic to test linearity in semi-functional partially linear regression models. Our test statistics are based on a residual marked empirical process indexed by a randomly projected functional covariate, %to test linearity in semi-functional partially linear regression models.  With the almost surely equivalence between the null hypothesis and the projected hypothesis conditionally on the projected functional covariate,  we construct the test statistics  based on  continuous functionals over the projected process. 	
%Our test statistics are free of user-chosen tuning parameters such as bandwidth and 
which is able to circumvent the ``curse of dimensionality'' brought by the functional covariate. The asymptotic properties of the proposed test statistics under the null, the fixed alternative, and a sequence of local alternatives converging to the null at the $n^{1/2}$ rate are established. %They are able to detect a broad class of local and fixed alternatives converging to the null at the parametric rate $n^{-1/2}$, where $n$ is the sample size. 
A straightforward wild bootstrap procedure is suggested to 
%The weak convergence of the proposed tests is obtained under mild assumptions, and wild bootstrap resampling is applied to calibrate the null distributions of the testing statistics. 
estimate the critical values that are required to carry out the tests in practical applications. Results from an extensive simulation study show that our tests perform reasonably well in finite samples. %Through extensive Monte Carlo simulations, we examine the finite sample performance of our tests. 
Finally, we apply our tests to the Tecator and AEMET datasets to check whether the assumption of linearity is supported by these datasets. %the proposed method is illustrated by analyzing the Tecator and AEMET data sets.
	\end{abstract}
	
	
	
	\thispagestyle{empty}
	\newpage 
	\renewcommand{\thefootnote}{\arabic{footnote}}
	
	\section{Introduction}
	Functional Data Analysis (FDA) has gained increasing attention over the last two decades due to the frequently encountered type of data recorded continuously during a time interval or intermittently at several discrete time points, see, e.g., \cite{RS05}, \cite{ferraty2006nonparametric},  \cite{horvath2012inference}, and more recently   \cite{wang2016functional}, 	for the development of its theory and applications.
	
	%for basic methods, theories, and applications of FDA and \cite{horvath2012inference} for statistical inference with some FDA methods. The literature on FDA is particularly extensive and \cite{wang2016functional} is recommended for a comprehensive overview.
	
	In FDA,   an active area of research focuses on the functional linear model (FLM),  %functional data are related to scalar variables in many cases and it's of great interest 
	which assesses the relationship between the functional covariate and other variables via a regression model. The simplest, yet most commonly used model is FLM with a scalar response. A partial list of literature on this topic includes \cite{yao2005functional} and \cite{cai2006prediction}
	for estimation and prediction by functional principal component analysis (FPCA) approach, 
	and \cite{cardot2003testing} for  statistical inference based on hypothesis testing. However,
	sometimes a single functional covariate is far from enough to explain the response efficiently.
	Take, for example,  the Tecator data set  which consists of  the fat, water, and protein content of $215$ finely chopped meat samples, 
	as well as $215$ spectrometric curves measuring the absorbance at  850nm-1050nm wavelength.
	The goal is to predict the fat content of a meat sample using other variables. If only the spectrometric curve is included as a predictor,
	the proposed test of \cite{cuesta2019goodness}   leads to a rejection of the null of correct specification of FLM, namely, there is no statistically significant evidence that the absorbance curves can sufficiently interpret the fat content.
	
	%	In some situations, a single functional covariate is far from enough to explain the response. 
	To address the above issue brought by \textquotedblleft mixed data\textquotedblright, which indicates that both a vector of finite length random variables and a function-valued random variable on each individual are of great interest,
	three main streams of models containing both functional and scalar covariates  have been proposed.  
	The first one is  a partially functional linear regression model with    linearity in both functional and scalar predictors, see, e.g., 	\cite{shin2009partial}, \cite{kong2016partially} and \cite{li2020inference}.
	The second kind combines the nonparametric regression for scalar covariates with a standard FLM component, 
	which is called the functional partial linear regression model,
	see, e.g., \cite{lian2011functional}.  In this paper, we focus our attention on the third type, 
	the semi-functional partially linear regression model (SFPLR), which has the form given as follows:
	\begin{equation} 
		\label{SFPLR}
		Y  = \gamma(Z_1,\cdots, Z_p, \bX) + \varepsilon = \bZ^{\top} \bbeta + m(\bX) +\varepsilon, 	
	\end{equation}
	where $\bZ = \left(Z_1, \cdots, Z_p\right)^{\top}$ is a vector of real explanatory variables, $\bX$ is another explanatory variable but of functional nature, $\varepsilon$ is a random error satisfying $\E\left(\varepsilon| \bZ,\bX\right) = 0$ almost surely ($a.s.$) by construction, $\bbeta = \left(\beta_1, \cdots, \beta_p\right)^{\top}$ is a vector of unknown real parameters and $m(\cdot)$ is an unknown %smooth real 
function.

The SFPLR model enjoys wide application since it perfectly captures the flexibility of the nonparametric functional model and  the interpretability of standard linear regression.
 Parameter estimation and its  corresponding asymptotic properties in SFPLR have been well developed, see, e.g., \cite{aneiros2006semi}, \cite{aneiros2008nonparametric}, \cite{boente2017robust}, and \cite{aneiros2018bootstrap}, while the goodness-of-fit tests of SFPLR are rather rare in the literature. 
 The term \textquotedblleft goodness-of-fit test\textquotedblright \  was first coined by Pearson for testing if a data distribution
  belongs to a certain parametric family and has become an important step in model analysis, %advanced considerably  ever since the two essential works of 
see, e.g.,  \cite{bickel1973some} and \cite{durbin1973weak} for basic ideas and \cite{gonzalez2013updated}  for a comprehensive review. In this paper, the goodness-of-fit goal  is to test the linear functional 
  nature of the SFPLR model with the  null hypothesis  given by 
	\begin{equation}
	\label{hypothesis}
	H_0: m(\bx) = \langle \bx,\brho \rangle  \quad \text{for all}~\bx\in\cH \text{ and for some}~ \brho \in \cH,
	\end{equation}
while the alternative hypothesis $H_1$ is the negation of $H_0$, namely, $H_1: m(\bx) \neq \langle \bx,\brho \rangle$ for some $\bx\in\cH$ and for any $\brho \in \cH$. Here, $\cH$ is a general separable Hilbert space endowed with the inner product $\langle \cdot,\cdot \rangle$, and the functional covariate $\bX$ in \eqref{SFPLR} also takes value in $\cH$. It is noteworthy that testing $H_0$ against $H_1$ is quite general within the framework of SFPLR in \eqref{SFPLR},  which includes as a special case testing the significance of functional covariate $\bX$ on $Y$ when $\brho = \mathbf{0}$ in \eqref{hypothesis} and reduces to the goodness-of-fit test for FLM adequacy when $\bbeta = 0$ in \eqref{SFPLR}.

 Testing linearity is  of great significance since many datasets  in %economics or other fields 
 scientific fields are not large enough to guarantee accurate nonparametric estimation and the partially linear models provide feasible and flexible alternatives in the presence of high-dimensional covariates.  Usually, some of the covariates are likely to  enter the model linearly. In terms of testing the linearity for regression models, the first type focuses on local smoothing-based tests using distances between estimated regression functions under the null and under the alternative, which was employed in  \cite{hardle1993comparing}, \cite{Fan1996}, \cite{zheng1996consistent} among many others.  Another perspective is to consider proper norms of the distance between the estimated integrated regression function and its version under the null, leading to tests based on the residual marked empirical process, see, e.g., \cite{stute1997nonparametric} and \cite{stute1998model}. However, it is worth noting that the goodness-of-fit tests in the presence of functional covariates are always intricate since the power performance may be greatly deteriorated due to the infinite dimensionality nature of functional data.
  \cite{cardot2003testing}, \cite{delsol2011structural}, and \cite{hilgert2013minimax} investigated the significance testing of the functional covariate $\bX$ on $\bY$ from different perspectives.
  %Due to the complexity of functional data processing, the literature on GOF tests involving functional variables is rather sparse.
 In the context of the FLM goodness-of-fit test,  %a class of parametric regression functions with a functional covariate,  
  \cite{patilea2012projection}, enlightened by \cite{escanciano2006consistent} for the finite-dimensional predictors, put forward the idea of random projections to alleviate the complexity of functional data and proposed a functional version of the local smoothing-based test to check the linearity in the functional covariate, while \cite{garcia2014goodness} constructed a test based on the projected empirical process taking advantage of the Cram\'{e}r--von Mises norm without theory. 
   \cite{cuesta2019goodness} successfully  derived the  weak convergence 
   results by employing the random projection methodology for marked empirical process and utilizing the functional-coefficient estimation proposed in \cite{cardot2007clt}. However, the testing procedures in all the above works were investigated under the pure FLM and lacked theoretical power analysis for the alternatives. To our best knowledge, there is no literature on goodness-of-fit tests for linearity under the framework of the SFPLR model. Therefore, a practical, computationally efficient, and theoretically reliable testing method is urgently called for to deal with the analysis of the SFPLR model.
	
%The effects of the association between the mixed type covariates is considered a major generalization. 
In this paper, we employ random projections to test the null hypothesis in order to overcome the well-known ``curse of dimensionality'' of the functional covariate, which is achieved by considering the inner product of the functional variable $\bX$ and a suitable family of projectors $\bh\in\cH$. Lemma \ref{lemma0} of Section \ref{method} indicates that only a finite number of random projections is enough for the characterization of the null hypothesis, making it feasible in practice. To construct tests from the residual marked empirical process based on projections, which are robust to the unknown dependence between the real and functional covariates, a two-step procedure is proposed for parameter estimation. Specifically, one first estimates $\bbeta$ under the alternative model, namely, one estimates $\bbeta$ under the SFPLR model \eqref{SFPLR}. The reason is that the null model and the alternative model  share identical  partially linear terms,  then the robust estimator $\tilde\bbeta$ obtained from the SFPLR model [see equation \eqref{beta_estimate}] is always consistent irrespective of the model functional form of $m(\bX)$. %in spite of model misspecification.  
After plugging $\tilde\bbeta$ into the model \eqref{SFPLR}, it reduces to a classical FLM and one can easily get the regularized estimate for $\brho$. 

Finally, our test statistics are built via the classical Kolmogorov--Smirnov (KS) and the Cram\'{e}r--von Mises (CvM) norms adapted to the residual marked empirical process, with the added bonus of $n^{1/2}$-rate weak convergence as established in Theorem \ref{theorem1}. Indeed, different from the local smoothing-based tests, our proposed tests are global in nature, with their limiting null distributions enjoying faster convergence rates and free of user-chosen tuning parameters (such as the bandwidths required in the local smoothing-based tests whose proper choice may depend on the underlying data-generating process), and thus our tests are more robust in practical data analysis. In addition, the proposed test statistics are easy to compute and the corresponding critical values can be estimated using a straightforward wild bootstrap procedure. %simple to calibrate in its distribution by a wild bootstrap on the residuals.	
The asymptotic properties of our tests under a fixed alternative hypothesis and a sequence of the local alternative hypotheses converging to the null hypothesis at the $n^{1/2}$ rate are also investigated. 
To minimize the potential influence of the projection direction $\bh$ and to achieve higher testing power,  in the simulations and the real data analysis we suggest choosing %a number of 
$K=7$ different random directions and then adjust the final $p$-value by the false discovery rate (FDR) method proposed in  \cite{benjamini2001control} %to achieve higher testing power. 
Extensive simulations yield attractive results in terms of empirical sizes and powers that strongly corroborate the asymptotic theory. 
  %Our tests are able to detect a broad class of local alternatives at a parametric rate of root-n. The wide bootstrap on the residuals is implemented in the simulation section and $K$ different directions of random projections are applied for the lower influence of the choice of $\bh$ and higher power based on the false discovery rate (FDR) introduce by \cite{benjamini2001control}. And the empirical analysis reports competitive results.

The rest of the paper is organized as follows. The testing framework including the hypothesis projection and the two-step estimation procedure is introduced in Section \ref{method}.  Section \ref{asy} presents the asymptotic properties of the proposed test statistics under the null, the alternative, and a sequence of local alternatives. %under mild assumptions. 
The practical aspects of implementing the tests are given in Section \ref{practical_aspects}, with a detailed discussion of parameter estimation, selection of projection directions, and estimation of critical values through the wild bootstrap procedure.
%Section 4 is the implementation part covering parameter estimation, projection directions selection, and obtaining critical values by the wild bootstrap procedure. 
Section \ref{sim} reports simulation findings, together with the empirical analysis of  the Tecator and AEMET datasets. Finally, Section \ref{con} concludes the paper.  %Concluding remarks is in section 6. 
All mathematical proofs are collected in the Appendix.
%The Appendix presents the main proofs, whereas the Supplementary Material contains the auxiliary lemmas and further results from the simulation study.

	\section{Methodology}
	\label{method}
Throughout this paper, given $\bh\in\cH$ ($\bh$ can be a random element), we denote by $\bX^{\bh}=\langle \bX,\bh \rangle$ the projected $\bX$ on the direction $\bh$. For any $\bX\in \cH$, denote its norm by $\Vert \bX \Vert=\langle \bX,\bX \rangle^{1/2}$.  For any $p$-dimensional real vector $\mathbf{a}=(a_{1},\ldots,a_{p})^\top \in \mathbb{R}^{p}$ with $p\geq 1$, denote the Euclidean norm by $\Vert \mathbf{a}\Vert =(a_{1} ^{2}+ \cdots  +a_{p} ^{2})^{1/2}$.
	
	\subsection{Hypothesis projection}	
	In this section, we introduce a (mixed) residual marked empirical process indexed by random projections of the functional covariate $\bX$ to construct robust specification tests for the linearity of $m(\bX)$ in the SFPLR model \eqref{SFPLR}. %goodness-of-fit test. %that will serve as a basic test process. 
	To this end, note that the null hypothesis $H_0$ in (\ref{hypothesis}) %we presented above 
	can be equivalently expressed as $H_0:\E[Y - \bZ^{\top}\bbeta - \bX^{\brho}  \vert \bX] = 0$ $a.s.$, which can also be characterized by means of the associated projected hypothesis on a randomly chosen direction $\bh\in\cH$, defined as $H_{0}^{\bh}: \E[Y -\bZ^{\top} \bbeta -\bX^{\brho} \vert \bX^{\bh}] = 0$ $a.s.$. For notational simplicity, we denote by $U=Y - \bZ^{\top} \bbeta - \bX^{\brho}$ in the rest of the article. Note that the error term $\varepsilon$ in \eqref{SFPLR} satisfies $\varepsilon=U$ $a.s.$ under the null $H_0$. The following important lemma specifies the necessary and sufficient condition such that  $\E[U  \vert \bX] = 0$ holds $a.s.$ based on projections of $\bX$.
	
	\begin{lemma} (Theorem 2.4, \cite{cuesta2019goodness}) \label{lemma0}
		Let $\mu$ be a non-degenerate Gaussian measure on $\cH$ and $\bX$
		be a $\cH$-valued random variable (r.v.) defined on a probability space $(\Omega,\sigma,\nu)$. 
		Assume that $m_k:=\int\Vert \bX\Vert^k d\nu <\infty$ for all $k\ge1$ with $\sum_{k=1}^\infty m_k^{-1/k}=\infty$, $\E \left\Vert \bZ \right\Vert ^2 < \infty$,  
		and $\E [Y^2]< \infty$.
		Denote $\mathcal{A}_0 := \{\bh \in \cH: \E[U| \bX^{\bh}]=0 \  a.s. \} $, 
		then %by the proposition given later, we have
		\begin{equation*}
			\E\left[U \vert \bX\right] = 0 \quad \ a.s. \iff \mu(\mathcal{A}_0) > 0.
		\end{equation*}
	\end{lemma}

  \begin{remark}
    It is obvious that %the holding of 
    $H_0$ holds ensures that $H_0^{\bh}$ holds for every $\bh \in \cH$. And the above lemma indicates that if $H_0$ fails, then $\mu(\mathcal A_0)=0$, implying that with probability one, one can choose a projection $\bh$ such that $H_0^{\bh}$ fails. Thus we build the $\mu$-$a.s.$ equivalence between the original null hypothesis $H_0$ and its projected version $H_0^{\bh}$. 
  \end{remark}
	This enables us to test the null hypothesis $H_0$ by first randomly choosing a projection direction $\bh\in\cH$ and then testing the projected null hypothesis conditional on $\bh$, namely $H_0^{\bh}:\E[U \vert \bX^{\bh}] = 0 $ $a.s.$. It is also clear that in $H_0^{\bh}$ the conditioning variable $\bX^{\bh}$ is real, which circumvents the ``curse of dimensionality'' of functional covariate $\bX$, greatly simplifying the testing problem. Nevertheless, sometimes it is possible that the power of the resulting tests could be sensitive to the selected projection. To minimize the influence of the projection direction and to enhance testing power, we suggest choosing several different directions. A detailed selection procedure on the projection directions is discussed in Section \ref{practical_aspects}.
 
 %As such, how to choose the projection direction $\bh$ is also crucial in practice. We discuss the detailed selection procedure in Section 4. 
	
	In light of Lemma \ref{lemma0}, given $n\geq 1$ independent and identically distributed  (i.i.d.) observations $\{Y_i, \bZ_{i}, \bX_i\}_{i=1}^n$ with $\bZ_i=(Z_{i1},\cdots, Z_{ip})^{\top}$, we consider the following mixed-type residual marked empirical process built upon the randomly projected functional covariate $\bX_{i}^{\bh}$:
	\begin{equation}
		T_{n,\bh}(x) := \frac{1}{\sqrt n} \sum_{i=1}^{n} {\mathds{1}}_{\left\{\bX_{i}^{\bh} \le x\right\}} \left(Y_i -   \bZ_{i}^{\top} \tilde{\bbeta} - \bX_{i}^{\hat{\brho}}\right), \quad x\in\mathbb R,\label{rmep}
	\end{equation}
where the  estimator $\tilde{\bbeta}$ for $\bbeta$ and the estimator $\hat{\brho}$ for $\brho$ are given in Section 2.2, and ${\mathds{1}}_{\{A\}}$ is the indicator function of the event $A$. To guarantee the consistency of the tests based on $T_{n,\bh}(x) $ against all fixed alternatives, it is crucial that $\tilde{\bbeta}$ should be a robust-type estimator in the sense that it is always consistent regardless of whether the null hypothesis is satisfied. That is, $\tilde\bbeta$ should be obtained under the SFPLR model in \eqref{SFPLR} rather than under the null model $Y  = \bZ^{\top} \bbeta + \bX^{\brho} +U$ with $U=\varepsilon$ $a.s.$, given that $\bZ$ and $\bX$ may have some unknown dependence structure.
	%where $a_n \to 0$ is a normalizing positive sequence to be determined later.
	
	Our test statistics are suitable continuous functionals of $T_{n,\bh}(x)$. In this paper, we focus on the popular KS type and  CvM type statistics, which are given by 
	$$\Vert T_{n,\bh} \Vert_{KS} := \sup_{x \in \mR} \vert T_{n,\bh}(x) \vert$$ 
	and 
	$$\Vert T_{n,\bh} \Vert_{CvM} := \int_{\mR} T_{n,\bh}(x)^2 dF_{n,\bh}(x),$$
	respectively, where $F_{n,\bh}(x)=n^{-1}\sum_{i=1}^n {\mathds{1}}_{\{\bX_{i}^{\bh} \le x\}}$ is the empirical distribution function (EDF) based on the randomly projected functional covariate $\{\bX_{i}^{\bh}\}_{i=1}^n$. The null hypothesis $H_0$ is rejected whenever the test statistics $\Vert T_{n,\bh} \Vert_{KS}$ and $\Vert T_{n,\bh} \Vert_{CvM}$ exceed some ``large'' values, which are consistently estimated using a wild bootstrap procedure, as described in Section 4.3.

	\subsection{Two-step Estimation Procedure}%Parameter  estimation} 
	
	Throughout this paper,  denote by $\mathcal H'$ the space of continuous linear operators defined in $\mathcal H$ and valued in $\mathbb R$. By Riesz's representation theorem, one can identify the spaces $\mathcal H$ and $\mathcal H'$, together with the induced norm of $\mathcal H'$
	%$\mathcal H'$ the norm induced 
	by the following identification condition: for $\mathbf T \in \mathcal H'$, $\Vert \mathbf T\Vert_{\mathcal H'} = \Vert \boldsymbol\tau \Vert_{\mathcal H}$, where $\boldsymbol \tau$ is the unique element in $\mathcal H$ such that $\mathbf T (x) = \langle\boldsymbol \tau,\bx \rangle$, $\bx \in \mathcal H$. For $\mathcal H$-valued continuous linear operator $\boldsymbol \Phi$ and continuous linear operator $\boldsymbol \Psi$ defined in $\mathcal H$, with a slight abuse of notation, define $\boldsymbol\Phi\boldsymbol\Psi(\bz) := \boldsymbol\Psi\left(\boldsymbol\Phi(\bz)\right)$, $\bz \in \mathcal H$,
	%$\boldsymbol\Phi\boldsymbol\Psi$ as 
	representing $\boldsymbol\Psi$ composed with $\boldsymbol\Phi$.
	%, that is $\boldsymbol\Phi\boldsymbol\Psi(\bz) := \boldsymbol\Psi\left(\boldsymbol\Phi(\bz)\right)$, $\bz \in \mathcal H$.
	
	In order to construct the feasible empirical process $T_{n,\bh}(x)$ in \eqref{rmep},   we first need to obtain an appropriate estimator for $\bbeta$ in the SFPLR model (\ref{SFPLR}).
	%We begin with the estimation of $\bbeta$ in SFPLR model (\ref{SFPLR}). 
	%Take conditional expectation w.r.t. $\bX$ on both side of (\ref{SFPLR}), we have 
	%\begin{equation*}
	%   E(Y \vert \bX) = E\left( \sum_{j=1}^{p} Z_j\beta_j \vert \bX \right) + m(\bX) + E(\varepsilon \vert \bX)
	%\end{equation*}
	%Subtract (\ref{SFPLR}) by the above equation and note that $E(\varepsilon \vert \bX)=0$, we get
	Note that (\ref{SFPLR}) implies
	\begin{equation*}
		Y-\E(Y \vert \bX) = \sum_{j=1}^{p} \left[ Z_j - \E(Z_j \vert \bX) \right] \beta_j + \varepsilon.
	\end{equation*}
	Then, after plugging in the nonparametric estimators for $\E(Y \vert \bX)$ and $\{\E(Z_j \vert \bX)\}_{j=1}^{p}$,  %we can obtain 
	the resulting least squares type estimator for $\bbeta$  is given by %by the least square solution
	\begin{equation}
		\label{beta_estimate}
		\tilde{\bbeta} = \left(\widetilde{\bZ}_{b}^{\top}\widetilde{\bZ}_b\right)^{-1}\widetilde{\bZ}_{b}^{\top}\widetilde{\bY}_b,
	\end{equation}
	where the subscript $b=b_n\in\mathbb R^+$ represents the bandwidth parameter that converges to zero at some appropriate rate as the sample size $n$ diverges to infinity. In \eqref{beta_estimate}, $\widetilde{\bZ}_b=\cZ-\bW_b\cZ :=(\widetilde{\bZ}_1,\ldots,\widetilde{\bZ}_n)^{\top}$ and $\widetilde{\bY}_b=\bY-\bW_b\bY:=(\widetilde Y_1,\ldots,\widetilde Y_n)^\top$, where $\cZ=(\bZ_1,\ldots,\bZ_n)^{\top}$, $\bY=(Y_1,\ldots,Y_n)^{\top}$, and $\bW_b=(w_{n,b}(\bX_i,\bX_j))_{i,j}$ is an $n\times n$ matrix with $w_{n,b}(\cdot,\cdot)$ being the functional covariate version of the Nadaraya--Waston weighting function that has the following form:
	\begin{equation*}
		w_{n,b}(t,\bX_i)=\frac{K(d(t,\bX_i)/b)}{\sum_{j=1}^{n} K(d(t,\bX_j)/b)},
	\end{equation*}
	where $K(\cdot)$ is a kernel function from $[0,\infty)$ to $[0,\infty)$, and $d(\cdot,\cdot)$ is a semi-metric in Hilbert space $\cH$. It is noteworthy that $\tilde{\bbeta}$ in \eqref{beta_estimate} is a robust type estimator for $\bbeta$ because it is obtained under the unrestricted model  (\ref{SFPLR}) and thus is always consistent regardless of whether the null hypothesis holds or not.
	
	%Having derived the estimator $\tilde{\bbeta}$ of $\bbeta$, we
	Plugging $\tilde{\bbeta}$ into model (\ref{SFPLR}), under $H_0$, the model is transformed into a classical function linear model
	\begin{equation*}
		\widetilde{D}= \bX^{\brho} + v,
	\end{equation*}
	where  $\widetilde{D}=Y-\bZ^{\top}\tilde{\bbeta}$ is the generated (i.e., estimated) dependent variable and $v=\varepsilon+\bZ^{\top}(\bbeta-\tilde{\bbeta})$ is the composite error term.
	Finding an estimator for $\brho$ then means seeking the solution to the following minimization problem:
	\begin{equation*}
		\inf_{\brho \in \cH} \E  \left\lvert \widetilde{D}-\bX^{\brho} \right\rvert^2.
	\end{equation*}
	Enlightened by the simple linear regression, $\brho$ is determined by the  moment equation
	%\begin{equation}
	$\Delta=\Gamma\brho$,
	%\end{equation}
	where $\Delta$ is the cross-covariance operator of $\bX$ and $\widetilde{D}$, defined as $\Delta(\bz):= \E[(\bX \otimes \widetilde{D}) (\bz)]$ for $\bz \in \cH$ and $\Gamma$ is the covariance operator of $\bX$ given by $\Gamma(\bz) := \E[(\bX\otimes \bX)(\bz)]$ for $\bz\in \cH$, with $\otimes$ being the Kronecker operator defined as $(\bx \otimes \by)\bz = \langle \bz,\bx \rangle \by$ for  $\bx ,\bz \in \cH $ and for $\by$ belonging to either $\cH$ or $\mathbb R$.
	
	Obviously, the estimation of $\brho$ needs the reversibility of operator $\Gamma$, which is nonnegative, self-adjoint, and nuclear, and thus Hilbert--Schmidt and thus compact. However, due to the infinite-dimensional nature of Hilbert space $\cH$, a bounded inverse of $\Gamma$ does not exist.
%One effective approach is	
The regularization method proposed in  \citet{cardot2007clt} can effectively tackle this ill-posed issue, for which we 
	first consider  the Karhunen--Lo\'eve expansion of $\bX$, as given by
	\begin{equation} \label{KL}
		\bX = \sum_{j=1}^{\infty} \lambda_j^{1/2}\xi_j\be_j, 
	\end{equation}
	where $\{\be_j\}_{j=1}^{\infty}$ is a sequence of orthonormal eigenfunctions of $\Gamma$ and $\{\xi_j\}_{j=1}^{\infty}$ are centered real r.v.'s such that $\E[\xi_j\xi_{j'}]=\delta_{jj'}$, where $\delta_{jj'}$ is the Kronecker's delta. Assuming the multiplicity of each eigenvalue is one, there then exists a sorted sequence of % nonnull 
	distinct eigenvalues $\lambda_1 > \lambda_2> \cdots > 0$ of $\Gamma$.
	
	%However, due to the  infinite dimension nature of Hilbert space $\cH$, a bounded inverse of $\Gamma$ does not exist. 
	To ensure the existence and uniqueness of $\brho$, Assumptions (B1) and (B2) in the next section are required.
	%We adopted the regularization method proposed by \citet{cardot2007clt}, 
	%to get an available estimation of $\brho$. 
	%the key to which is to derive 
	An empirical finite rank estimator $\Gamma_{n}^{\dag}$ for $\Gamma^{-1}$  can be derived by the following procedure: 
\begin{itemize}	
\item[(i)] 	 Compute the functional principal components (FPC) of $\{\bX_i\}_{i=1}^{n}$, that is, calculate the eigenvalues $\{\hat{\lambda}_j\}$ and eigenfunctions $\{\hat{\be}_j\}$ of $\Gamma_n$ with $\Gamma_n = n^{-1} \sum_{i=1}^{n} \bX_i \otimes \bX_i$;
	
\item[(ii)]  Define a sequence $\{\delta_j\}$, with $\delta_1:=\lambda_1-\lambda_2$ and $\delta_j := \min(\lambda_j-\lambda_{j+1},\lambda_{j-1}-\lambda_j)$ for $j\geq 2$, and consider a sequence of thresholds $c_n \in (0,\lambda_1), n \in \mathbb{N}$, with $c_n \to 0$, then set 
	\begin{equation}
		\label{kn}
		k_n := \sup\left\{ j \in \mathbb{N}: \lambda_j+\delta_j/2 \ge c_n \right\};
	\end{equation}

\item[(iii)] Compute $\Gamma_{n}^{\dag}$ by the following equation:
	\begin{equation*}
		\Gamma_n^{\dag} = \sum_{j=1}^{k_n} \frac{1}{\hat{\lambda}_j}\hat{\be}_j\otimes\hat{\be}_j.
	\end{equation*}
\end{itemize}

Observe that $\Gamma_{n}^{\dag}$ is the population version of $\Gamma^{\dag}$, which is defined as $\Gamma^{\dag}=\sum_{j=1}^{k_n}\lambda_j^{-1}{\be}_j\otimes{\be}_j$. Together with the empirical finite rank estimator $\Delta_n$ for $\Delta$, denoted by $\Delta_n:= n^{-1}\sum_{i=1}^{n} \bX_i \otimes \widetilde{D}_i$, the regularized estimator $\hat{\brho}$ for $\brho$ is given by %and the form of estimator of $\brho$ is given as following:
	\begin{equation}
		\label{rhohat}
		\hat{\brho} := \Gamma_{n}^{\dag}\Delta_{n} = \frac{1}{n} \sum_{j=1}^{k_n}\sum_{i=1}^{n} \frac{\langle \bX_i \widetilde D_i, \hat{\be}_j\rangle}{\hat{\lambda}_j}\hat{\be}_j.
	\end{equation}

Having obtained both estimators $\tilde{\bbeta}$ and $\hat{\brho}$, we can readily compute the residual $\hat{U}_i = Y_i - \bZ_i^{\top} \tilde{\bbeta} - \bX_{i}^{\hat{\brho}}$ for $i=1,\cdots,n$ to construct the residual marked empirical process $T_{n,\bh}(x)$ and the test statistics based on it. It is worthwhile to emphasize again that the residual $\hat{U}_i $ is a mixed type (and thus robust) residual in the sense that $\tilde{\bbeta}$ is obtained under the alternative while $\hat{\brho}$ is obtained under the null. As such, in $\hat{U}_i$ the unknown dependence structure between the real covariate $\bZ$ and the functional covariate $\bX$ has been taken into account through $\tilde{\bbeta}$ even if the alternative hypothesis is true. This is in sharp contrast to the standard residual obtained under the null model $Y  = \bZ^{\top} \bbeta + \bX^{\brho} +U$, say, $\breve{U}_i = Y_i - \bZ_i^{\top} \breve{\bbeta} - \bX_{i}^{\breve{\brho}}$ for some $\breve{\bbeta}$ and $\breve{\brho}$ obtained when the null holds. Indeed, using the mixed type residual $\hat{U}_i$ in our semiparametric model \eqref{SFPLR} is very important and guarantees a consistent testing procedure that is robust to the presence of real covariate $\bZ$ in \eqref{SFPLR}.

	\section{Asymptotic properties}
	\label{asy}
	\subsection{Technical assumptions}
	To study the asymptotic properties of the proposed test statistics based on $T_{n,\bh}(x)$  such as $\Vert T_{n,\bh} \Vert_{KS}$ and $\Vert T_{n,\bh} \Vert_{CvM}$ defined in Section \ref{method}, 
	we impose the following technical assumptions.
	
	\

	%\begin{itemize}
\noindent\emph{Regularity assumptions}
	\begin{itemize}
		\item[(A1)]  $m_k := \int \Vert \bX \Vert^k d\nu < \infty$ for all $k \ge 1$, and $\sum_{k=1}^{\infty} m_{k}^{-1/k} = \infty$. 
		\item[(A2)] The first and the second moments of $\varepsilon$ given $\bX$ and $\bZ$ are equal to $\E\left(\varepsilon \vert \bX,\bZ\right)=0$ $a.s.$ and $\E\left(\varepsilon^2 \vert \bX,\bZ\right)=\sigma_{\varepsilon}^2$ $a.s.$, respectively. 
\end{itemize}	

\noindent Assumption (A1) is a condition similar to that required in Theorem 3.6 of \cite{cardot2007clt} to guarantee the validity of hypothesis projection in Lemma \ref{lemma0}. Assumption (A2) gives the first two moment constraints of $\varepsilon$ given $\bX$ and $\bZ$. In particular, the condition $\E\left(\varepsilon \vert \bX,\bZ\right)=0$ $a.s.$ guarantees that we are testing linearity of the nonparametric functional component within the framework of the SFPLR model in \eqref{SFPLR}. It may be of some interest to test the correct specification of the SFPLR model itself by verifying whether $\E\left(\varepsilon \vert \bX,\bZ\right)=0$ $a.s.$ holds, which we leave as future research. Although not the weakest, the conditional homoskedasticity is much milder compared with the independence assumption imposed in \cite{cuesta2019goodness}. 	

\

\noindent\emph{Assumptions on the estimator $\tilde\bbeta$ in \eqref{beta_estimate}}	
	\begin{itemize}	
		\item[(B1)] $\{\bX_i\}_{i=1}^{n}$ take value in some given compact subset $\cC \subset \cH$ such that $\cC \subset \cup_{k=1}^{\tau_n} \cB(x_k,l_n)$, where $\cB(x,h) = \left\{x' \in \cH: d(x',x) < h\right\}$ (with $d(\cdot,\cdot)$ being a semi-metric in $\cH$), $x_k's$ is a series of points in $\cH$, $\tau_n l_n^\gamma = C$, where $\gamma$ and $C$ are real positive constants and $\tau_n \to \infty$ and $l_n \to 0$, as $n \to \infty$.
		
		\item[(B2)] Define $g_j(x) = \E(Z_{ij}\vert \bX_i=x)$ and $\eta_{ij} = Z_{ij} - \E(Z_{ij} \vert \bX_i)$ for $j=1,\ldots,p$, and $\etab_i = \left(\eta_{i1}, \cdots, \eta_{ip}\right)^{\top}$. For $ (u,v) \in \cC \times \cC$, $ f \in \left\{m,g_1,\cdots,g_p\right\}$, there exist some $C < \infty$ and $\alpha > 0$ such that
		\begin{equation*}
			\vert f(u) - f(v) \vert \le C d(u,v)^{\alpha}.
		\end{equation*}
		There also exists $r \ge 3$ such that $\E\vert \varepsilon_1 \vert^r + \E\vert \eta_{11} \vert^r + \cdots + \E \vert \eta_{1p} \vert^r < \infty$. In addition, $\bB = \E(\etab_1\etab_1^{\top})$ is a $p\times p$ positive definite matrix.
		
		\item[(B3)] There exist a positive valued function $\phi(\cdot)$ on $(0,\infty)$ and positive constants $\alpha_0$, $\alpha_1$ and $\alpha_2$ such that 
		\begin{equation*}
			\int_{0}^{1} \phi(hs) ds > \alpha_0 \phi(h), \, \alpha_1 \phi(h) \le \mP \left(X \in \cB(x,h)\right) \le \alpha_2 \phi(h), \, \forall x \in \cC, \, h>0.
		\end{equation*}

		\item[(B4)]  The kernel function $K(\cdot)$ has support $[0,1]$, is Lipschitz continuous on $[0,\infty]$, and $\exists \, \theta>0$ such that $\forall \, u \in [0,1]$, $-K'(u)>\theta$. The bandwidth $b$ satisfies that  $nb^{4\alpha} \to 0$ as $n \to \infty$ and $\phi(b) \ge n^{\frac{2}{r}+d-1}/(\log n)^2$ for $n$ large enough and  some constant $d>0$ satisfying $2/r+d >1/2$. %\textcolor{red}{Dear both, it seems that only bias has been taken care of through $b\to 0$ and $nb^{4\alpha} \to 0$. What is the variance restriction for the bandwidth $b$, say, $nb \to \infty$? Pls, double-check this.}
\end{itemize}

\noindent Assumptions (B1)-(B4), mainly taken from \cite{aneiros2006semi}, allow us to obtain the standard $n^{1/2}$-rate asymptotic convergence of $\tilde{\bbeta}$ in \eqref{beta_estimate}. The compactness of $\cC$ in (B1) and the constraints in (B2) are regular conditions in the setting of nonfunctional partial linear models, see, e.g., \cite{chen1988convergence}, \cite{bhattacharya1997semiparametric} and \cite{liang2000asymptotic}, while requirements on $\tau_n$ and $l_n$ in (B1) are typical under the framework of functional nonparametric models, see, e.g., \cite{ferraty2006nonparametric}. %\textcolor{red}{Some brief discussions on (B3) would be perfect. } 
(B3) concerns the concentration properties of the small ball probability, which is associated with the semi-metric selection introduced in Chapter 13 of \cite{ferraty2006nonparametric}. %with details. 
The conditions for the kernel function $K(\cdot)$ in (B4) are commonly imposed, 
%conditions on the rate of 
while the rates of the bandwidth $b$ in (B4) are required to study the tradeoff between the bias and variance of the estimator $\tilde{\bbeta}$.  

\
		
\noindent\emph{Assumptions on the estimator $\hat\brho$ in \eqref{rhohat}}
	\begin{itemize}		
		\item[(C1)]  $\bX$, $\bZ$ and $Y$ satisfy $\sum_{j=1}^{\infty} \frac{1}{\lambda_j^2} \langle \E(\bX W),\be_j \rangle^2 < \infty$, for $W=Y$ or $Z_{ij}$, $j=1,2,\cdots,p$.
		\item[(C2)] The kernel of $\Gamma$ is $\left\{\mathbf{0}\right\}$.
		\item[(C3)] $\E \left(\Vert \bX \Vert^2\right) < \infty$ and $\E\left(\Vert\bZ\Vert^2\right) < \infty$.
		\item[(C4)] $\sum_{l=1}^{\infty} \vert \langle \brho,\be_l \rangle \vert < \infty$.
		\item[(C5)] For $j$ large enough, $\lambda_j = \lambda(j)$ with $\lambda(\cdot)$ a convex positive function.
		\item[(C6)] $\frac{\lambda_n n^4}{\log n} = \bO(1)$.
		\item[(C7)] $\inf \left\{\vert\langle\brho,\be_{k_n} \rangle\vert, \frac{\lambda_{k_n}}{\sqrt{k_n \log k_n}}\right\} = \bO(n^{-1/2})$.
		\item[(C8)] $\sup_{j} \left\{ \E\left(\vert \xi_j \vert^5\right)\right\} \le M < \infty$ for some $M \ge 1$.
		%\item \label{A19} There exist $C_1, \, C_2>0$, such that $C_1 n^{-1/2} < c_n < C_2 n^{-1/2}$ for every $n$.
		\item[(C9)] There exist $C_1$, $C_2 > 0$ such that $C_1 n^{-1/2} < c_n < C_2 n^{-1/2}$ for every $n$.

	\end{itemize}
	
\noindent	Assumption (C1) ensures the existence of a solution to $\Delta = \Gamma \brho$, and the set of solution is of the form $\brho + Ker (\Gamma)$, as shown in  \cite{cardot2003testing}. To further simplify the theory development, Assumption (C2) is imposed for identification.
%in order to be identifiable. 
Assumptions (C3)-(C8) are standard for the functional linear models, see, e.g., \cite{cardot2007clt}. In particular, (C5) implies that $\delta_j = \lambda_j - \lambda_{j+1}$ and holds for most classical decreasing rates  for eigenvalues, either polynomials or exponential. As such, (C5) is not restrictive. (C6) is analogous to an assumption in Theorem 2 in \cite{cardot2007clt}. (C7)  controls the order of  $\left\langle \bX, \bL_n \right \rangle$ [see Lemma A.7 in \cite{cuesta2019goodness}]. 
It can be easily deduced from (C8) that $\E (\Vert \bX \Vert^4) < \infty$, which is used to prove Lemma A.3 in the Appendix.  (C9) is able to control the behaviour of $k_n$, as shown by Lemma A.1 in \cite{cuesta2019goodness}, entailing that $k_n^3(\log k_n)^2 = o(n^{1/2})$ together with (C6).

%\begin{remark}A few comments on the regularity conditions are in order.	\end{remark}

\subsection{Asymptotic null distribution}	
	
In this section, we establish the asymptotic property of the projected residual marked empirical process $T_{n,\bh}(x)$ in \eqref{rmep} under $H_0^{\bh}$ as well as those of the test statistics based on $T_{n,\bh}(x)$ such as $\Vert T_{n,\bh}\Vert_{KS}$ and $\Vert T_{n,\bh}\Vert_{CvM}$.
First, we introduce two  necessary conditions for our theory development.

\begin{itemize}		
\item[(i)] 
$\lim_{n} t_{n,\E_{x,\bh}} < \infty$, where  $t_{n,\bx} = \sqrt {\sum_{j=1}^{k_n} \lambda_j^{-1} \left\langle \bx,\be_j\right\rangle^2}$, and $\E_{x,\bh}=\E \left(\ind \bX_1\right)$.

\item[(ii)] $\mathbb{E}\left[\left\Vert\hat\rho-\rho\right\Vert^4\right]=\mathcal{O}\left(n^{-2}\right)$.
\end{itemize}	

The following theorem states that $T_{n,\bh}(x)$ converges weakly to a centered Gaussian process with a complicated covariance function under the projected null $H_0^{\bh}$. %\textcolor{red}{Dear both, pls define the space of functions $D(\mR)$ in the following theorem.}

	\begin{theorem} \label{theorem1}
 Under $H_0^{\bh}$, Assumptions (A1)-(A2), (B1)-(B4) and  (C1)-(C9), additionally with  conditions (i) and (ii),
		it follows that $T_{n,\bh} \dto \cG$ in $D(\mR)$, where  $\cG$ is a centered Gaussian process with covariance function $K(s,t) \equiv C_1(s,t)-C_2(s,t)-C_3(s,t)-C_2(t,s)+C_4(s,t)+C_5(s,t)-C_3(t,s)+C_5(t,s)+C_6(s,t) $ , and $D(\mR)$ is the space of c\`{a}dl\`{a}g functions on $\mR$ that are continuous on the right with limit on the left, where 
        \begin{align*}
		C_1(s,t) &= \int_{\left\{(\bx,\bz): \bx^{\bh} \le s\wedge t \right\}} \Var\left[Y \vert \bX=\bx,\bZ = \bz \right] dP_{(\bX,\bZ)}(\bx,\bz), \\
		C_2(s,t)  &= \int_{\left\{(\bx,\bz): \bx^{\bh} \le s \right\}} \Var\left[Y \vert \bX = \bx, \bZ=\bz \right] \langle \bE_{t,\bh},\Gamma^{-1} \bx\rangle dP_{(\bX,\bZ)}(\bx,\bz), \\
		C_3(s,t)  &= \int_{\left\{(\bx,\bz): \bx^{\bh} \le s \right\}}  \Var\left[Y \vert \bX=\bx,\bZ = \bz \right] \left(\bE_{t,\bh}^{\bZ}+\bE_{t,\bh}^{\bX,\bZ}\right) \bB^{-1} \etab_1 dP_{(\bX,\bZ)}(\bx,\bz),\\
		C_4(s,t)  &= \int \Var\left[Y \vert \bX=\bx,\bZ = \bz \right] \langle \bE_{s,\bh},\Gamma^{-1} \bx \rangle  \langle \bE_{t,\bh},\Gamma^{-1} \bx \rangle dP_{(\bX,\bZ)}(\bx,\bz), \\
		C_5(s,t)  &= \int \Var\left[Y \vert \bX=\bx,\bZ = \bz \right] \langle \bE_{s,\bh},\Gamma^{-1} \bx \rangle \left(\bE_{t,\bh}^{\bZ}+\bE_{t,\bh}^{\bX,\bZ}\right) \bB^{-1} \etab_1 dP_{(\bX,\bZ)}(\bx,\bz),\\
		C_6(s,t)  &= \int \Var\left[Y \vert \bX=\bx,\bZ = \bz \right] \left(\bE_{s,\bh}^{\bZ}+\bE_{s,\bh}^{\bX,\bZ}\right) \bB^{-1} \left(\bE_{t,\bh}^{\bZ}+\bE_{t,\bh}^{\bX,\bZ}\right)^{\top} dP_{(\bX,\bZ)}(\bx,\bz).
	\end{align*}		
	\end{theorem}

        \begin{remark}
	    Both conditions (i) and (ii) follow from \cite{cuesta2019goodness}.
	    Condition (i) entails that the dominating term of $T_{n,\bh}$ in (\ref{decom_T}) in the Appendix is $T_{n,\bh}^{1}+T_{n,\bh}^{3}+T_{n,\bh}^{5}$ and $\left\Vert\Gamma^{-1/2}x\right\Vert < \infty$, thus ensuring that the covariance function $K(s,t)$ of $\cG$ is well-defined. 
	    Not as restrictive as it seems to be, condition (ii) can be easily achieved in various practical situations, such as when $\rho$ is a linear combination of a finite number of eigenfunctions of  $\Gamma$. Moreover, it is only required to obtain the tightness of $T_{n,\bh}$. 
	    \end{remark}

        %\begin{remark}
         %   Define $t_{n,\bx} = \sqrt {\sum_{j=1}^{k_n} \lambda_j^{-1} \left\langle \bx,\be_j\right\rangle^2}$, then situation (c) refers to (c) of Theorem 3.1 in \cite{cuesta2019goodness}, that is, $\lim_{n} t_{n,\E_{x,\bh}} < \infty$ with $\E_{x,\bh}$ introduced in Appendix. Situation (c) is required so that the 
           % dominant term of $T_{n,\bh}$ in (\ref{decom_T}) is $T_{n,\bh}^{1}+T_{n,\bh}^{3}+T_{n,\bh}^{5}$ and $\left\Vert\Gamma^{-1/2}x\right\Vert < \infty$, guaranteeing the covariance function $K(s,t)$ is well defined. 
       % \end{remark}
	
	\begin{remark}
	    It is worth noting that the Gaussian process $\cG$ defined in Theorem \ref{theorem1} reduces to $\cG_2$ in Theorem 3.3 in \cite{cuesta2019goodness} if one is only interested in testing the linearity of the functional component in the degenerate SFPLR model when $\bbeta=\boldsymbol{0}$ is imposed. Our theoretical findings thus include checking the adequacy of the classical FLM as an important special case. %are thus consistent with those obtained under the classical FLM with FLM as the null hypothesis. %Therefore, the above theorem and corollary can be regarded as an extension of the theory in \cite{cuesta2019goodness}. 
	    \end{remark}

	    \begin{remark}
	    In general, the dependence between the real covariate $\bZ$ and the functional covariate $\bX$ cannot be simply ignored. In particular, note that $T_{n,\bh}^{5}$ in (\ref{decom_T}) represents the estimation uncertainty caused by using $\tilde\bbeta$, the presence of which affects the limiting distribution of the process $T_{n,\bh}$. In a very special case, if we assume the mixed-type covariates are mutually independent and the response and predictors are all centered, then  $T_{n,\bh}^{5}$ in (\ref{decom_T}) would disappear. Together with Lemma \ref{lemma2} and Lemma \ref{lemma8}, we can see that for this special case the estimation of the nuisance parameter $\bbeta$ plays no role asymptotically in testing linearity under \eqref{SFPLR}. However, since the relation between $\bZ$ and $\bX$ is usually unknown and non-independent in most practical cases, our paper does complement the existing literature on testing FLM by using a mixed and thus robust type residual to construct our linearity tests in the SFPLR model.
	\end{remark}

Theorem \ref{theorem1} and the continuous mapping theorem then yield the asymptotic null distributions of the continuous functionals of $T_{n,\bh}$, including the test statistics $\left\Vert T_{n,\bh} \right\Vert_{KS}$ and $\left\Vert T_{n,\bh} \right\Vert_{CvM}$ based on the KS and CvM norms, respectively.
	
	\begin{corollary} \label{Corollary1}
		Under $H_{0,\bh}$, together with Assumptions and conditions in Theorem \ref{theorem1}, if $\Vert T_{n,\bh} \Vert_{KS} := \sup_{x \in \mR} \vert T_{n,\bh}(x) \vert$ and $\Vert T_{n,\bh} \Vert_{CvM} := \int_{\mR} T_{n,\bh}(x)^2 dF_{n,\bh}(x)$, then
		\begin{equation*}
			\left\Vert T_{n,\bh} \right\Vert_{KS} \dto \Vert \cG \Vert_{KS} \quad  \text{and} \quad \Vert T_{n,\bh} \Vert_{CvM} \dto \int_{\mR} \cG(x)^2 dF_{\bh}(x),
		\end{equation*}
		where $\cG$ is the same Gaussian process as defined in Theorem \ref{theorem1}, and $F_{\bh}(\cdot)$ is the distribution function of $\bX^{\bh}$. %and $F_{n,\bh}(\cdot)$ is its empirical distribution function. 
	\end{corollary}

We reject the null hypothesis whenever $\left\Vert T_{n,\bh} \right\Vert_{KS}$ and $\left\Vert T_{n,\bh} \right\Vert_{CvM}$ exceed some overly ``large'' values. However, the asymptotic null distributions $\Vert \cG \Vert_{KS}$ and $\int_{\mR} \cG(x)^2 dF_{\bh}(x)$ in Corollary \ref{Corollary1} depend on the underlying data-generating process in a highly complicated manner, making a direct application of them infeasible. To implement our KS and CvM tests in practice, in Section \ref{practical_aspects}, we suggest an easy-to-implement wild bootstrap procedure to approximate the critical values of $\Vert \cG \Vert_{KS}$ and $\int_{\mR} \cG(x)^2 dF_{\bh}(x)$.	
	
\subsection{Asymptotic power}	
	Now, we investigate the asymptotic power properties of the KS and CvM tests. We first consider the fixed alternative of the following form:
\begin{equation}
 H_1:~ \E\left[\bY - \bX^{\brho} - \bZ^{\top}\bbeta \vert \bX  \right] \neq 0 ~ a.s.,  \quad \text{for all}~ \brho \in \cH,   	\end{equation}
 with its corresponding projected version
\begin{equation}
\label{H1_h}
 H_1^{\bh}:~ \E\left[\bY - \bX^{\brho} - \bZ^{\top}\bbeta \vert \bX^{\bh}  \right] \neq 0 ~ a.s.,  \quad \text{for all}~ \brho \in \cH.  
 \end{equation} 
Note that $H_1$ and $H_1^{\bh}$ are simply the negations of $H_0$ and $H_0^{\bh}$, respectively. The following theorem analyzes the asymptotic property of $T_{n,\bh}$ under $H_1^{\bh}$.  
	
	\begin{theorem} 
		\label{theorem2}
		Under $H_1^{\bh}$, together with Assumptions and conditions in Theorem \ref{theorem1}, it follows that $n^{-1/2}T_{n,\bh}(x) \pto \E \left[\left(m(\bX) - \bX^{\brho^{*}} \right) \indx\right]$ uniformly in $x\in\mathbb R$ and for some $\brho^{*}$ satisfying $\E \left[\left\Vert\hat{\brho} - \brho^{*}\right\Vert\right]= o(1)$.
	\end{theorem}

Note that $\brho^{*}$ in Theorem \ref{theorem2} should be understood as the probability limit of the two-step estimator $\hat{\brho}$ in $\cH$ under $H_1^{\bh}$, i.e., a pseudo true value. As an immediate consequence of Theorem \ref{theorem2}, under $H_1^{\bh}$ such that there exists some $x$ with a positive measure such that $\E \left[\left(m(\bX) - \bX^{\brho^{*}} \right) \indx\right]\neq 0$, the KS statistic $\left\Vert T_{n,\bh} \right\Vert_{KS}$ diverges to positive infinity at the rate of $n^{1/2}$, and the CvM statistic $\left\Vert T_{n,\bh} \right\Vert_{CvM}$ diverges to positive infinity at the rate of $n$. This then indicates that our proposed tests are consistent against the fixed projected alternative $H_1^{\bh}$ and thus $H_1$. %with probability approaching one. 
Perhaps more interestingly, because of the robust type estimator $\tilde\bbeta$ we used in the mixed residual $\hat U_i$, the consistency of our tests based on $T_{n,\bh}$ always holds regardless of any (unknown) dependence between the real covariates $\bZ$ and the functional covariate $\bX$.

Next, we study the power performance of our tests under a sequence of local alternative hypotheses converging to the null at the parametric rate $n^{-1/2}$ given by:
\begin{equation}
 H_{1n}: ~\E\left[\bY - \bX^{\brho} - \bZ^{\top}\bbeta \vert \bX\right] = n^{-1/2}r\left(\bX\right),\quad  \text{for some}~ \brho \in \cH, \label{H1n}   
\end{equation}
 with its corresponding projected version
\begin{equation}
 H_{1n}^{\bh}: ~\E\left[\bY - \bX^{\brho} - \bZ^{\top}\bbeta-n^{-1/2}r\left(\bX\right) \vert \bX^{\bh}\right] = 0,\quad  \text{for some}~ \brho \in \cH, \label{H1n}   
\end{equation} 
where $r(\cdot)$ is a non-zero function satisfying $\E \left\vert r(\bX)\right\vert < \infty$. Note that the function $r(\cdot)$ represents directions of departure from $H_{0}^{\bh}$, and $n^{-1/2}$ specifies the rate of convergence of $ H_{1n}^{\bh}$ to $H_{0}^{\bh}$, the fastest possible rate known in goodness-of-fit testing problems. The following theorem states the asymptotic distribution of $T_{n,\bh}$ under the sequence of local alternatives $H_{1n}^{\bh}$.
	
	\begin{theorem} 
		\label{theorem3}
		Under $H_{1n}^{\bh}$, together with Assumptions and conditions in Theorem \ref{theorem1}, it follows that $T_{n,\bh} \dto \cG + \Delta$, where $\cG$ is the same Gaussian process as defined in Theorem \ref{theorem1}, %except that $\brho$ in $\mathcal G$ is replaced by $\brho_0$, 
  and $\Delta$ is a deterministic shift function given by $\Delta(x) = \E \left[r\left(\bX\right) \indx\right]$ .
	\end{theorem}

Theorem \ref{theorem3} and the continuous mapping theorem yield that, under $H_{1n}^{\bh}$,  
\begin{equation*}
			\left\Vert T_{n,\bh} \right\Vert_{KS} \dto \Vert \cG+\Delta \Vert_{KS} \quad  \text{and} \quad \Vert T_{n,\bh} \Vert_{CvM} \dto \int_{\mR} \left(\cG(x)+\Delta(x)\right)^2 dF_{\bh}(x).
		\end{equation*}
 Consequently, our KS and CvM tests have non-negligible asymptotic powers against the sequence of local alternatives $H_{1n}^{\bh}$ because the deterministic function $\Delta(x)\neq 0$ for at least some $x$ with a positive measure. 
	
	%\begin{remark}
	    %The $\brho_0$ in Theorem \ref{theorem3} is specified as the true value of coefficients of the linear functional terms, while $\brho^{*}$ in Theorem \ref{theorem2} is a pseudo true value consistent with the two-step estimator $\hat{\brho}$ in $\cH$.
	%\end{remark}

	\section{Practical Aspects}\label{practical_aspects}
	\subsection{Parameter  estimation}
	Estimating the unknown linear coefficients $\bbeta$ mainly involves the choice of the semi-metric $d(\cdot,\cdot)$, the  kernel function $K(\cdot)$, and the bandwidth $b$. 
	One has to select among different kinds of semi-metrics which can be drawn by the shape of  trajectories of the functional
	covariate $\bX$. Several approaches, such as
	Functional PCA semi-metric, partial least-squares (PLS) semi-metric, or derivatives semi-metric are recommended for selecting the semi-metric in  \citet{ferraty2006nonparametric}. 
	Throughout our paper, we use the following widely adapted semi-metric:
	\begin{equation*}
		d\left(\mathcal{X}_i,\mathcal{X}_j\right)=\left\{\int\left(\mathcal{X}_i(t)-\mathcal{X}_j(t)\right)^2dt\right\}^{1/2},  \quad 
		\mathcal{X}_i,\mathcal{X}_j \in \mathcal{H}.
	\end{equation*}
	
	To derive $\tilde \bbeta$ in (\ref{beta_estimate}),  the quartic kernel $K( u) =15( 1-u^{2})
	^{2}\mathbb{I}_{\left\{ \left\vert u\right\vert \leq 1\right\}} /16$ is chosen, and the bandwidth has the form $b = cn^{-1/5}$, satisfying Assumption (B4), where $c$ is an adjustment parameter. %tuning constant.  
 We have found in extensive simulations that our tests are relatively insensitive to the choice of $c$. In particular, for the simulations considered, $c=3$ works quite well and is what we recommended, see Section \ref{band}.
	
	On the other hand, the estimation of the functional coefficient $\brho$ depends on  the truncated number $k_n$. 
	However, it is hard to implement (\ref{kn}) directly  since $\{\lambda_j\}$ is usually unknown in practice.
	As \citet{cuesta2019goodness} suggested, a data-driven way by selecting from a set of candidate $k_n$'s and choosing the optimal one in terms of a model-selection criterion,
	such as the Schwartz Information Criterion (SIC), is preferred. To be more specific, denote
	\begin{equation}\label{kn}\mathrm{SIC}\left(k_n\right)=\ell\left(\hat\brho_{k_n}\right)+\frac{\log (n) k_n}{n-k_n-2},
	\end{equation}
	where $\ell(\hat\brho_{k_n})$ represents the log-likelihood of the FLM for $\brho$	estimated with $k_n$ FPC's. The second term 
	$\log (n) k_n/(n-k_n-2)$ over-penalizes large $k_n$, which leads to noisy estimates of  $\brho$.
	Then one can obtain the regularized estimator $\hat\brho$ via (\ref{rhohat}).

	\subsection{Selection of projection directions}
        Theoretically speaking, the selection law of projection $\bh$ can be arbitrary as long as it has a non-degenerate distribution in $\cH$; that is, the range of the distribution of the random projection has a positive $\mu$ measure. And one only needs to randomly choose one direction. However, in practice, we pay special attention to the selection of projection direction $\bh$ in order to guarantee the robustness of the finite-sample performance.
 
	The projection direction $\bh$ plays an important role in testing $H_0$ since it may influence the power of the test significantly.
	For a special direction  that  is orthogonal to the data, namely $\bX^\bh=0$, it will 
	fail to calibrate the level of the test.
	Since under this projection direction, $T_{n,\bh} (x)= n^{-1/2} \sum_{i=1}^{n} \mathds{1}_{\left\{0 \le x\right\}}\hat U_i$ and $\left\Vert T_{n,\bh}\right\Vert_N=0$, then the  $p$-value of $H_0^{\bh}$  is always $1$, which obviously makes no sense.
	Moreover, different projection directions may yield  various outcomes of the test. 
	Both above issues can be tackled by a data-driven approach, which  avoids sampling  orthogonal directions and allows drawing several 
	directions $\bh_1,\ldots, \bh_K$. After obtaining  a number $K$ of different $p$-values,  the final $p$-value is determined by merging the resulting $p$-values with the False Discovery Rate (FDR) method in  \citet{benjamini2001control}. 
	The detailed procedure is shown in Table \ref{Algorithml}.

	\begin{table}[h!]
		\centering
		\caption{A data-driven approach of  projection direction selection.}
		\label{Algorithml}
		\begin{tabular}{l}
			\hline
			\textbf{Algorithm 1}: Construction of projection directions\\
			\hline
			\textbf{Input}: Functional covariates $\left\{\bX_{i}\right\}_{i=1}^n$\\
			\qquad Step 1: Compute the FPC of $\bX_1,\ldots,\bX_n$, namely the eigenpairs $\left\{\left(\hat\lambda_j,\hat{\be}_j\right)\right\}_{j=1}^n$.\\ 
			\qquad Step 2: Choose $j_n :=\min\left\{k=1,\ldots,n-1: \left(\sum_{j=1}^k 
			\hat\lambda_j^2\right)/\left(\sum_{j=1}^{n-1} \hat\lambda_j^2\right)\le 0.95\right\}$. \\
			\qquad Step 3: Generate the data-driven projection direction $\bh=\sum_{j=1}^{j_n}\eta_j\hat\be_j$, with $\eta_j\sim N(0,s_j^2)$,\\
			\quad\quad \quad\qquad with $s_j^2$ the sample variance of the scores in the $j$-th FPC. \\
			\qquad Step 4: Repeat Step 3 for $K$ times and then get projection directions $\bh_1,\ldots, \bh_K$.\\
			\textbf{Output}:  Projection directions $\bh_1,\ldots, \bh_K$.\\
			\hline
		\end{tabular}
		%\caption{Data-driven approach of  projection direction selection}
		%\label{Algorithml}
	\end{table}

	\subsection{Critical values}
	\label{boot}
	%We compute critical values via the wild bootstrap method to approximate the distribution of the statistics under the null   hypothesis. 
 As discussed before, the asymptotic null distributions of our test statistics depend on the underlying data-generating process and the corresponding critical values cannot be tabulated. To implement our tests, we use the wild bootstrap method to estimate the critical values in this section. 
	%which enjoys good theoretical and empirical properties is easy to implement and is free of computing new parameter estimates at each replication.
	The wild bootstrap procedure is widely adopted in specification testing literature.  It particularly works for situations with an unknown form of heteroskedasticity, which is commonly observed in the functional data setting. % is consistent in the finite-dimensional case and efficient  for situations with potential heteroskedasticity, quite popular in functional data. 

%$\hat{U}_i = Y_i - \bZ_i^{\top} \tilde{\bbeta}-\bX_{i}^{\hat{\brho}}$

	To be more precise, the wild bootstrap residual  $U_i^{\ast}$ for $1\le i \le n$ is defined as $U_i^{\ast}=V_i\hat U_i$, 
	where $\{V_i\}_{i=1}^n$ is a sequence of i.i.d. random variables with mean zero and variance one and also independent of the original sample $\{Y_i,\bZ_i,\bX_i\}_{i=1}^n$. One popular choice is $\mathbb P(V_i=1-\kappa)=\kappa/\sqrt 5$ and 	  $\mathbb P(V_i=\kappa)=1-\kappa/\sqrt 5$ with $\kappa=(\sqrt 5 +1)/2$, as suggested by \cite{hardle1993comparing}.
 Then, the asymptotic null behavior of $T_{n,\bh}(x)$ can be consistently approximated by the following bootstrap process: %test statistic 
	\begin{align*}
		T_{n,\bh}^{\ast} (x)= n^{-1/2}\sum_{i=1}^n \left(Y_i^{\ast}-\bZ_i^{\top}\tilde {\bbeta}^{\ast}-\bX_{i}^{\hat\brho^{\ast}}
		\right) \indi,
	\end{align*}
	where  $Y_i^{\ast} =\bZ_i^{\top}\tilde{\bbeta}+\bX_{i}^{\hat{\brho}}+U_i^{\ast}$ for $1\le i \le n$ is the bootstrap dependent variable, and $\tilde\bbeta^{\ast}$ and $\hat {\brho}^{\ast}$ are the estimators from the bootstrap sample $\{Y_i^{\ast},\bZ_i,\bX_i\}_{i=1}^n$.
   After repeating  the above scheme for $B$ times, the $p$-value is computed by $1/B\sum_{b=1}^B \mathds{1}_{\left\{\left\Vert T_{n,\bh}\right\Vert_N \le \left\Vert T_{n,\bh}^{\ast b}\right\Vert _N   \right\}}$, with $N$ being either KS or CvM. The next theorem ensures the asymptotic validity of the wild  bootstrap procedure.
   
	\begin{theorem} \label{theorem4}
		Under $H_0^{\bh}$, together with Assumptions and conditions in Theorem \ref{theorem1}, for the wild bootstrap process $T_{n,\bh}^{*}$,  there exists $T_{n,\bh}^{*} \dto \cG^{*}$ in $D(\mR)$, where $\cG^{*}$ and $\cG$ have the same distribution.
	\end{theorem}
	
%	\begin{theorem} \label{theorem5}
%		Under Assumptions (A1) to (A6) and  (B1) to (B8), for the wild bootstrap process $T_{n,\bh}^{**}$, we have $T_{n,\bh}^{**} \dto \cG^{**}$ in $D(\mR)$, where $\cG^{**}$ and $\cG$ have the same distribution.
%	\end{theorem}

To clarify the whole testing procedure, we give the detailed computational algorithm  in the following table.%Table \ref{Algorithm2}.
	\begin{table}[h!]
		\centering
		%\caption{The detailed algorithm for testing $H_0$.}
		%\label{Algorithm2}
\begin{tabular}{l}
			\hline
			\textbf{Algorithm 2}: Testing procedure for $H_0$ \\
			\hline
			\textbf{Input}: Data $\left\{Y_i, \bZ_i, \bX_{i}\right\}_{i=1}^n$\\
			\qquad Step 1. Estimate $\bbeta$ via (\ref{beta_estimate})   with the quadratic kernel and the bandwidth $h=cn^{-1/5}$.\\
			%and \\\qquad \qquad\quad denote  $\hat{D}_i=Y_i -  \bZ_i^{\top}\tilde{\bbeta}$. \\  
			\qquad Step 2. Estimate $\brho$ based on (\ref{rhohat})  for a  $k_n$ chosen by (\ref{kn}) and obtain $\hat U_i=Y_i -  \bZ_i^{\top}\tilde{\bbeta}-\bX_{i}^{\hat{\brho}}$.  \\
			\qquad   Step 3. Sample several projection directions $\bh_1,\ldots, \bh_K$.\\ 
			\qquad  Step 4. For a given $\bh_i$, $i=1,\ldots, K$,\\
			\qquad \qquad  (i) Compute $\left\Vert T_{n,\bh}\right\Vert_N=\left\Vert n^{-1/2}\sum_{i=1}^n \mathds{1}_{\left\{\bX_i^{\bh}\le x\right\}}\hat U_i\right\Vert_N$  with $N$ being either KS or CvM. \\
			\qquad \qquad (ii) Wild bootstrap resampling. For $b=1,\ldots, B$:\\
			
			\qquad \qquad\qquad\quad (a) Draw $\{V_i\}_{i=1}^n$, a sequence of i.i.d. random variables such that $\mathbb P(V_i=1-\kappa)$\\
			\quad\qquad \qquad\qquad\quad$=\kappa/\sqrt 5$
			and $\mathbb P(V_i=\kappa)=1-\kappa/\sqrt 5$, where $\kappa=(\sqrt 5 +1)/2$. \\
			\qquad \qquad\qquad\quad (b) Set $Y_i^{\ast} :=\bZ_i^{\top}\tilde{\bbeta}+\bX_{i}^{\hat{\brho}}+U_i^{\ast}$
			with the bootstrap residual $U_i^{\ast}=V_i\hat U_i$.\\
			\qquad \qquad\qquad\quad (c) Estimate ${\bbeta}^{\ast}$ and $\brho^{\ast}$ from $\{Y_i^{\ast},\bZ_i,\bX_i\}_{i=1}^n$ via Step 1 and Step 2.\\
			\qquad \qquad\qquad\quad (d) Obtain the estimated bootstrap residual 
			$\hat U_i^{\ast}=Y_i^{\ast} -\bZ_i^{\top}{\tilde{\bbeta}^{\ast}}-\bX_{i}^{\hat\brho^{\ast}}$.\\
			\qquad \qquad\qquad\quad (e) Compute $\left\Vert T_{n,\bh}^{\ast b}\right\Vert_N=\left\Vert n^{-1/2}\sum_{i=1}^n \mathds{1}_{\left\{\bX_i^{\bh}\le x\right\}}\hat U_i^{\ast}\right\Vert_N$.\\
			\qquad \qquad (iii) Approximate the $p$-value of $H_0^{\bh_i}$ by 
			$p_i=B^{-1}\sum_{b=1}^B\mathds{1}_{\left\{\left\Vert T_{n,\bh}\right\Vert_N\le\left\Vert T_{n,\bh}^{\ast b}\right\Vert_N\right\}}$. \\
			\qquad Step 5: Set the final $p$-value of $H_0$ as $\min_{i=1,\ldots,K}\frac{K}{i}p_{(i)}$, where $p_{(1)}\le\ldots\le p_{(K)}$.\\
			\textbf{Output}:  The  $p$-value of $H_0$.\\
			\hline
		\end{tabular}
	\end{table}

\section{Simulation study and data application}
	\label{sim}
	In this section, we carry out various simulations to illustrate the finite sample performance of our proposed projection-based tests.
	
	\subsection{Simulation scenarios}
	In order to investigate the potential  influence of  different  covariates $\bZ$  and $\bX$,    as well as the linear coefficient $\bbeta$
	and the functional coefficient $\brho$,  we consider 8 possible scenarios. Denote the $k$-th scenario as S$k$, with coefficients $\bbeta_k$
	and  $\brho_k$. The deviation from $H_0$ is measured by a coefficient $\delta_d$, with $\delta_0=0$ corresponding to the null hypothesis,
	and nonnegative $\delta_d$ for $d=1,2$, which yield various alternatives. For $H_{k,d}$, the  scenario S$k$  with coefficient $\delta_d$, the
	data is generated from
	\begin{equation}
	\label{datapro}
		Y=\bZ_k^{\top}\bbeta_k+\left\langle \bX, \brho_k\right\rangle+\delta_d \triangle_{{\btheta}_k}\left(\bX\right)+\varepsilon,
	\end{equation}
	where ${\btheta}=(1,2,2,2,1,2,3,3)^{\top}$. The non-linear terms include $\triangle_1\left(\bX\right)=\Vert \bX\Vert$,
	$\triangle_2\left(\bX\right)=25\int_{0}^1\int_{0}^1 \sin(2\pi ts)s(1-s)(1-t)\bX(s)\bX(t)dsdt$ and 
	$\triangle_3\left(\bX\right)=\left\langle e^{-\bX},\bX^2\right\rangle$.
	The scalar covariate $\bZ$ contains variables $Z_1\sim N(1,0.25)$, $Z_2\sim N(2, 1)$,    independent of $\bX$, 
	as well as variables $Z_3$, $Z_4$,  which have some certain correlation with $\bX$.
	The error $\varepsilon$ follows the normal distribution $N(0,0.01)$.
	The functional process $\bX(t)$, centered and valued in [0,1] with 200 discretized equally spaced points, is as follows:\\
	{ \bf{Brownian Motion}} (BM): denoted by $\mathbf{B}$, with eigenfunctions $\psi_j(t)=\sqrt 2 \sin \left\{\left(j-1/2\right)\pi t\right\}$,  $j\ge 1$.\\
	{ \bf{Functional Process}} (FP):  given by $\bX(t)=\sum_{j=1}^{20}\xi_j\phi_j(t)$, where $\phi_j(t)=\sqrt 2 \cos(j\pi t)$
	and $\xi_j$ are independent r.v's distributed as $N\left(0,j^{-2}\right)$.\\%, with $l=1,2$.\\
	{ \bf{Brownian Bridge}} (BB):  defined as $\bX(t)=\mathbf{B}(t)-t\mathbf{B}(1)$, whose eigenfunctions are ${\widetilde\psi}_j(t)=\psi_{j+1/2}(t)$.\\
	{ \bf{Ornstein-Uhlenbeck Process}} (OU):  the Gaussian process with mean function $\mathbb E X_t=x_0e^{-\theta t}
	+\mu\left(1-e^{-\theta t}\right)$ and covariance function 
	$\mathrm{Cov}\left(\bX(s),\bX(t)\right)=\frac{\sigma^2}{2\theta}\left(e^{-\theta\vert t-s\vert}-e^{-\theta(t+s)}\right)$, 
	where $\theta=1/3$, $\sigma=1$ and $x_0=0$.\\
	{ \bf{Geometric Brownian Motion}} (GBM):  defined as $\bX(t)=s_0\exp\left\{(\mu-\sigma^2/2)t+\sigma\mathbf{B}(t)\right\}$, with $\sigma=1$, $\mu=1/2$
	and $s_0=2$.
	
Table \ref{simulated scenarios} displays the simulation scenarios, including deviations from the null hypothesis.	
The functional coefficient in S1 and S2 is a finite linear combination of the eigenfunctions of the functional process, 
with S1 based on Brownian motion and S2 based on Brownian bridge.
 S3 to S5 consider a finite-dimensional smooth process, and $Z_3=10\zeta_3$, $Z_4=4\zeta_4$, where $\zeta_3$ and  $\zeta_4$ are third and fourth FPC scores of the functional covariate $\bX$, respectively. The  functional coefficient  is not expressible as a finite combination of eigenfunctions in the next three scenarios: S7 and S8 with the Ornstein--Uhlenbeck processes and S9 with the geometric Brownian motion. The criteria for choosing deviation coefficients $\delta_d$, $d = 1,2$,  in (\ref{datapro}) is to add difficulty in distinguishing between the null hypothesis 
and the alternative.  Densities of the response Y under different scenarios are shown in Figure \ref{density}, which  provides a graphical visualization
that all three densities (one under the null and two under the alternative) in every scenario are close, thus making  the distinction between the hypotheses tough.
Figure \ref{sce} exhibits functional curves of $\bX$ and the functional coefficients $\brho$, as well as its estimator $\hat\brho$ in each scenario. Since S4 and S5 share the same functional covariate and coefficient with S3, thus omitted to save space. Figure \ref{boxplot} exhibits the scaled deviation between the linear
coefficient $\bbeta$ and its estimator $\tilde\bbeta$ in  each scenario, defined as 
$\Vert\bbeta-\tilde\bbeta\Vert/\Vert\bbeta\Vert$.  
One can easily find that the  deviation is decreasing with the increase in sample sizes. 

	\begin{table}[h]
		\caption{Summary of the simulated scenarios.}
		\label{simulated scenarios}
		\vspace {-0.7cm}
		\begin{center}
			\resizebox{\textwidth}{28mm}{
				\begin{tabular}{ccccc}
					\hline
					Scenario & Linear parts $\bZ^{\top}\bbeta$  &  Coefficient $\brho(t)$ & Process $\bX$ & Deviation  \\
					\hline
					S1 &  $2Z_1+Z_2$ &$(2\psi_1(t)+4\psi_2(t)+5\psi_3(t))/\sqrt 2$ & BM &$\triangle_1$, $\delta=(0,2/5,4/5)^{\top}$\\
					S2 &  $3Z_1+2Z_2$ &$(2{\widetilde\psi}_1(t)+4{\widetilde\psi}_2(t)+5{\widetilde\psi}_3(t))/\sqrt 2$ & BB &$\triangle_2$, $\delta=(0,5/2,15/2)^{\top}$\\
					%S3 &  $Z_1$ &$(2\psi_2(t)+4\psi_4(t)+5\psi_7(t))/\sqrt 2$ & BM &$\triangle_1$, $\delta=(0,-1/5,-1/2)^{\top}$\\
					S3 &  $Z_1+2Z_2$ &$\sum_{j=1}^{20}2^{3/2}(-1)^jj^{-2}\phi_j(t)$ & FP &$\triangle_2$, $\delta=(0,-1,-2)^{\top}$\\
					S4 &  $Z_3+2Z_4$ &$\sum_{j=1}^{20}2^{3/2}(-1)^jj^{-2}\phi_j(t)$ & FP &$\triangle_2$, $\delta=(0,-1,-3)^{\top}$\\
					S5 &  $2Z_3$ &$\sum_{j=1}^{20}2^{3/2}(-1)^jj^{-2}\phi_j(t)$ & FP &$\triangle_1$, $\delta=(0,-1,-2)^{\top}$\\
					S6 &  $2Z_1+Z_2$ &$\sin(2\pi t)-\cos (2\pi t)$ & OU &$\triangle_2$, $\delta=(0,-1/4,-1)^{\top}$\\
					S7 &  $2Z_2$ &$t-\left(t-3/4\right)^2$ & OU&$\triangle_3$, $\delta=(0,-1/100,-1/2)^{\top}$\\
					S8 &  $2Z_1+Z_2$ &$\pi^2\left(t^2-1/3\right)$ & GBM &$\triangle_3$, $\delta=(0,5/2,9/2)^{\top}$\\
					\hline
				\end{tabular}
			}
		\end{center}
	\end{table}

\begin{figure}[h]
\begin{minipage}[t]{0.22\linewidth}
\centering
\includegraphics[width=1.65in,height=2in]{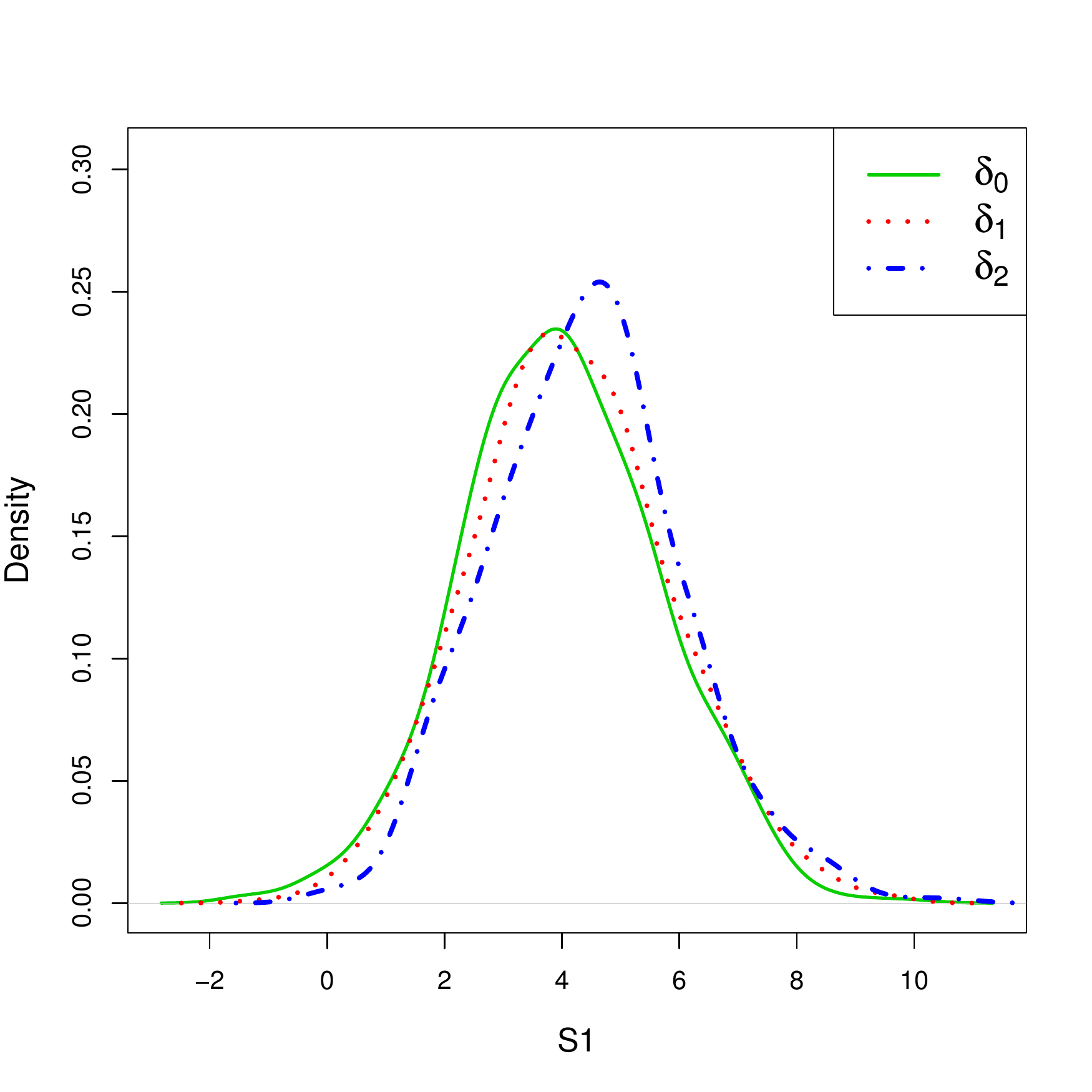}
\caption*{}
\end{minipage}
\hfill
\begin{minipage}[t]{0.22\linewidth}
\centering
\includegraphics[width=1.65in,height=2in]{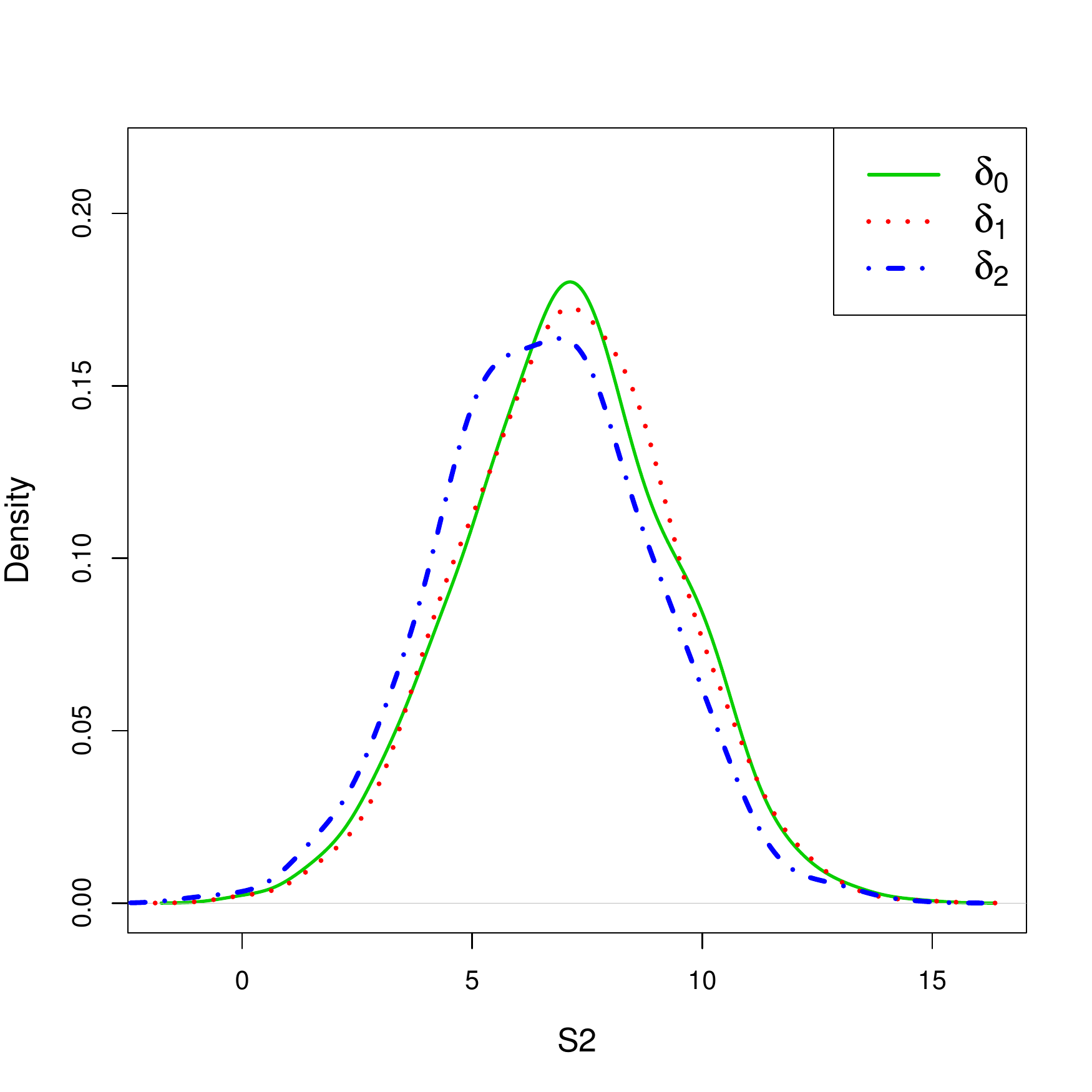}
\caption*{}
\end{minipage}
\hfill
\begin{minipage}[t]{0.22\linewidth}
\centering
\includegraphics[width=1.65in,height=2in]{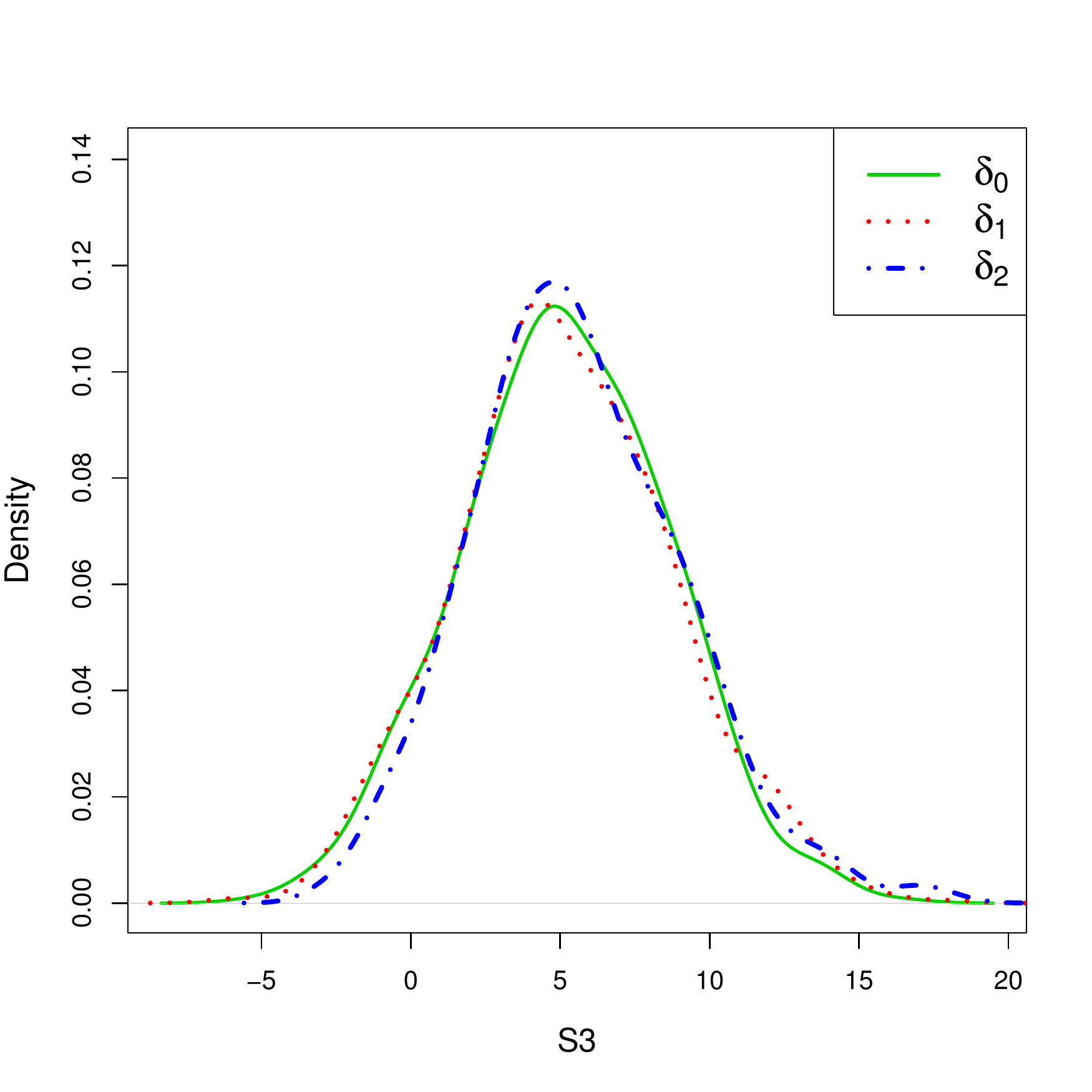}
\caption*{}
\end{minipage}
\hfill
\begin{minipage}[t]{0.22\linewidth}
\centering
\includegraphics[width=1.65in,height=2in]{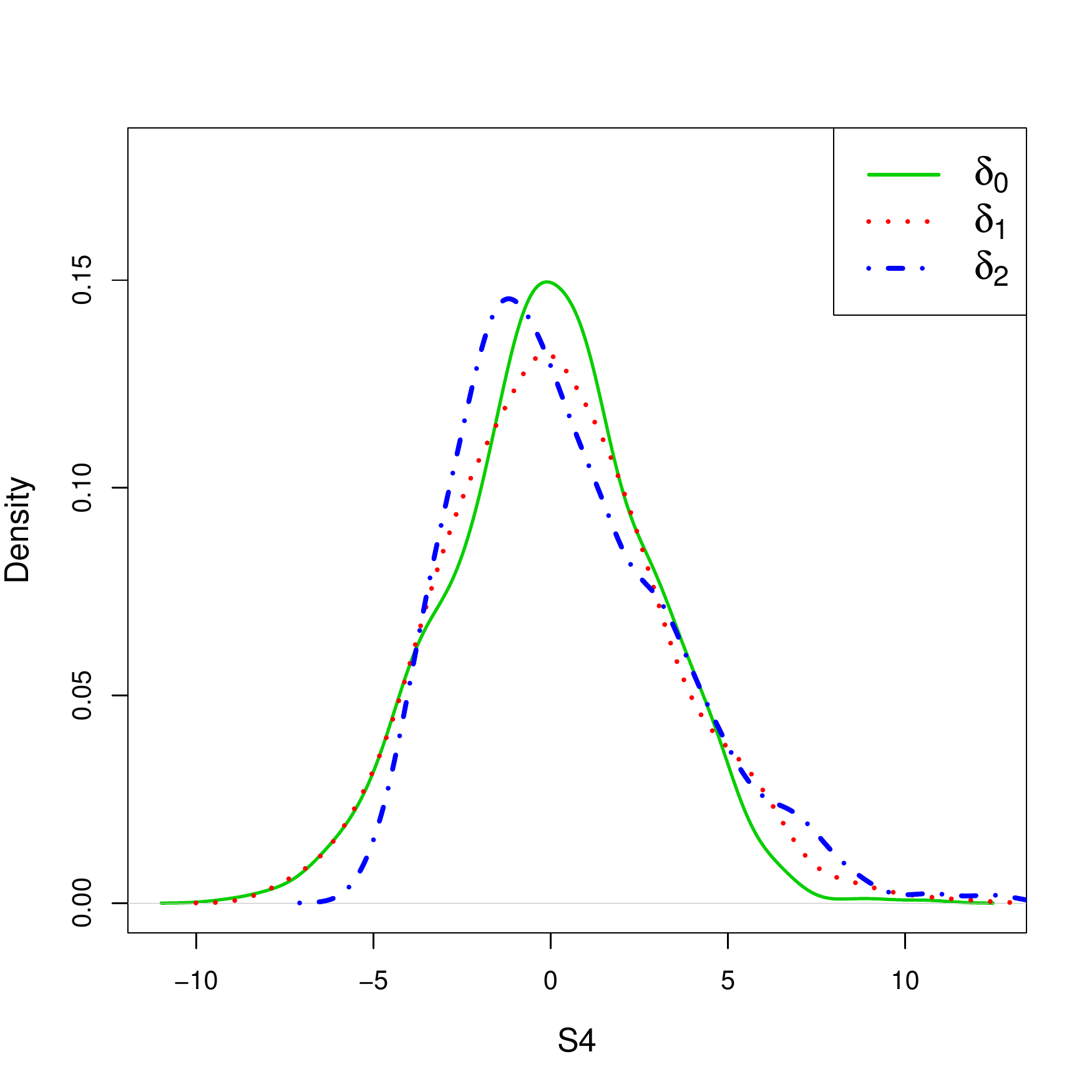}
\caption*{}
\end{minipage}
\begin{minipage}[t]{0.22\linewidth}
\centering
\includegraphics[width=1.65in,height=2in]{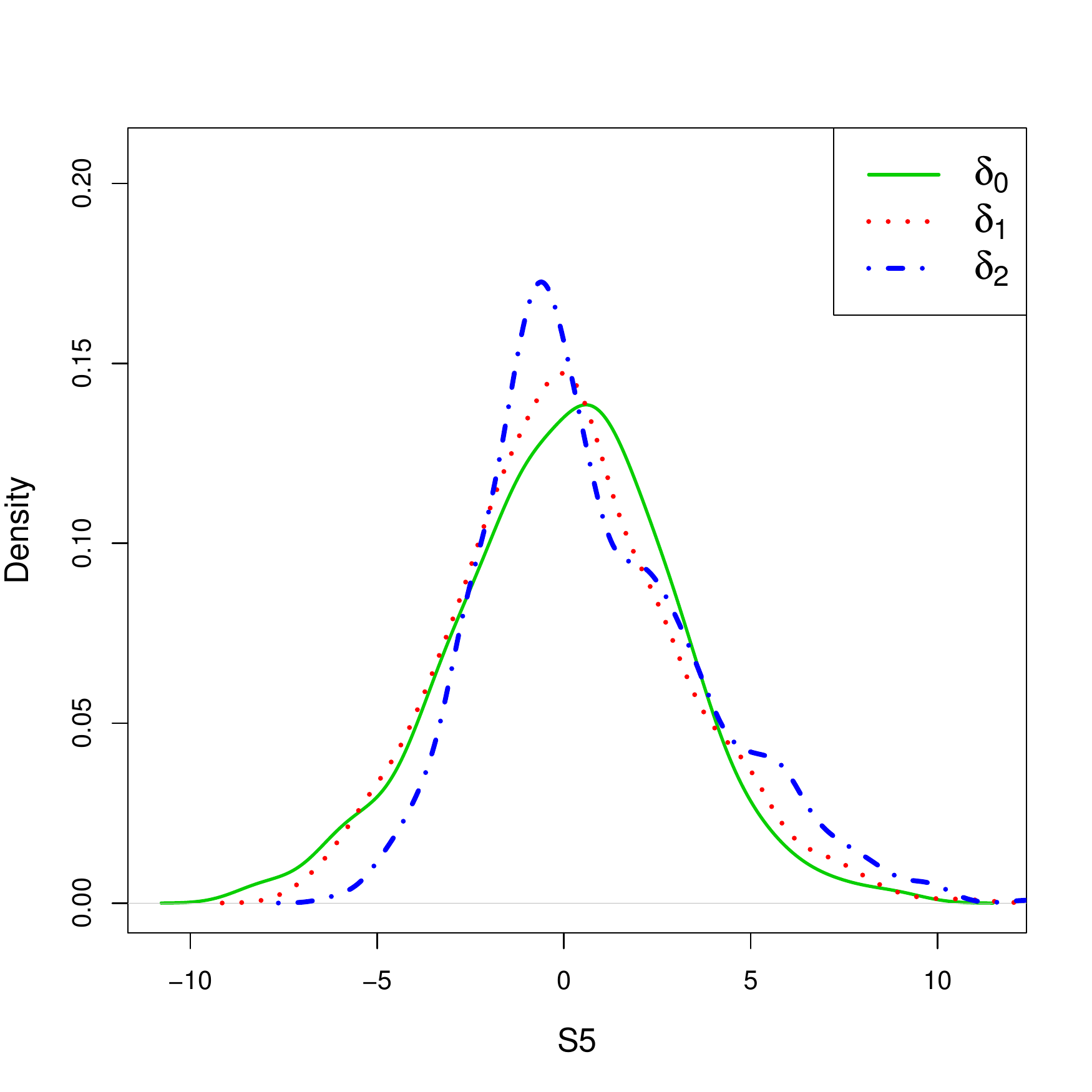}
\caption*{}
\end{minipage}
\hfill
\begin{minipage}[t]{0.22\linewidth}
\centering
\includegraphics[width=1.65in,height=2in]{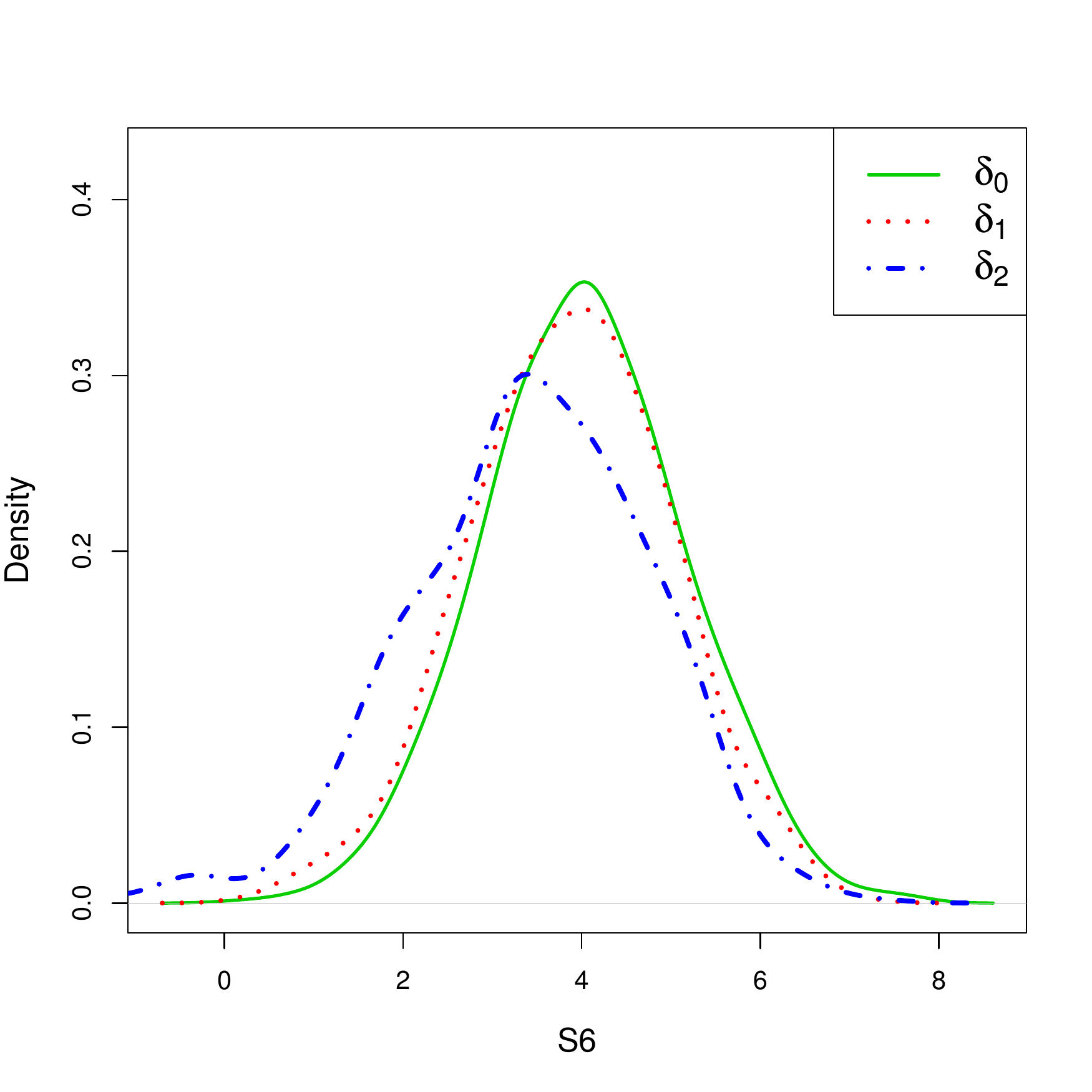}
\caption*{}
\end{minipage}
\hfill
\begin{minipage}[t]{0.22\linewidth}
\centering
\includegraphics[width=1.65in,height=2in]{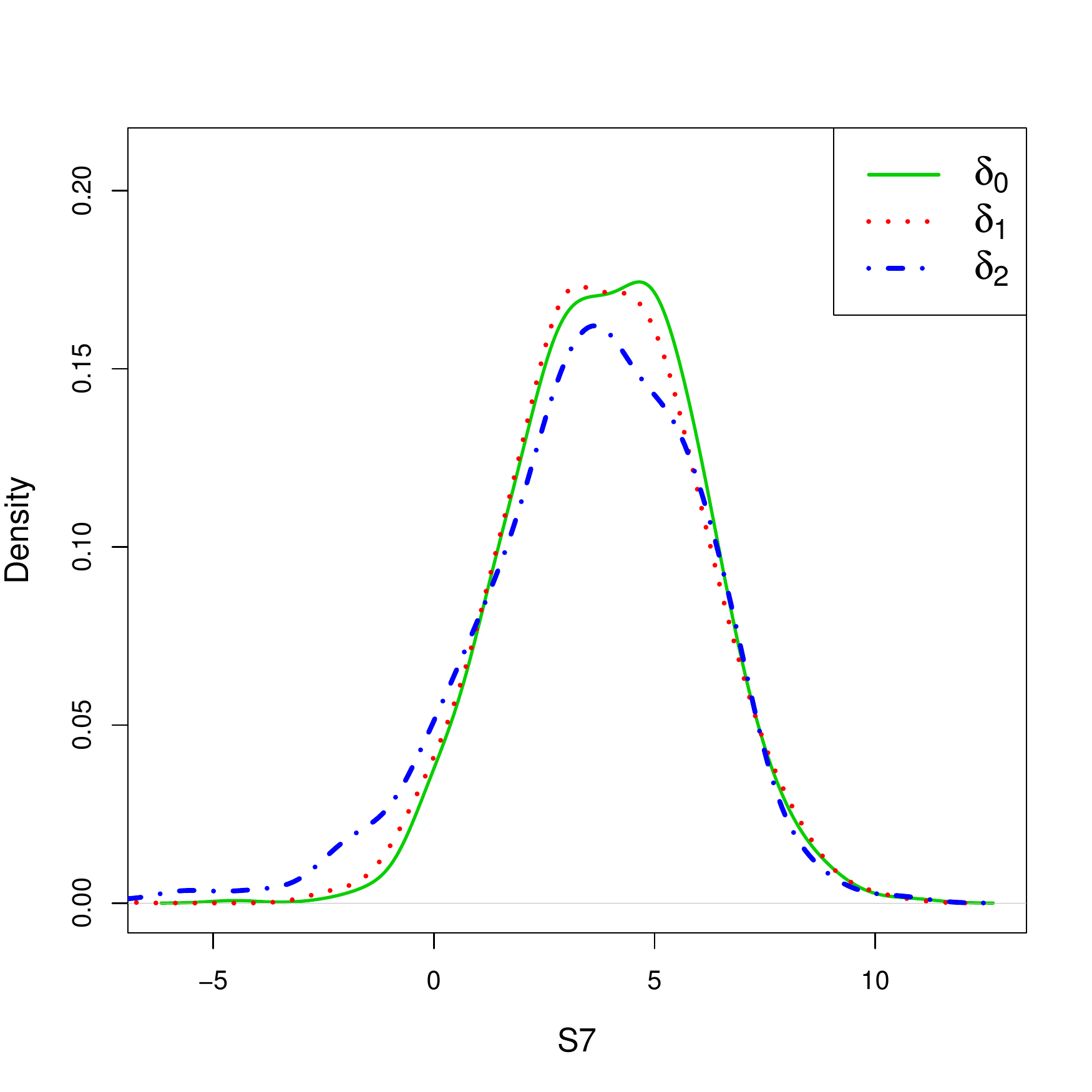}
\caption*{}
\end{minipage}
\hfill
\begin{minipage}[t]{0.22\linewidth}
\centering
\includegraphics[width=1.65in,height=2in]{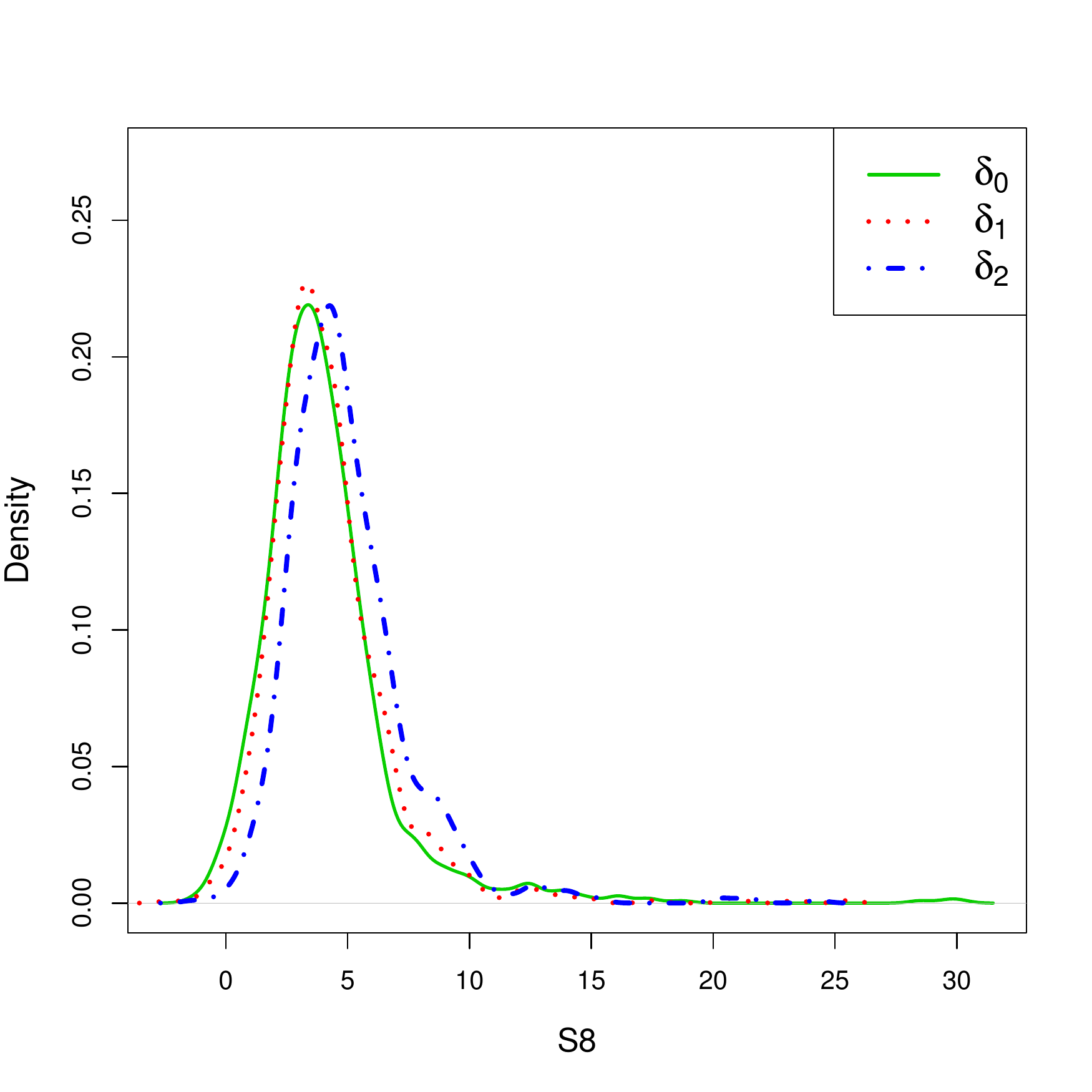}
\caption*{}
\end{minipage}
\caption{Density plots of the response $Y$ under the null 
hypothesis $H_{k,0}$ (solid) and two alternatives $H_{k,1}$ (dotted), $H_{k,2}$ (dashed) in scenario S$k$, $k=1,\cdots,8$. }
\label{density}
\end{figure}

\begin{figure}[h]
\begin{minipage}[t]{0.3\linewidth}
\centering
\includegraphics[width=2in,height=2in]{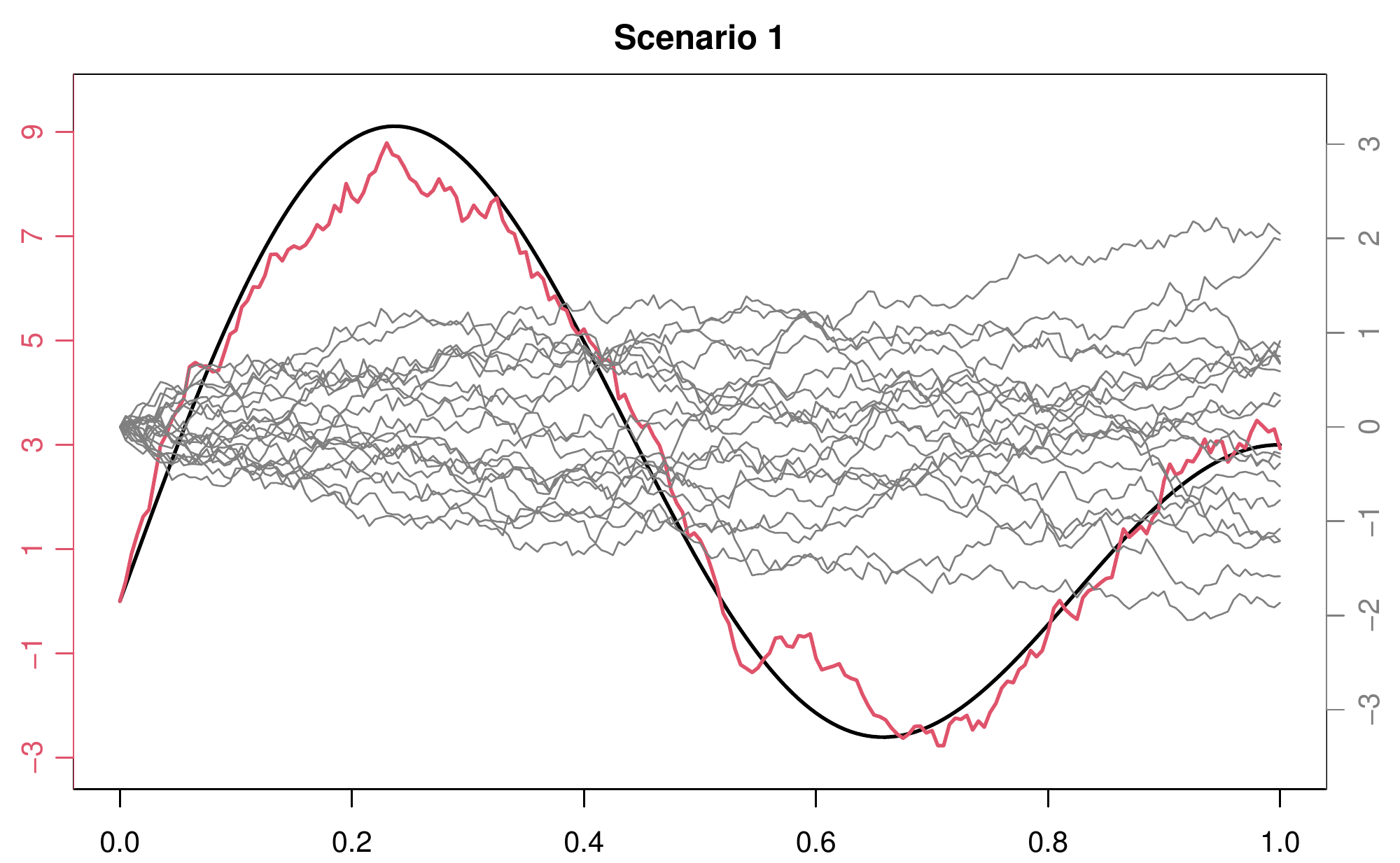}
\caption*{}
\end{minipage}
\hfill
\begin{minipage}[t]{0.3\linewidth}
\centering
\includegraphics[width=2in,height=2in]{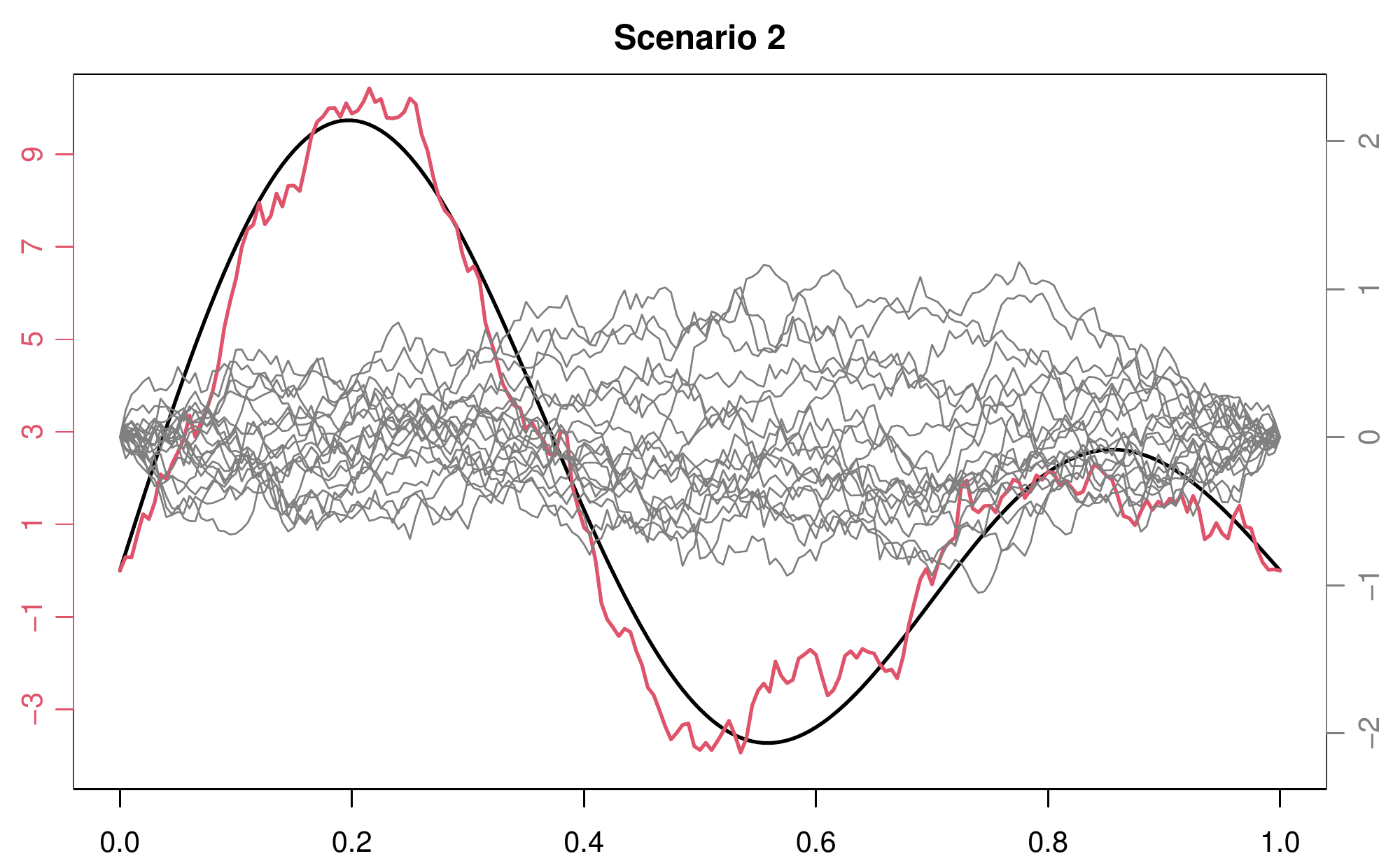}
\caption*{}
\end{minipage}
\hfill
\begin{minipage}[t]{0.3\linewidth}
\centering
\includegraphics[width=2in,height=2in]{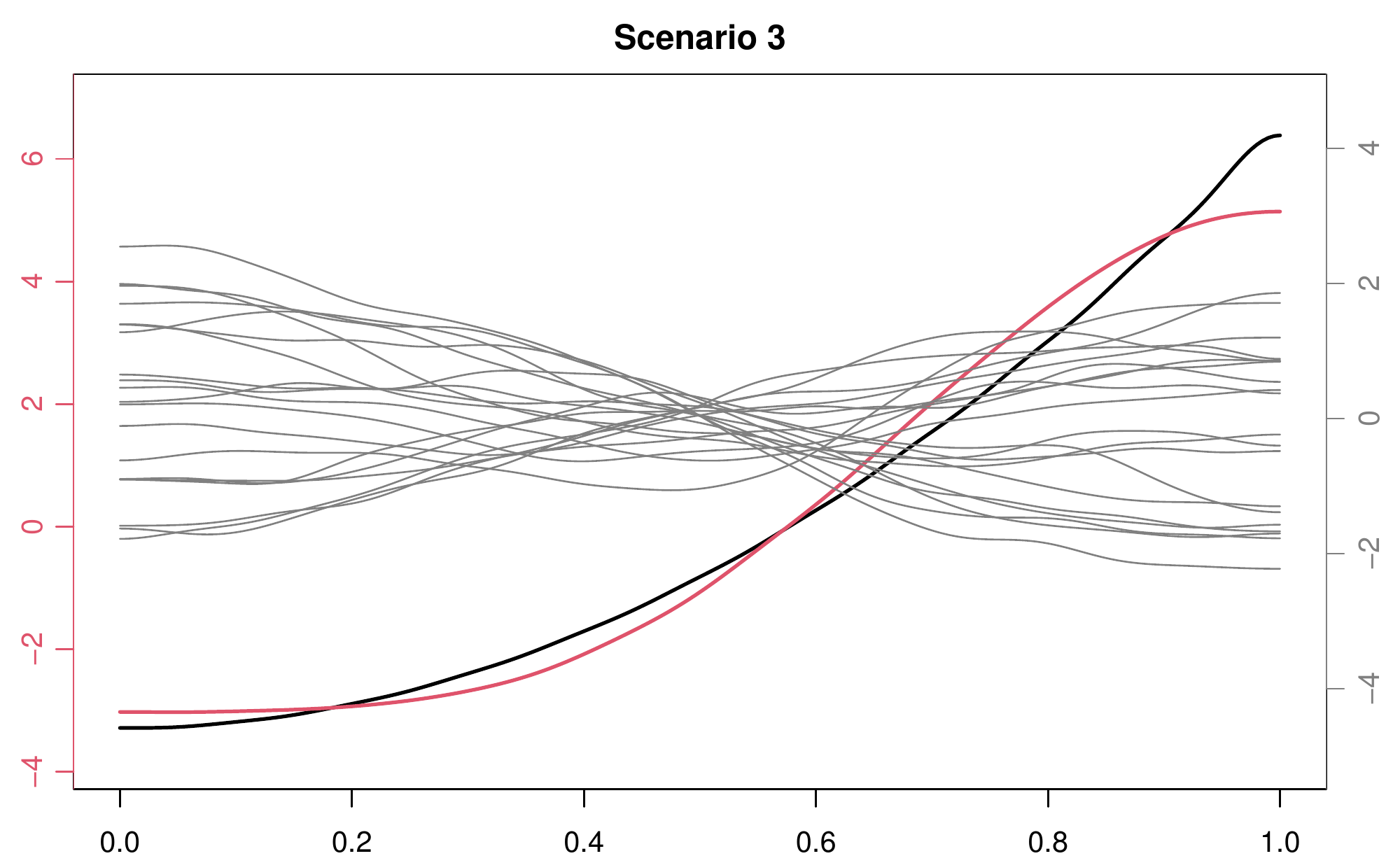}
\caption*{}
\end{minipage}
\hfill
\begin{minipage}[t]{0.3\linewidth}
\centering
\includegraphics[width=2in,height=2in]{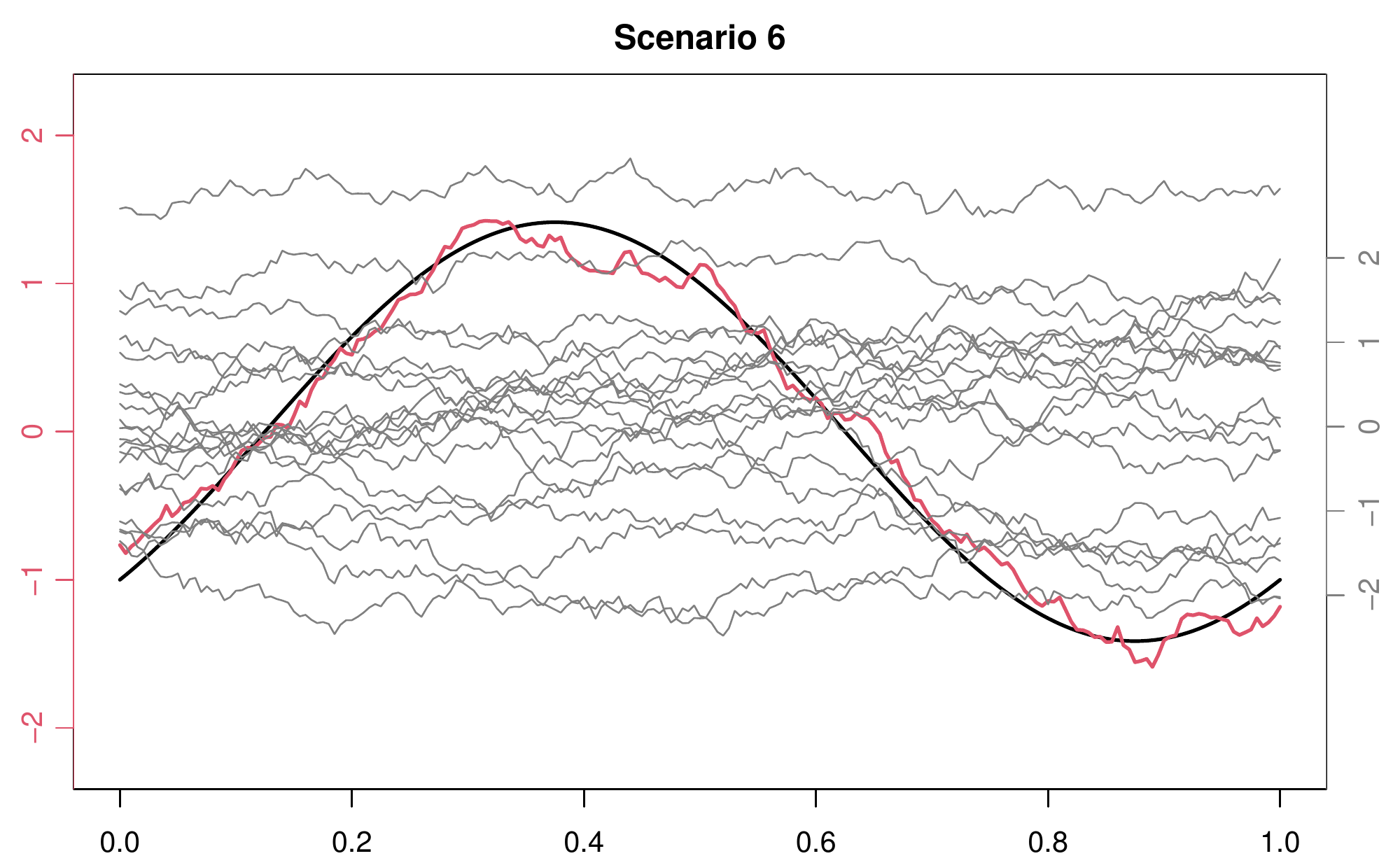}
\caption*{}
\end{minipage}
\hfill
\begin{minipage}[t]{0.3\linewidth}
\centering
\includegraphics[width=2in,height=2in]{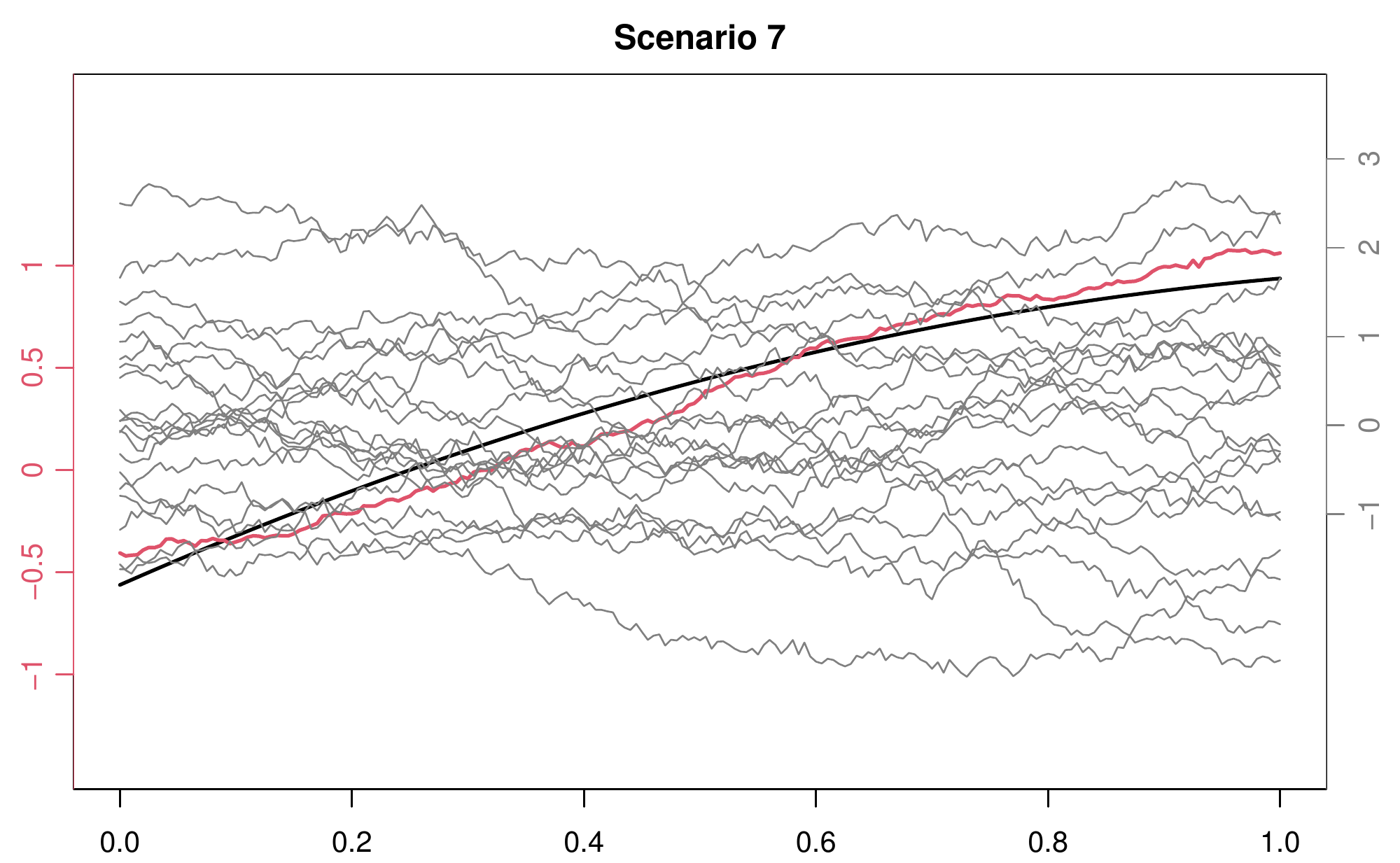}
\caption*{}
\end{minipage}
\hfill
\begin{minipage}[t]{0.3\linewidth}
\centering
\includegraphics[width=2in,height=2in]{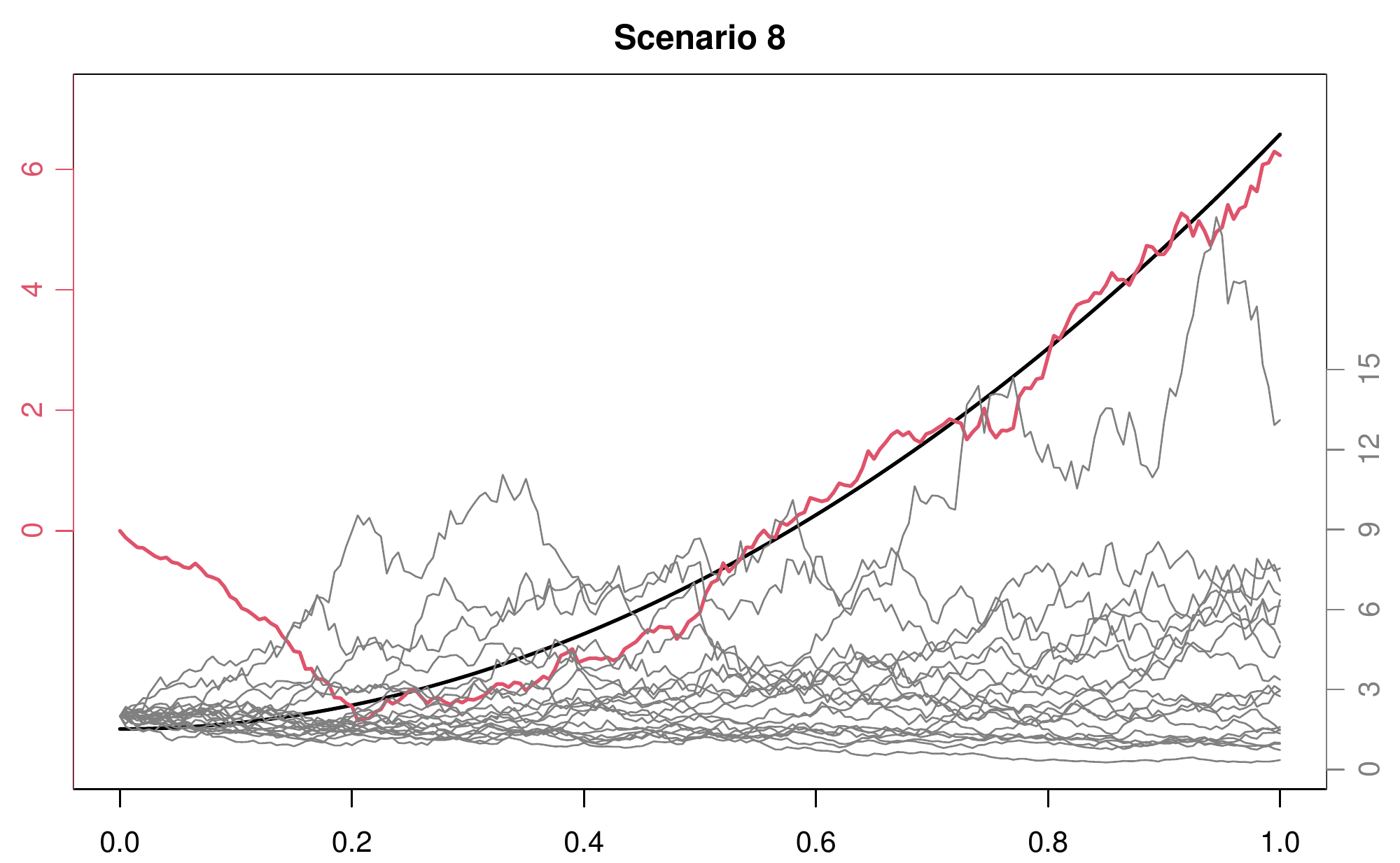}
\caption*{}
\end{minipage}
\caption{Plots of functional coefficient $\brho$ (black, left scale), its estimate $\hat\brho$ (red, left scale) and  20 realizations
of the functional covariate $\bX$ (grey, right scale) in scenario S$k$, $k=1,2,3,6,7,8$. }
\label{sce}
\end{figure}

\begin{figure}[h]
\begin{minipage}[t]{0.22\linewidth}
\centering
\includegraphics[width=1.5in,height=2in]{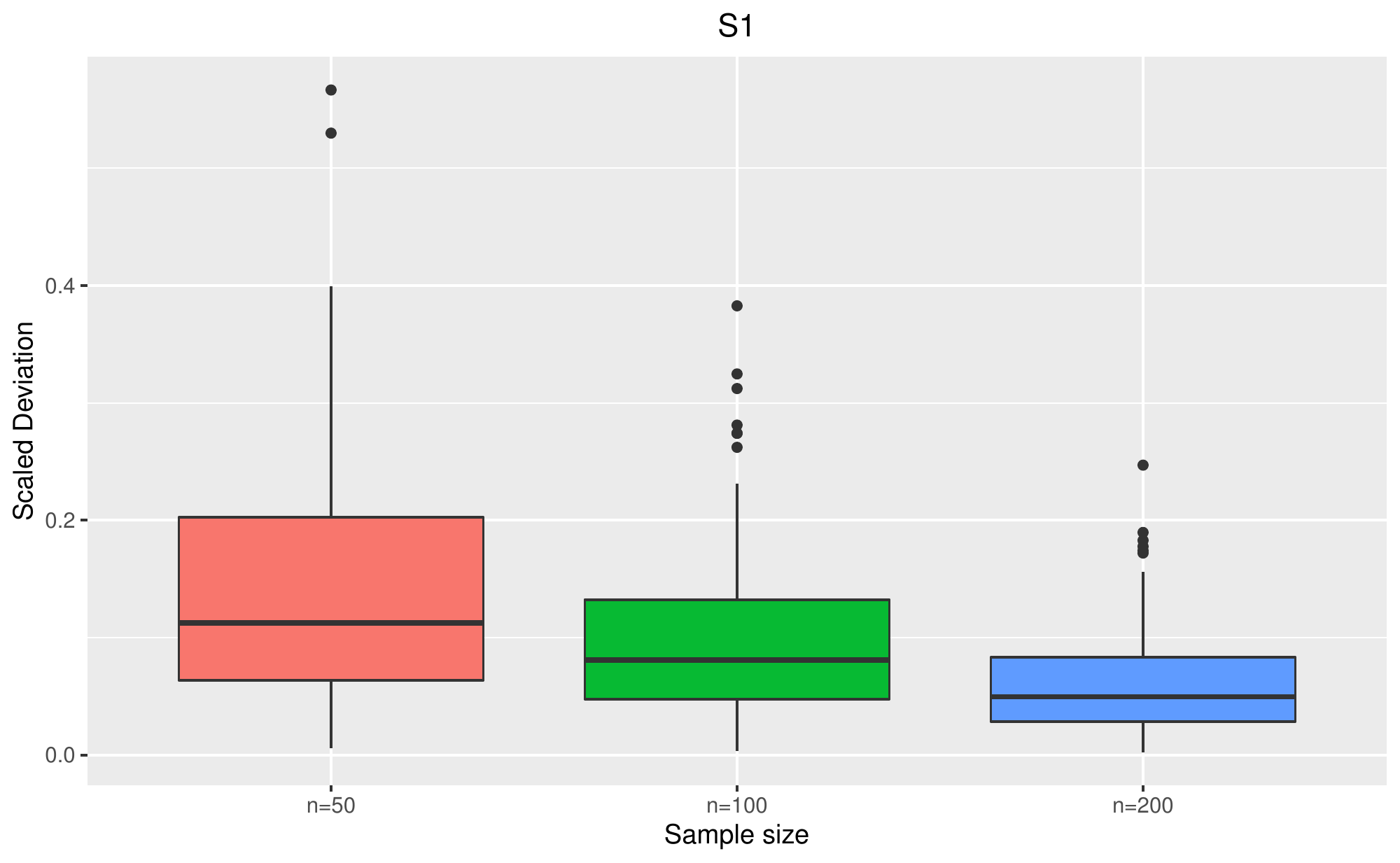}
\caption*{}
\end{minipage}
\hfill
\begin{minipage}[t]{0.22\linewidth}
\centering
\includegraphics[width=1.5in,height=2in]{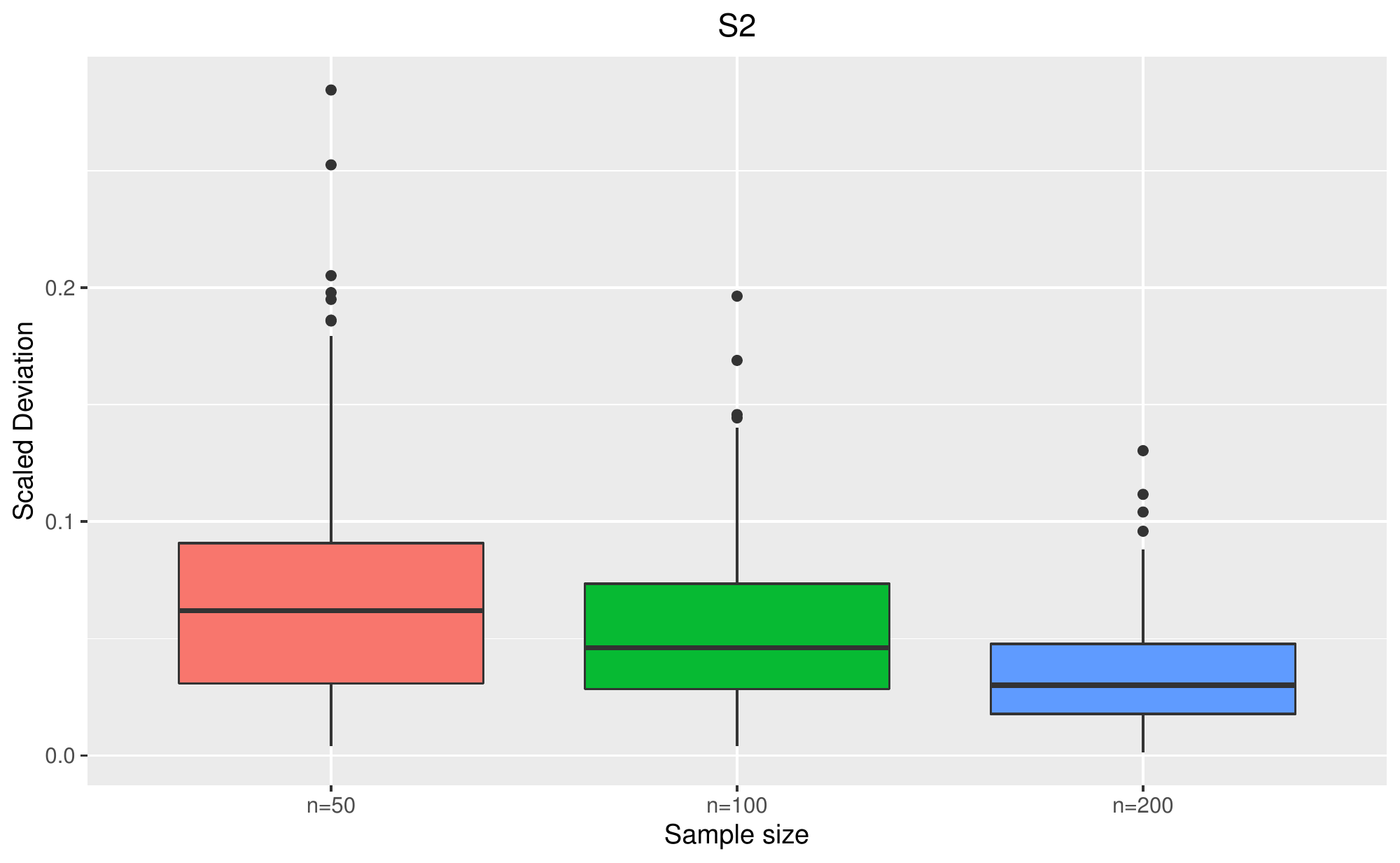}
\caption*{}
\end{minipage}
\hfill
\begin{minipage}[t]{0.22\linewidth}
\centering
\includegraphics[width=1.5in,height=2in]{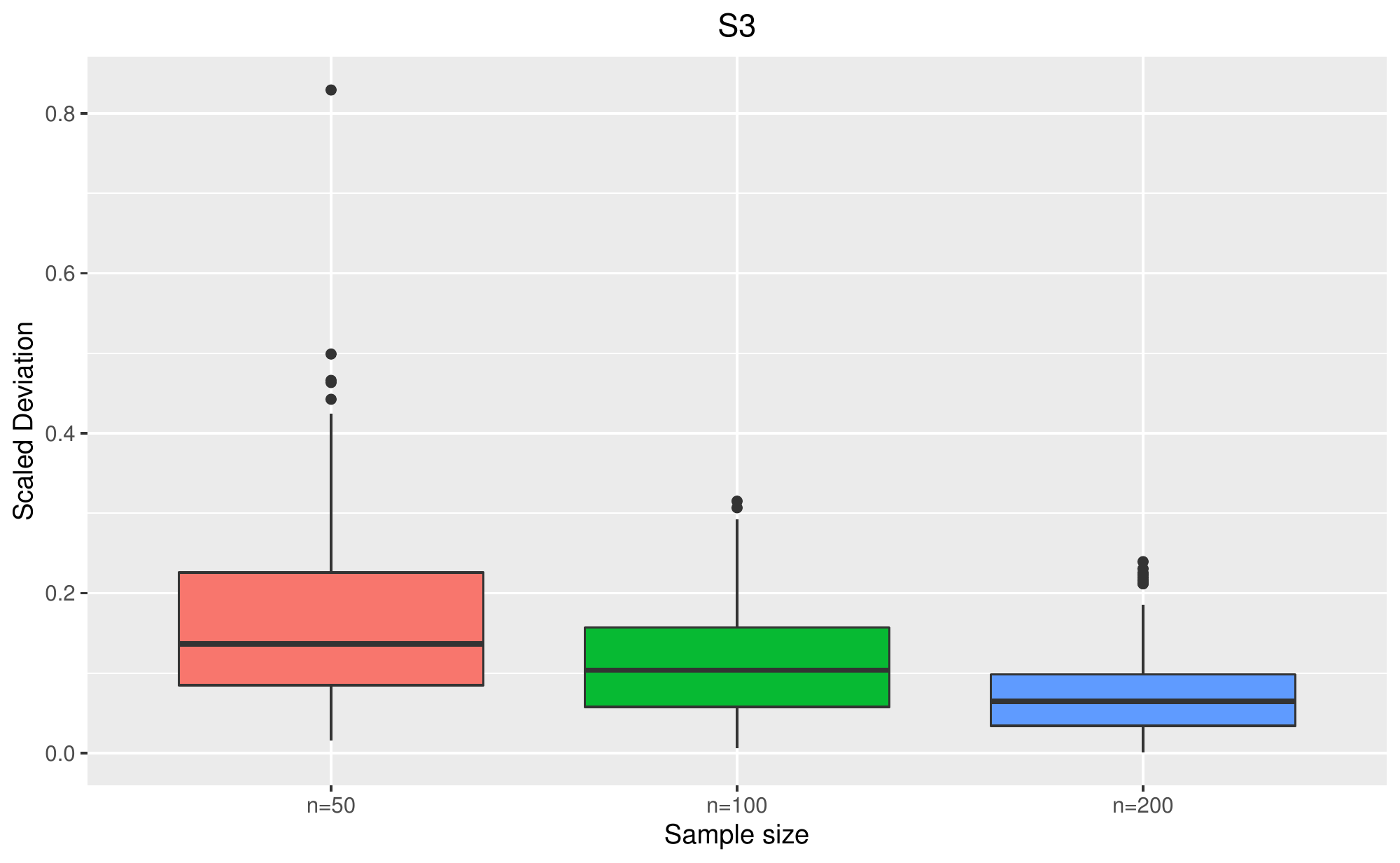}
\caption*{}
\end{minipage}
\hfill
\begin{minipage}[t]{0.22\linewidth}
\centering
\includegraphics[width=1.5in,height=2in]{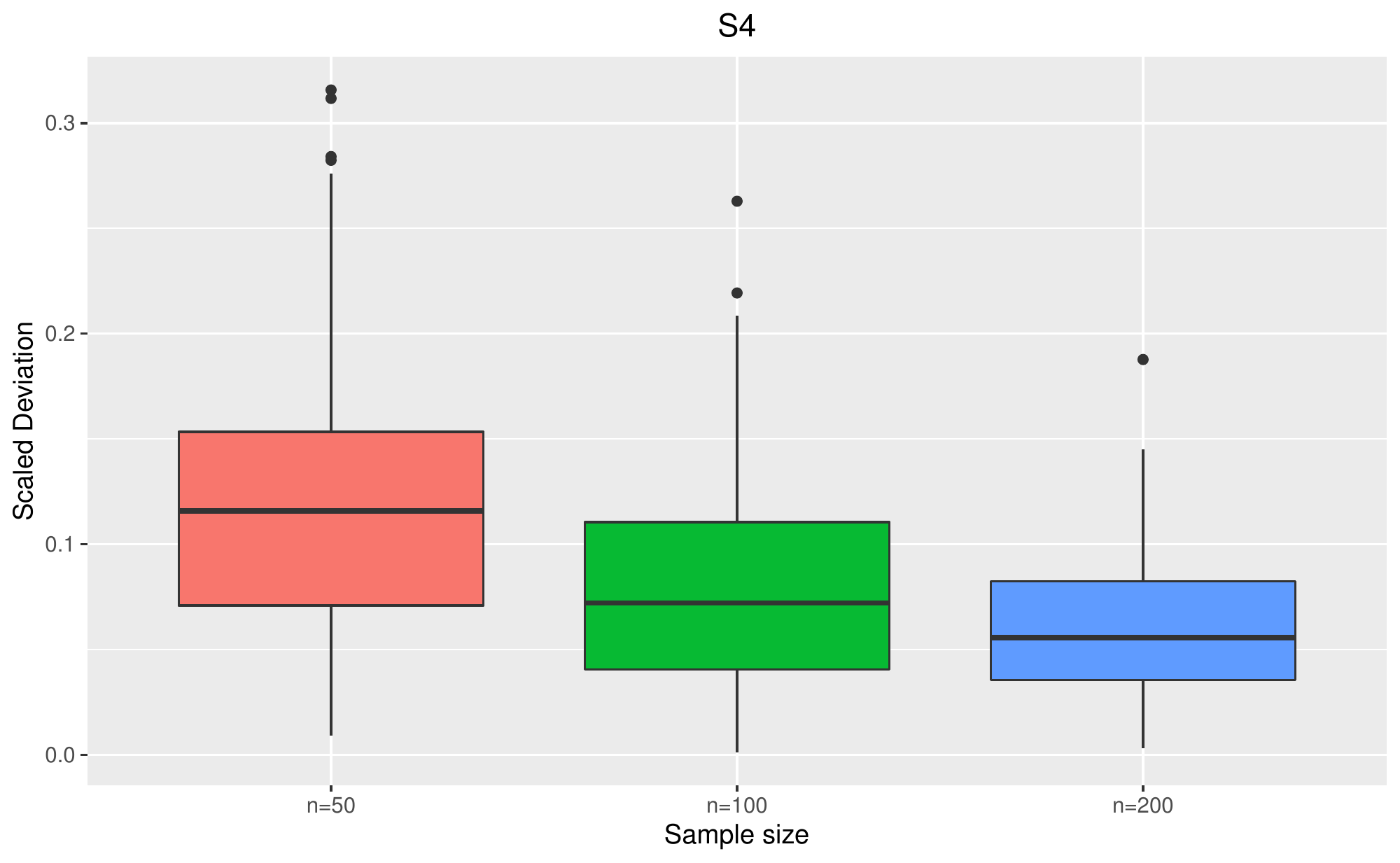}
\caption*{}
\end{minipage}
\begin{minipage}[t]{0.22\linewidth}
\centering
\includegraphics[width=1.5in,height=2in]{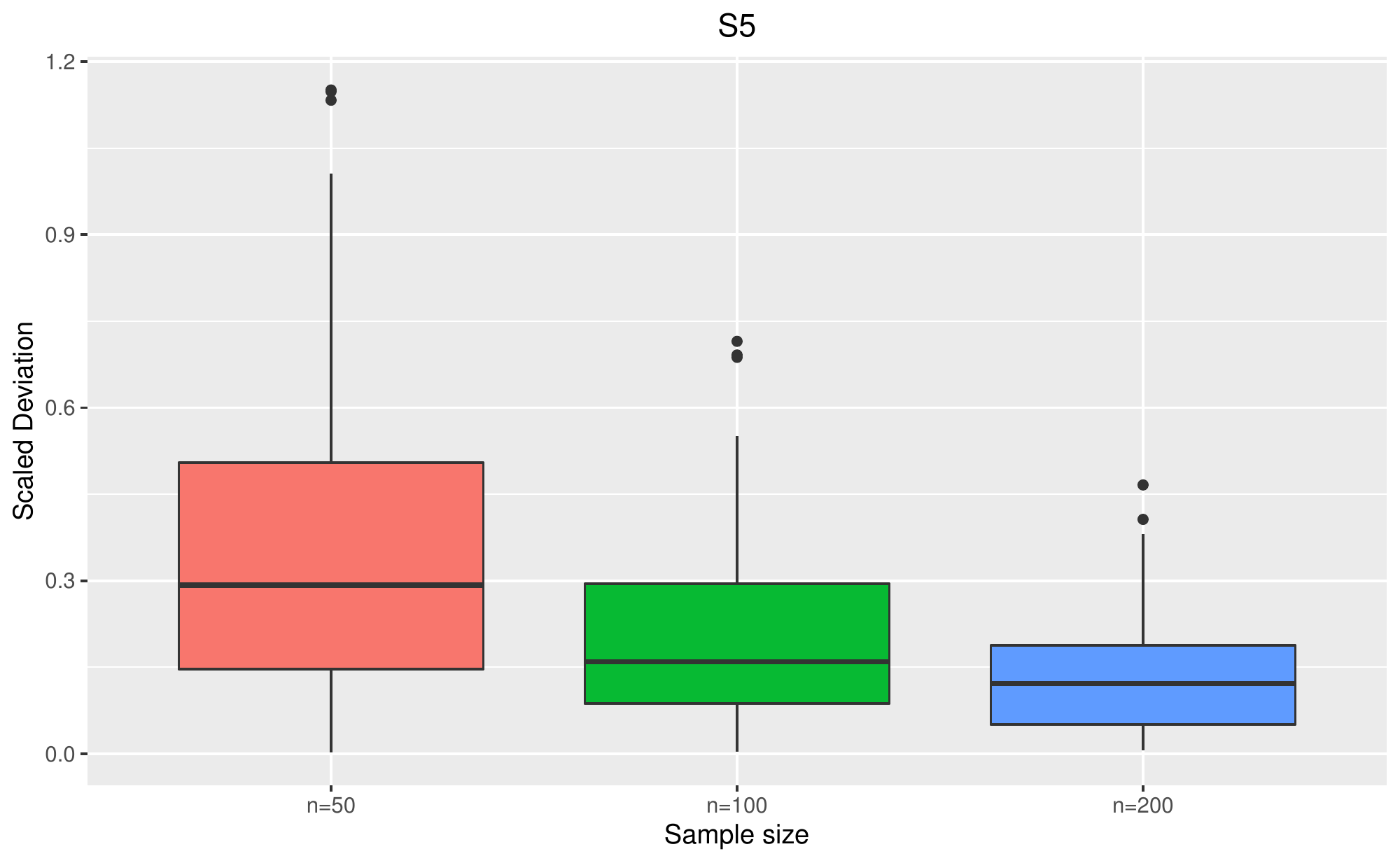}
\caption*{}
\end{minipage}
\hfill
\begin{minipage}[t]{0.22\linewidth}
\centering
\includegraphics[width=1.5in,height=2in]{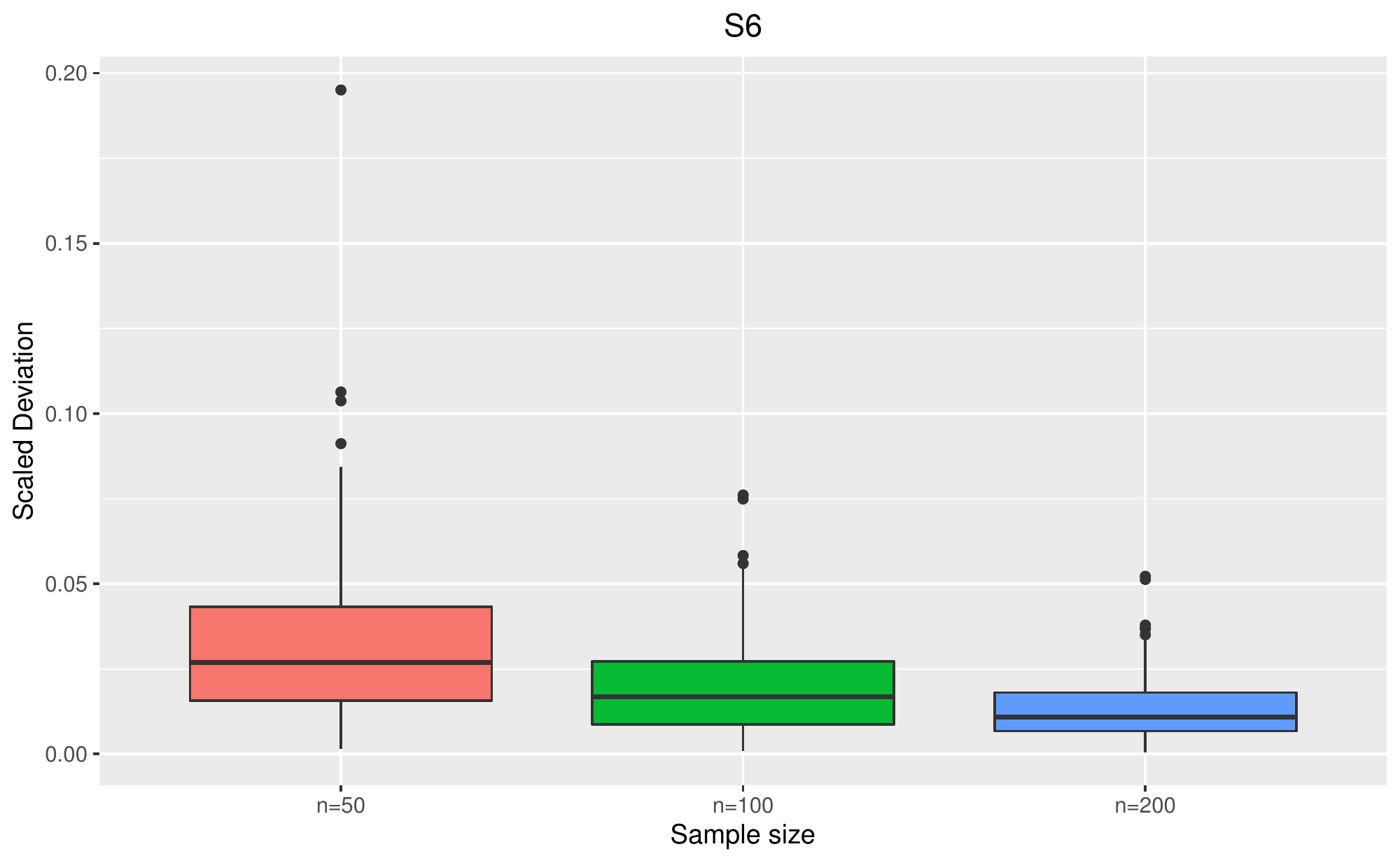}
\caption*{}
\end{minipage}
\hfill
\begin{minipage}[t]{0.22\linewidth}
\centering
\includegraphics[width=1.5in,height=2in]{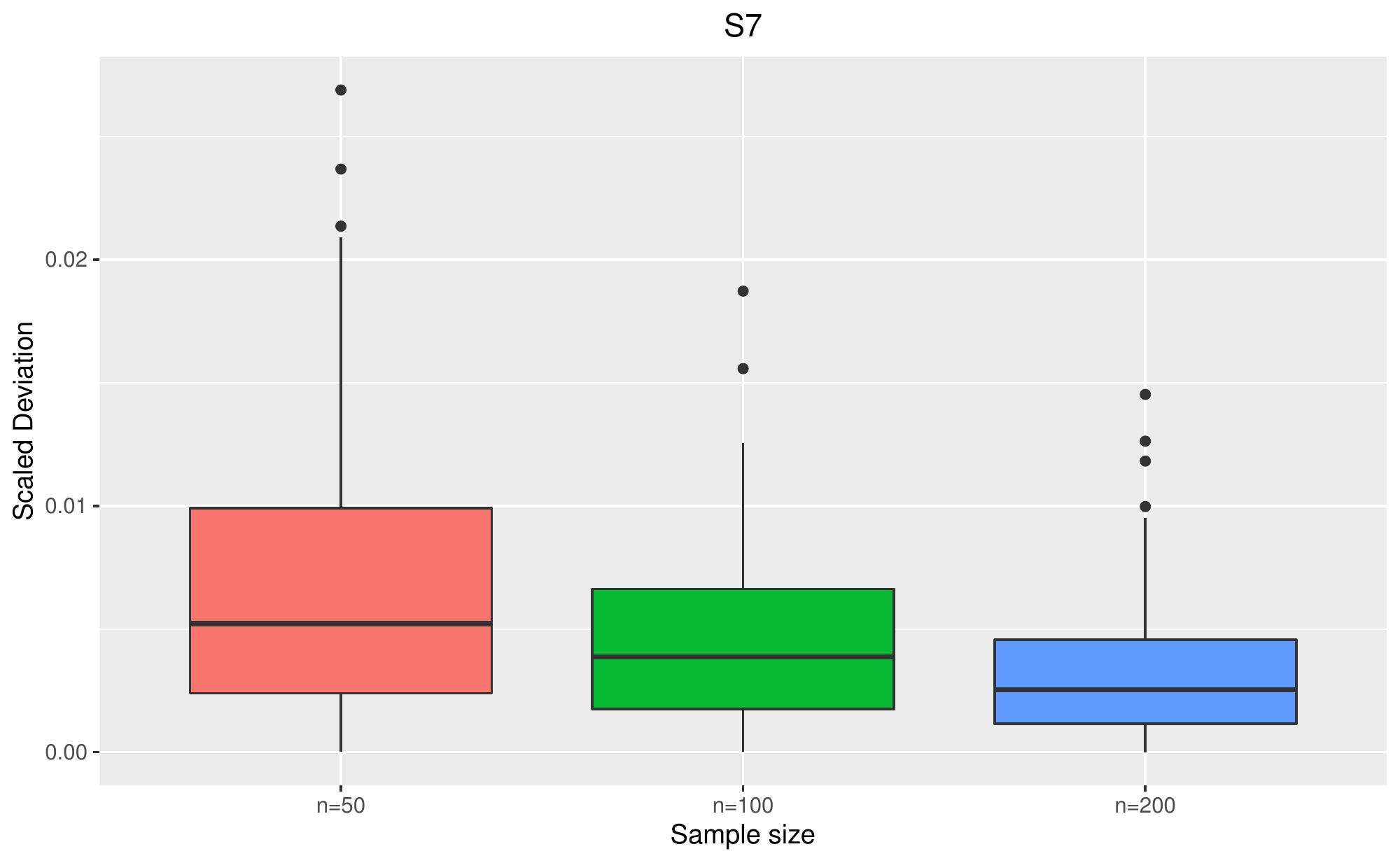}
\caption*{}
\end{minipage}
\hfill
\begin{minipage}[t]{0.22\linewidth}
\centering
\includegraphics[width=1.5in,height=2in]{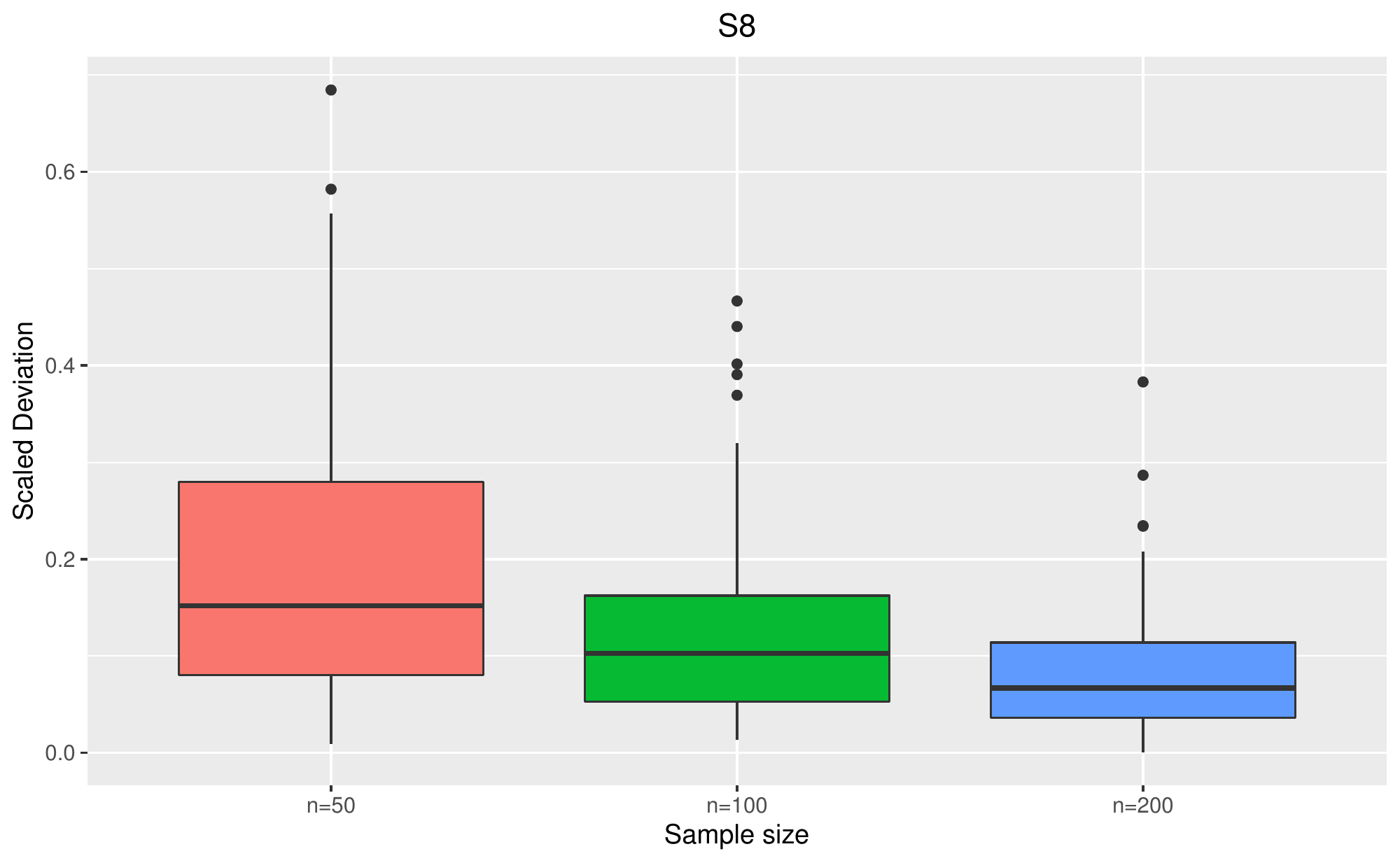}
\caption*{}
\end{minipage}
\caption{Boxplots of the scaled deviation between the linear
coefficient $\bbeta$ and its estimator $\tilde\bbeta$ in  scenario S$k$, $k=1,\cdots,8$. }
\label{boxplot}
\end{figure}

	\subsection{Size and power analysis}
	\subsubsection{Dependence on the number of projections}
First, we investigate the adequacy of the tests  with respect to   the number of projections $K$, ranging from 1 to 30. To this end, 
empirical sizes and powers are examined  in each scenario with $M=500$ Monte Carlo experiments and  $B=10000$ bootstrap samples, 
at level $\alpha=0.01,0.05,0.10$.  The sample size $n$ is taken to be 50, 100, and 200, with random projections drawn from the data-driven
approach in Table  \ref{Algorithml}.
Figure \ref{Ksize} shows the empirical  sizes of the CvM and KS tests
under the null hypothesis. The empirical rejection rate for each scenario at significance level $\alpha=0.05$   is displayed in Figure \ref{Kpower}.
The main findings are summarized as follows:	
\begin{enumerate}
\item[(i)] \textbf{ L-shaped patterns in the size curves}.  It is clear from Figure \ref{Ksize} that the empirical sizes  
present  apparent decrease as the number $K$ increases, and stabilize below the  significance level  in each scenario. This is caused by the conservativeness of the FDR method, leading to an increment of the type I error for large $K$, see \cite{benjamini2001control}.

% by the conservativeness of the FDR correctionâunder H0, it ensures that the rejection rate is at most 

\item[(ii)] \textbf{Occasional bumps yielding power gains}.  We can see from Figure \ref{Kpower} that 
 the empirical powers remain constant with large K, and sharp  jumps  
 exist  at some moderate  values of K,  providing  a significant power gain, especially for the sample 
 size $n=200$.

\item[(iii)] \textbf{Predominance of CvM over KS}.  In general,  the CvM tests perform much better than the KS tests in power  and no worse in size.
It is reasonable since quadratic norms in goodness-of-fit are usually  more powerful than sup-norms.
\end{enumerate}

For large $K$, the tests have an obvious  under-rejection of the null hypothesis but a strong  power  against  the alternative.
A small $K$  may  result in  a large size, together with a weak power. 
Combining these facts,  we propose to choose a moderate  number $K=7$, which 
achieves a  perfect balance between the size and power performance. 
For $n=200$ and $K=7$,  both CvM and KS tests   show accurate calibrations  almost in all scenarios  based on  three significance levels, except  $\alpha=0.10$ for scenarios S2 and S8, and reach a high power either  $d=1$ or $d=2$. 
Moreover,  the computational ease brought by a not large $K$ is rather important, which avoids a large number of  bootstrap  times $B$ to maintain the precision, see \citet{cuesta2019goodness}.

\begin{figure}[h]
\begin{minipage}[t]{0.3\linewidth}
\centering
\includegraphics[width=2.2in,height=2in]{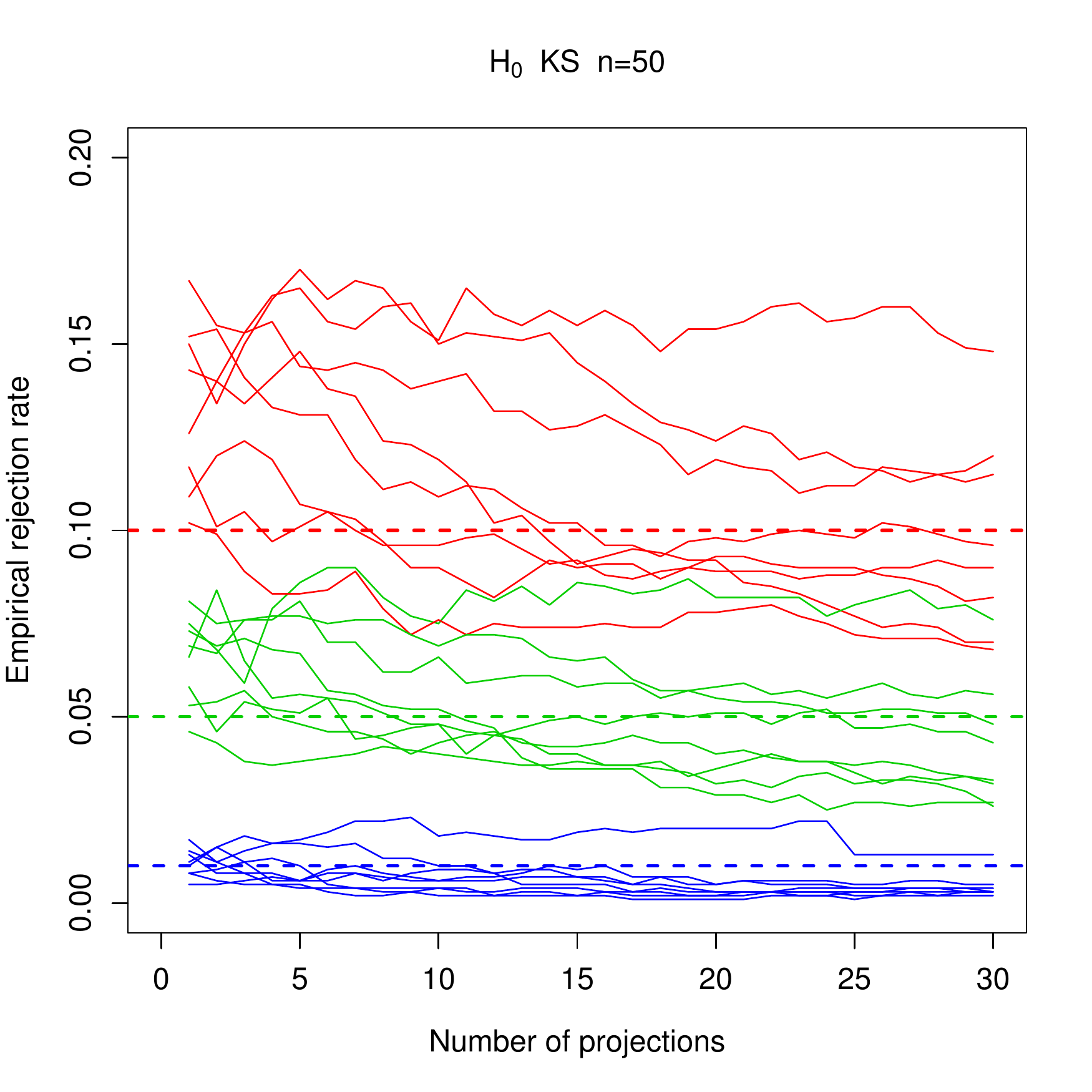}
\caption*{}
\end{minipage}
\hfill
\begin{minipage}[t]{0.3\linewidth}
\centering
\includegraphics[width=2.2in,height=2in]{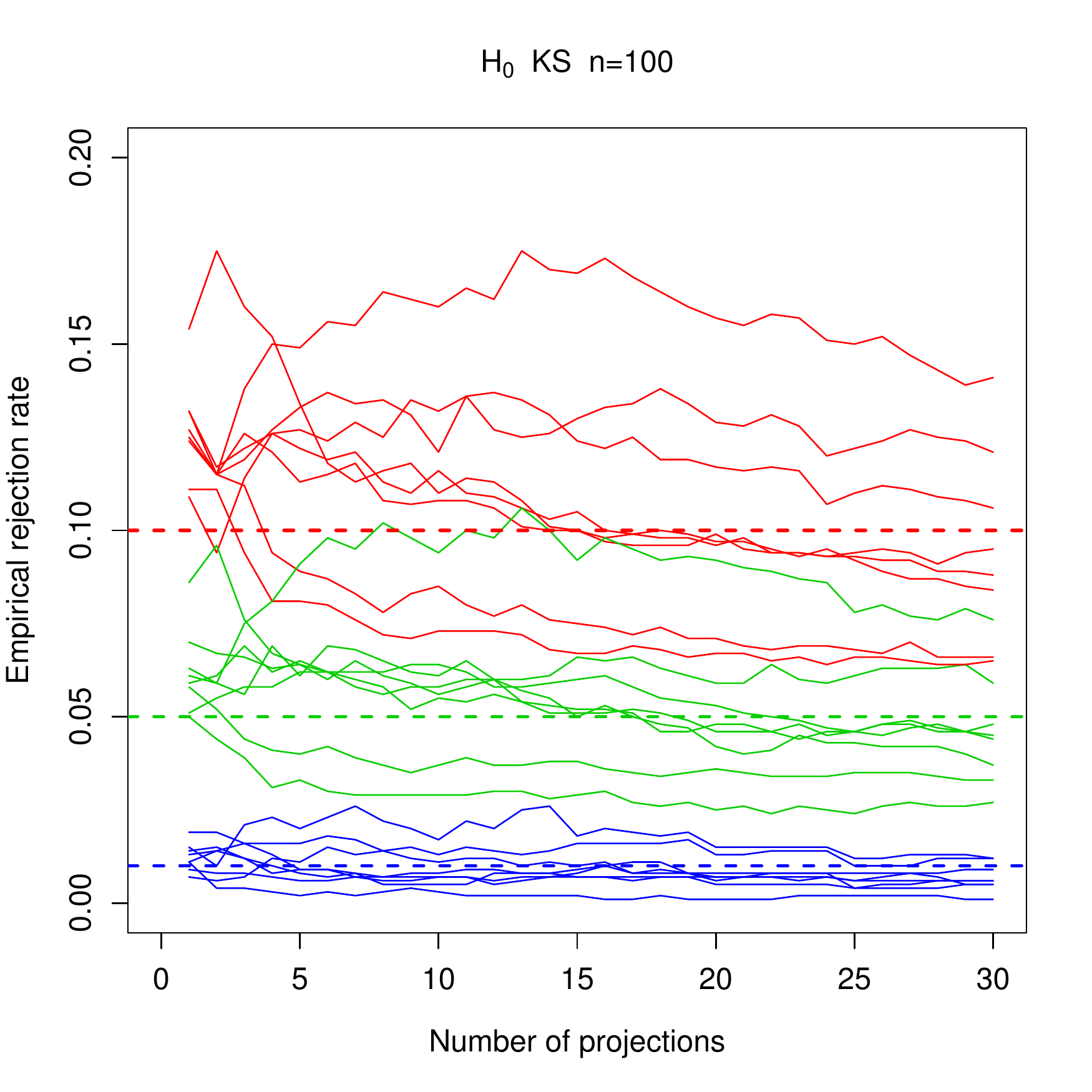}
\caption*{}
\end{minipage}
\hfill
\begin{minipage}[t]{0.3\linewidth}
\centering
\includegraphics[width=2.2in,height=2in]{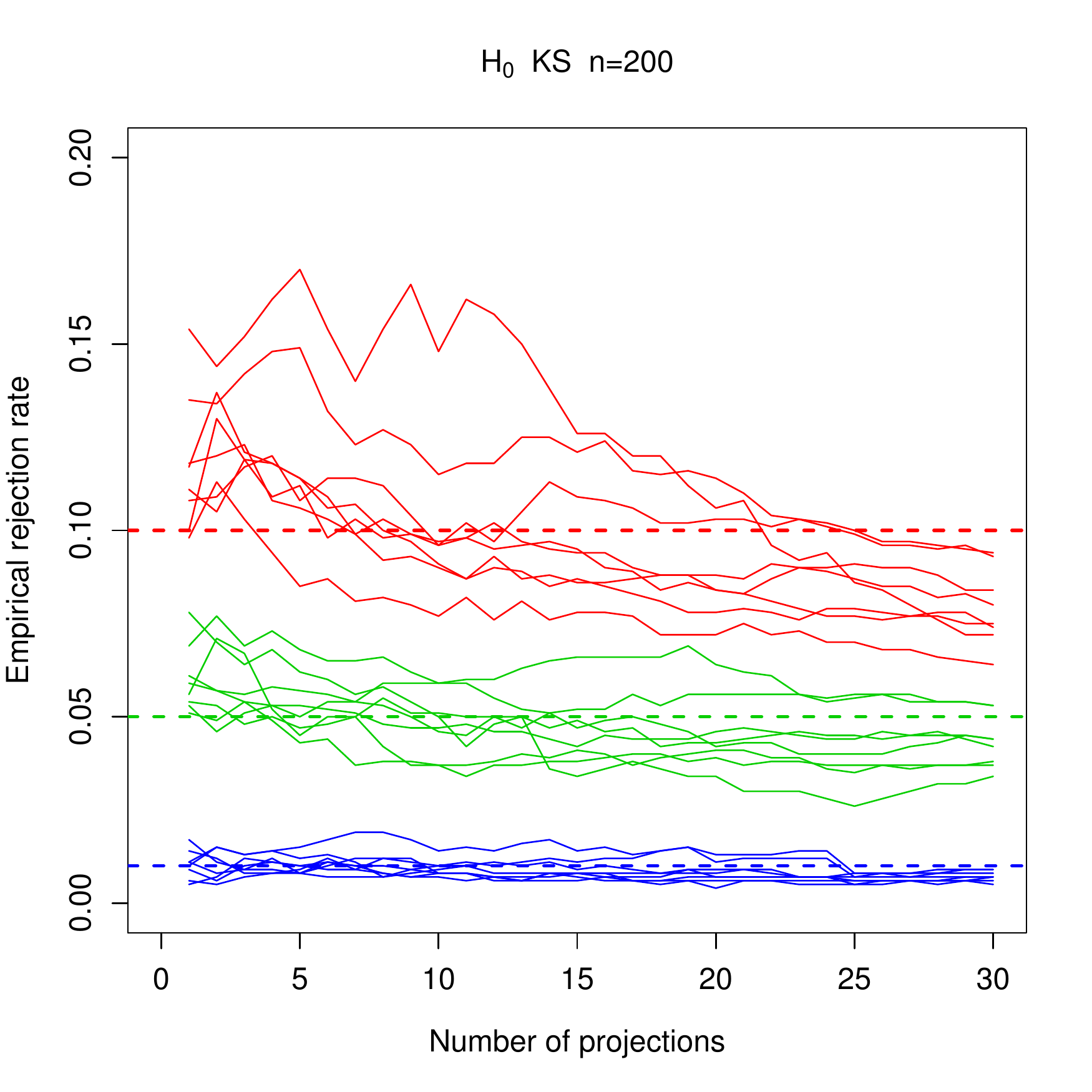}
\caption*{}
\end{minipage}
\hfill
\begin{minipage}[t]{0.3\linewidth}
\centering
\includegraphics[width=2.2in,height=2in]{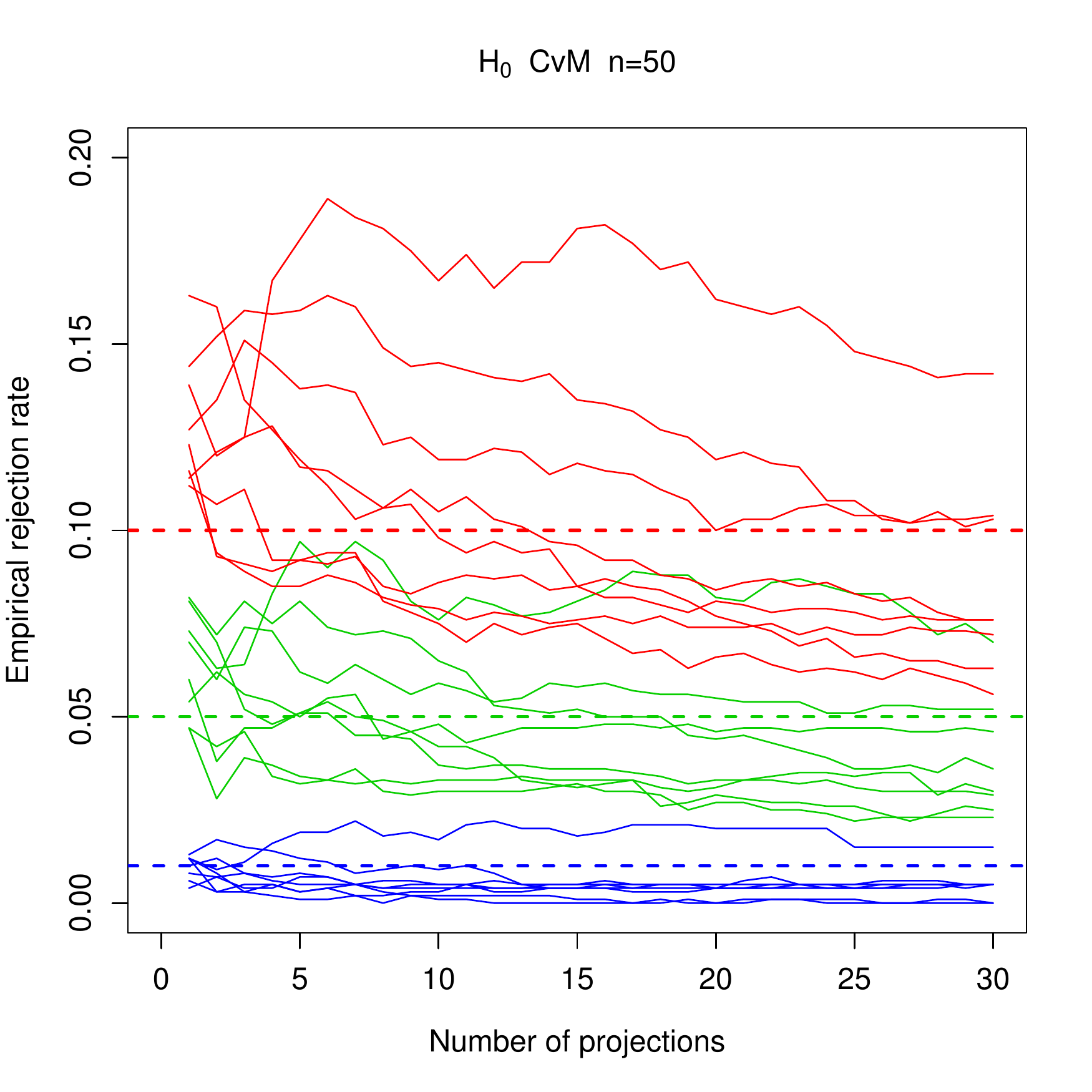}
\caption*{}
\end{minipage}
\hfill
\begin{minipage}[t]{0.3\linewidth}
\centering
\includegraphics[width=2.2in,height=2in]{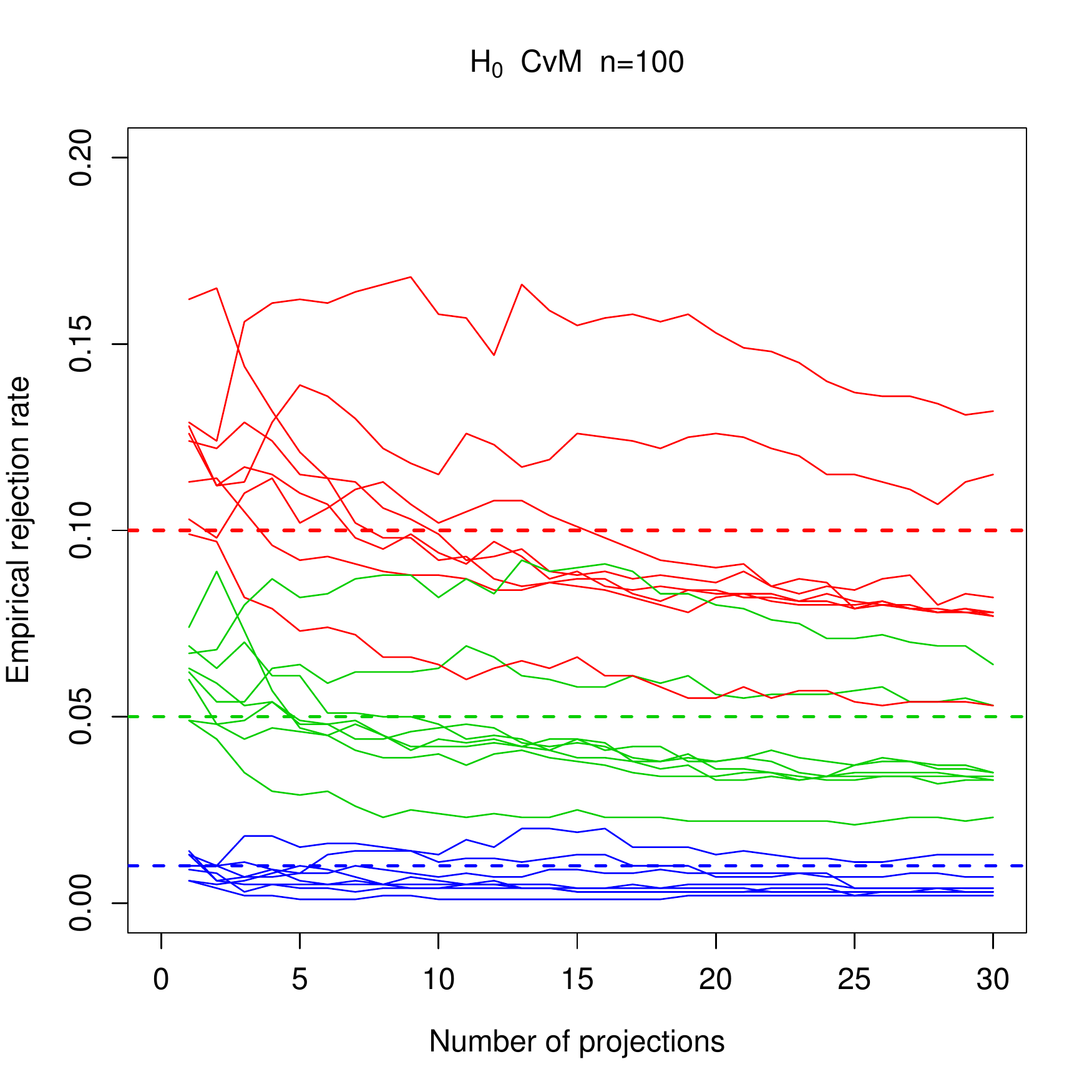}
\caption*{}
\end{minipage}
\hfill
\begin{minipage}[t]{0.3\linewidth}
\centering
\includegraphics[width=2.2in,height=2in]{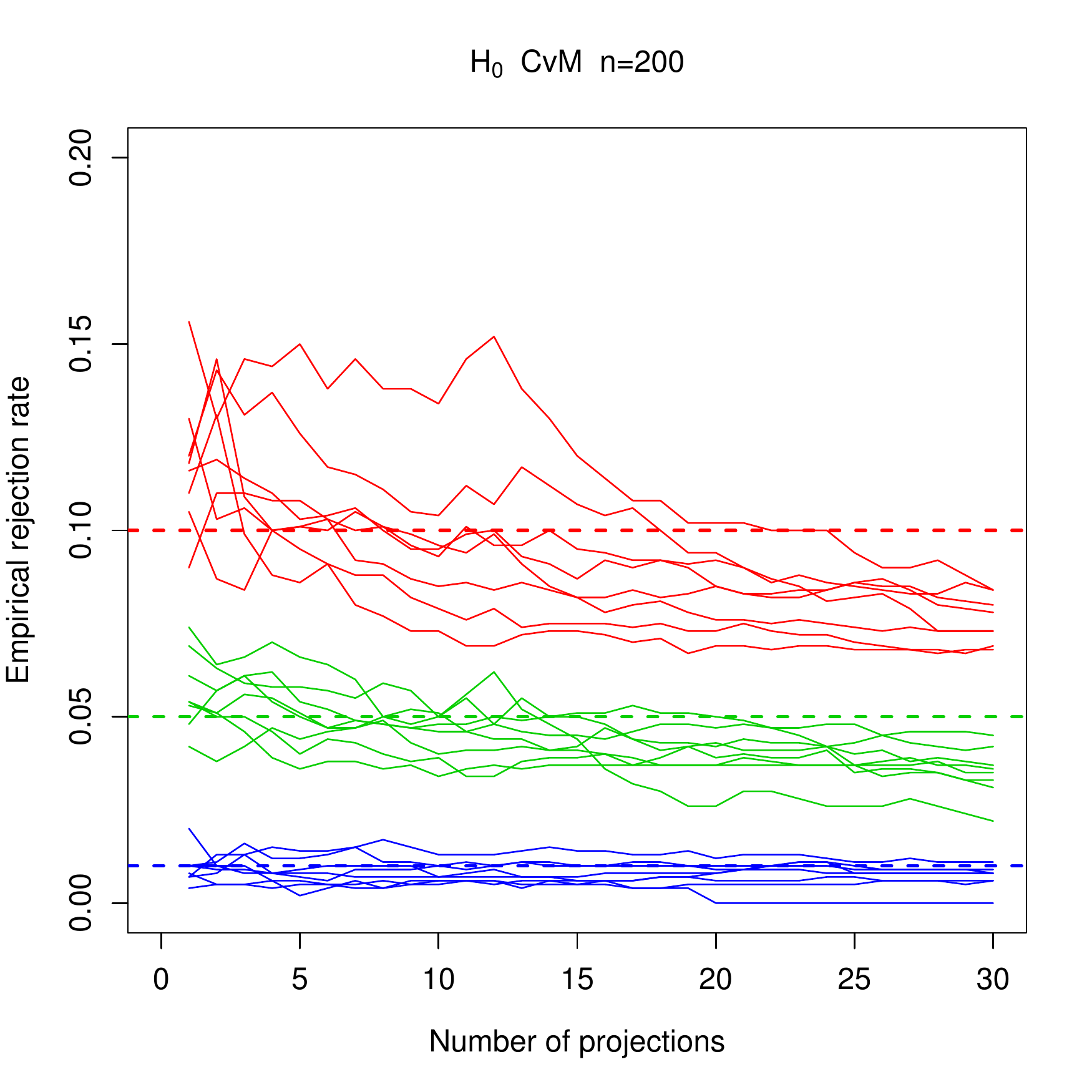}
\caption*{}
\end{minipage}
\caption{Empirical sizes of the CvM and KS tests with the number of projections from 1 to 30  in scenario S$k$, $k=1,\cdots,8$, based on the sample sizes $n=50$, $100$, and $200$, respectively.  The upper, middle, and lower solid lines correspond to the significance levels  $\alpha=0.10$, $0.05$, and $0.01$, respectively.}
\label{Ksize}
\end{figure}

\begin{figure}[h]
\begin{minipage}[t]{0.3\linewidth}
\centering
\includegraphics[width=2in,height=2in]{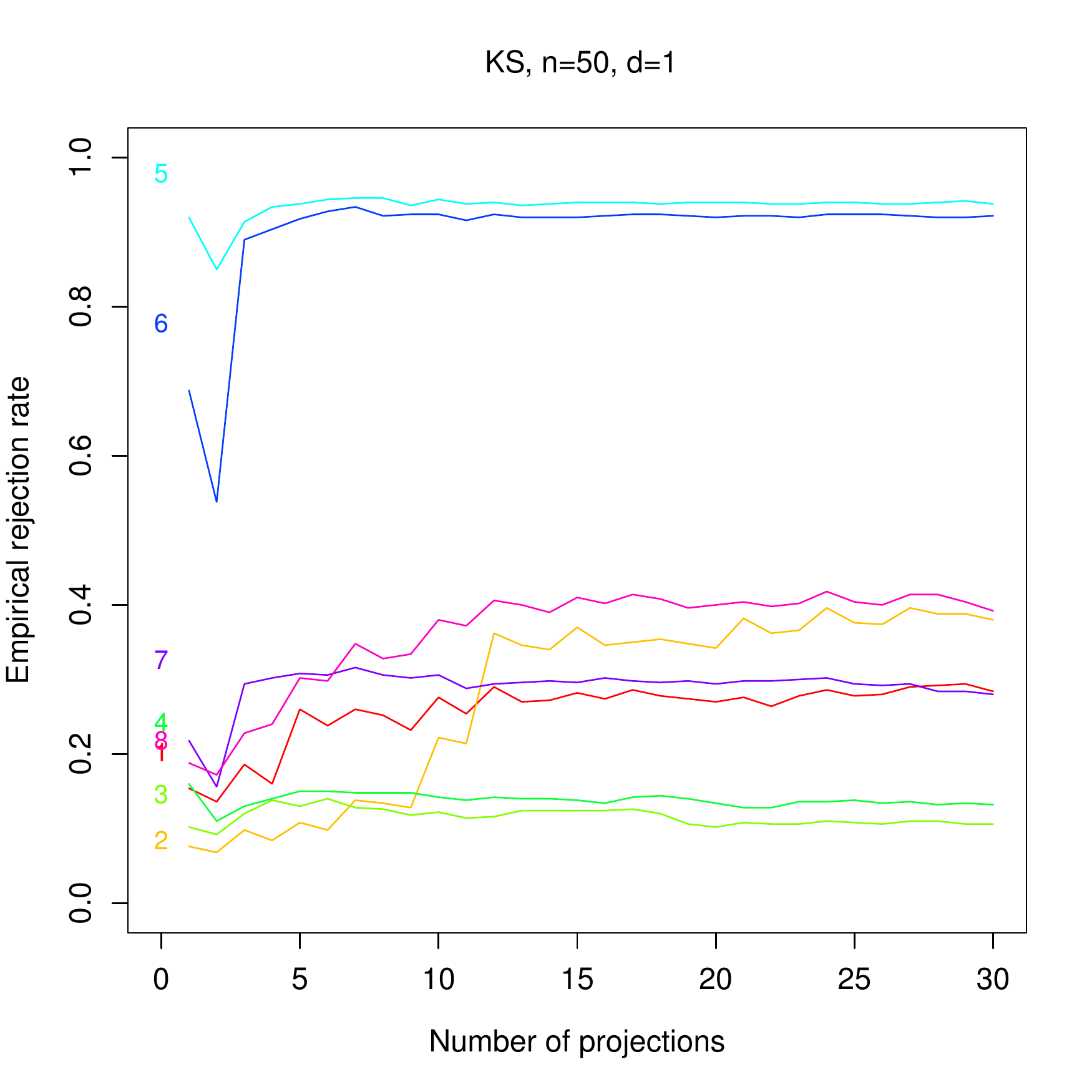}
\caption*{}
\end{minipage}
\hfill
\begin{minipage}[t]{0.3\linewidth}
\centering
\includegraphics[width=2.2in,height=2in]{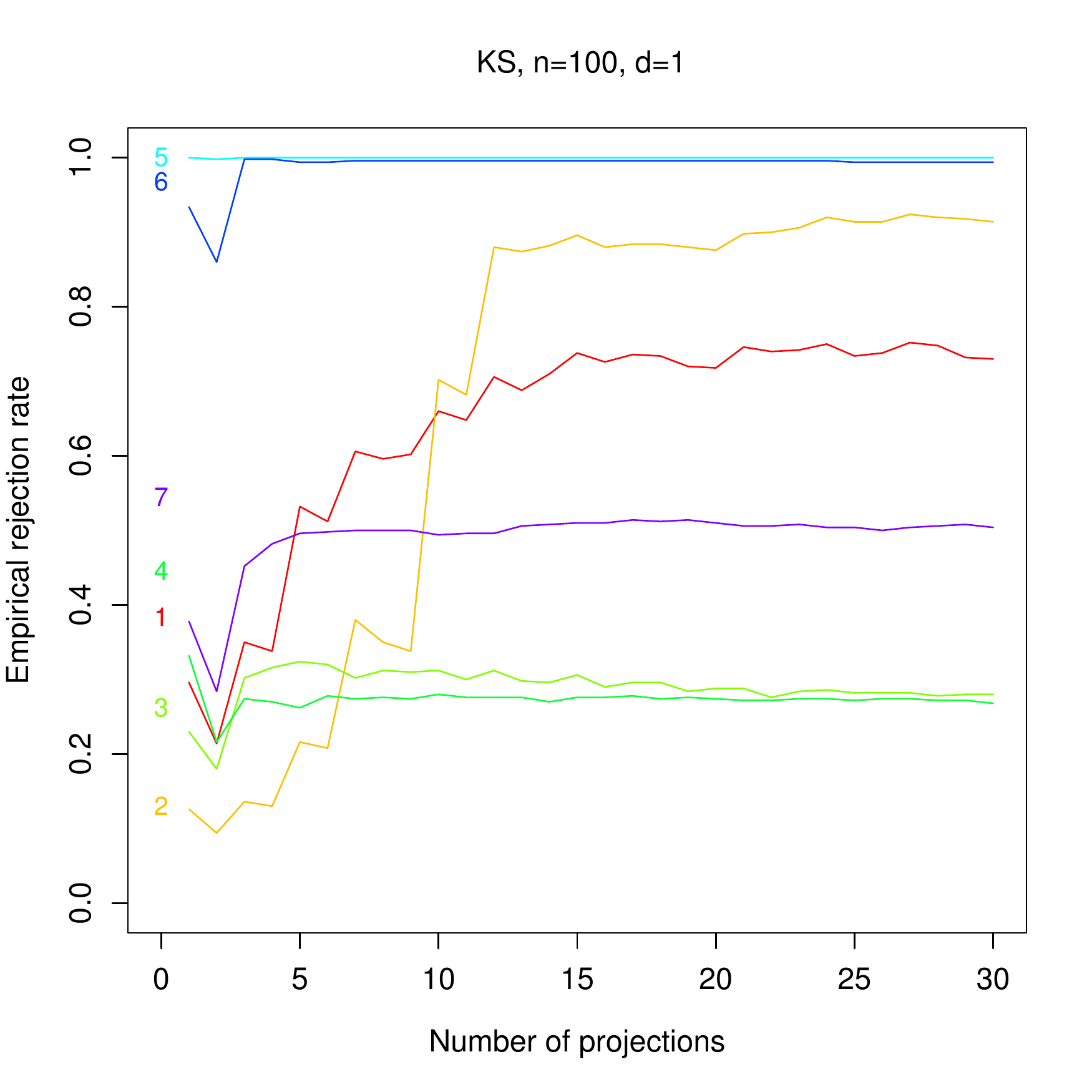}
\caption*{}
\end{minipage}
\hfill
\begin{minipage}[t]{0.3\linewidth}
\centering
\includegraphics[width=2.2in,height=2in]{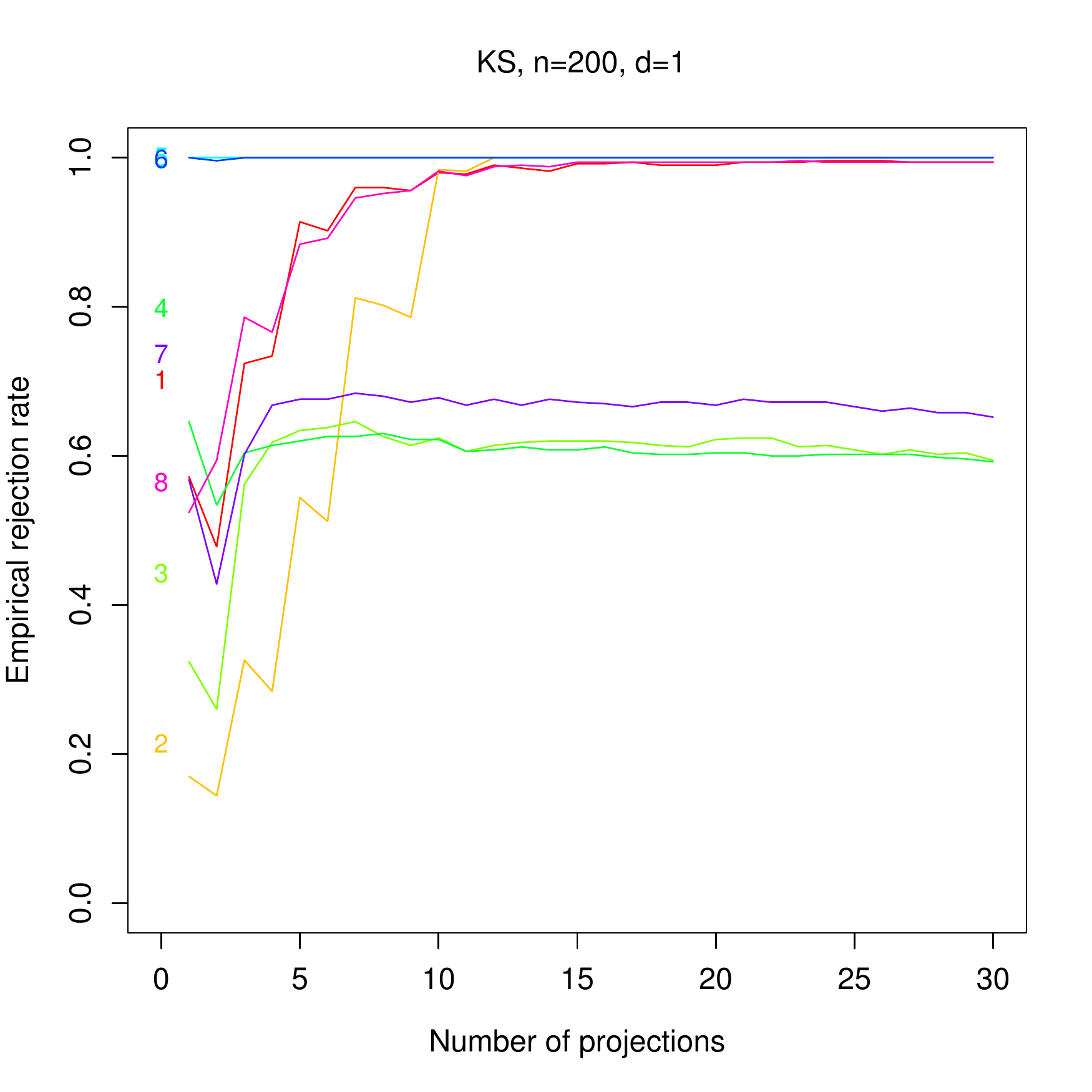}
\caption*{}
\end{minipage}
\hfill
\begin{minipage}[t]{0.3\linewidth}
\centering
\includegraphics[width=2in,height=2in]{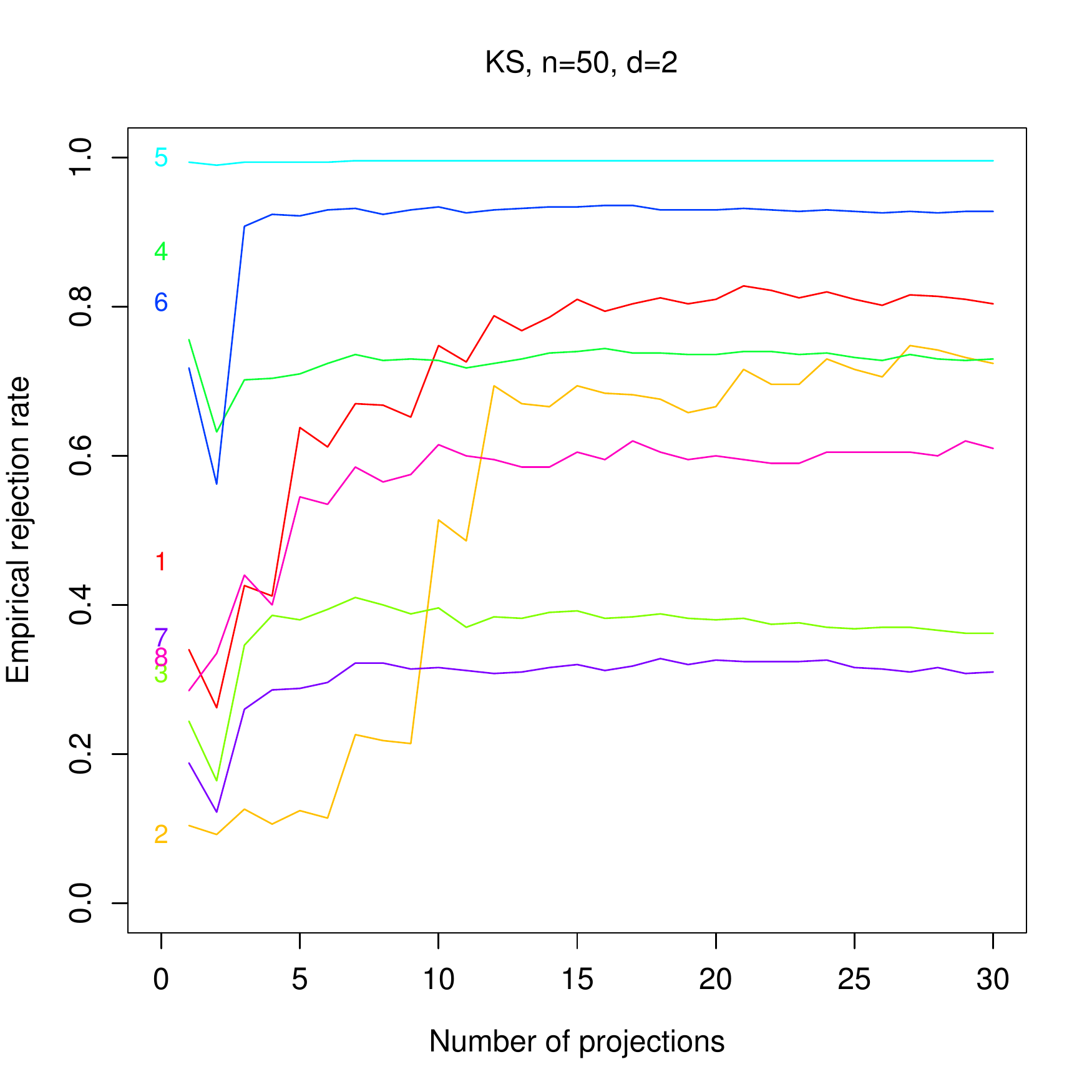}
\caption*{}
\end{minipage}
\hfill
\begin{minipage}[t]{0.3\linewidth}
\centering
\includegraphics[width=2.2in,height=2in]{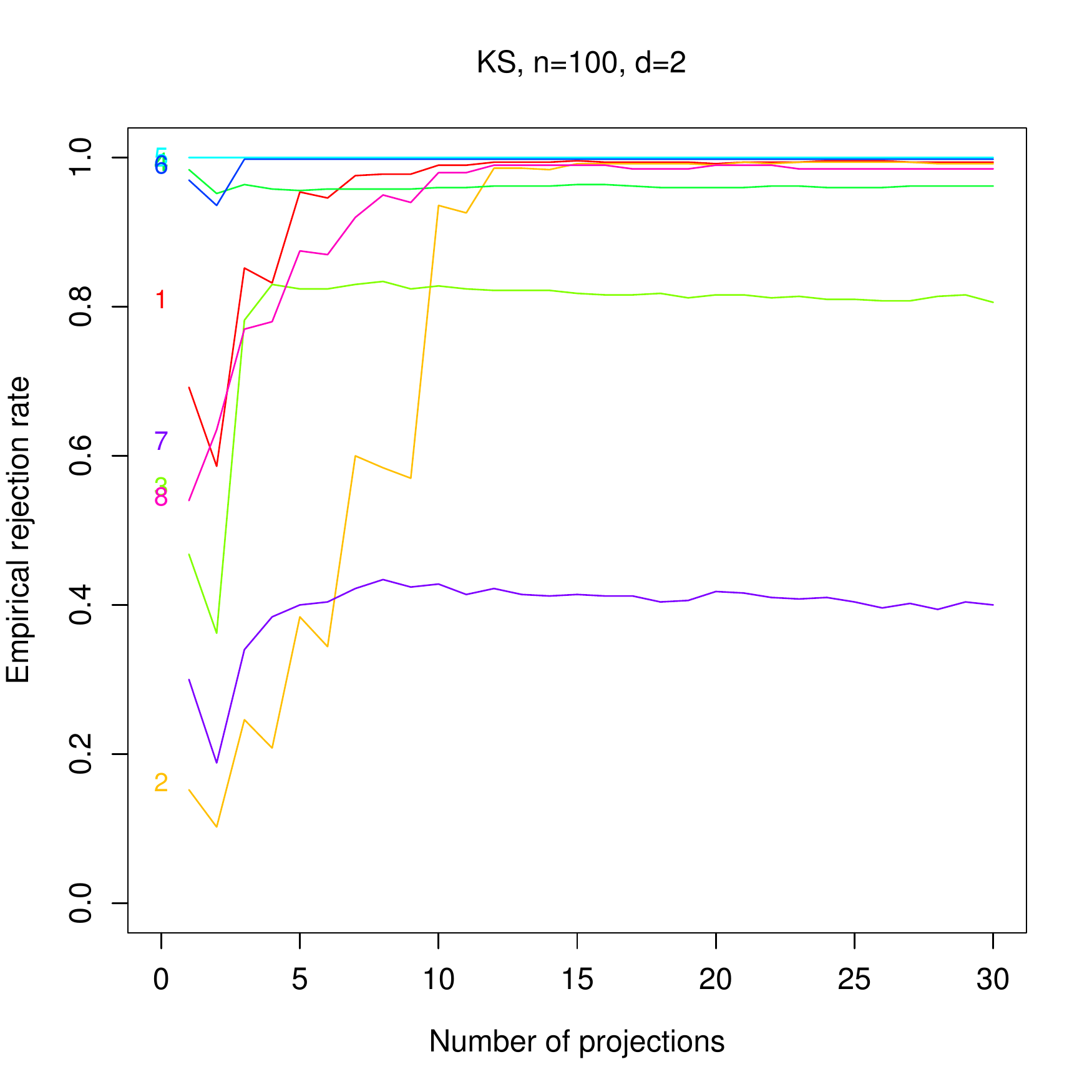}
\caption*{}
\end{minipage}
\hfill
\begin{minipage}[t]{0.3\linewidth}
\centering
\includegraphics[width=2.2in,height=2in]{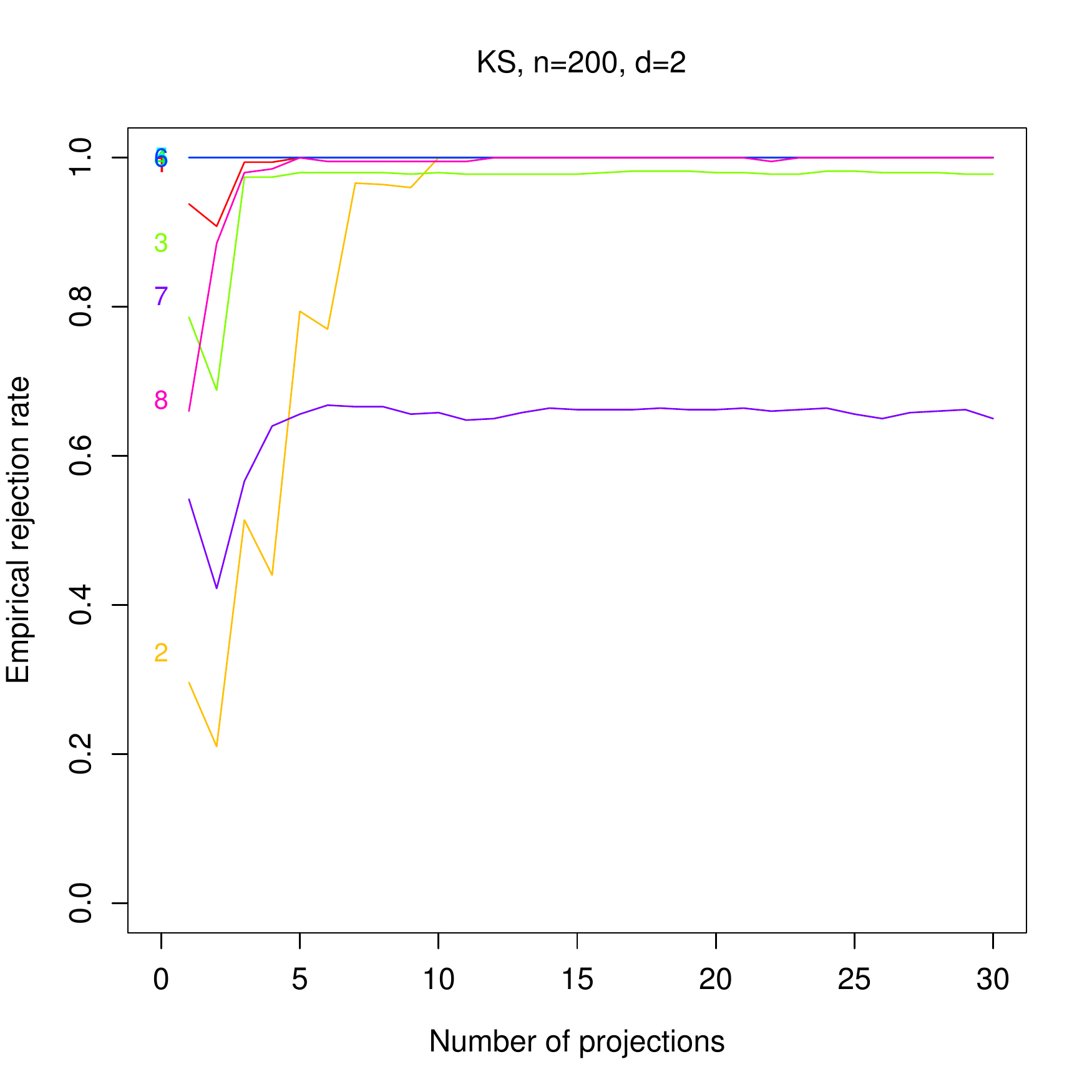}
\caption*{}
\end{minipage}
\hfill
\begin{minipage}[t]{0.3\linewidth}
\centering
\includegraphics[width=2.2in,height=2in]{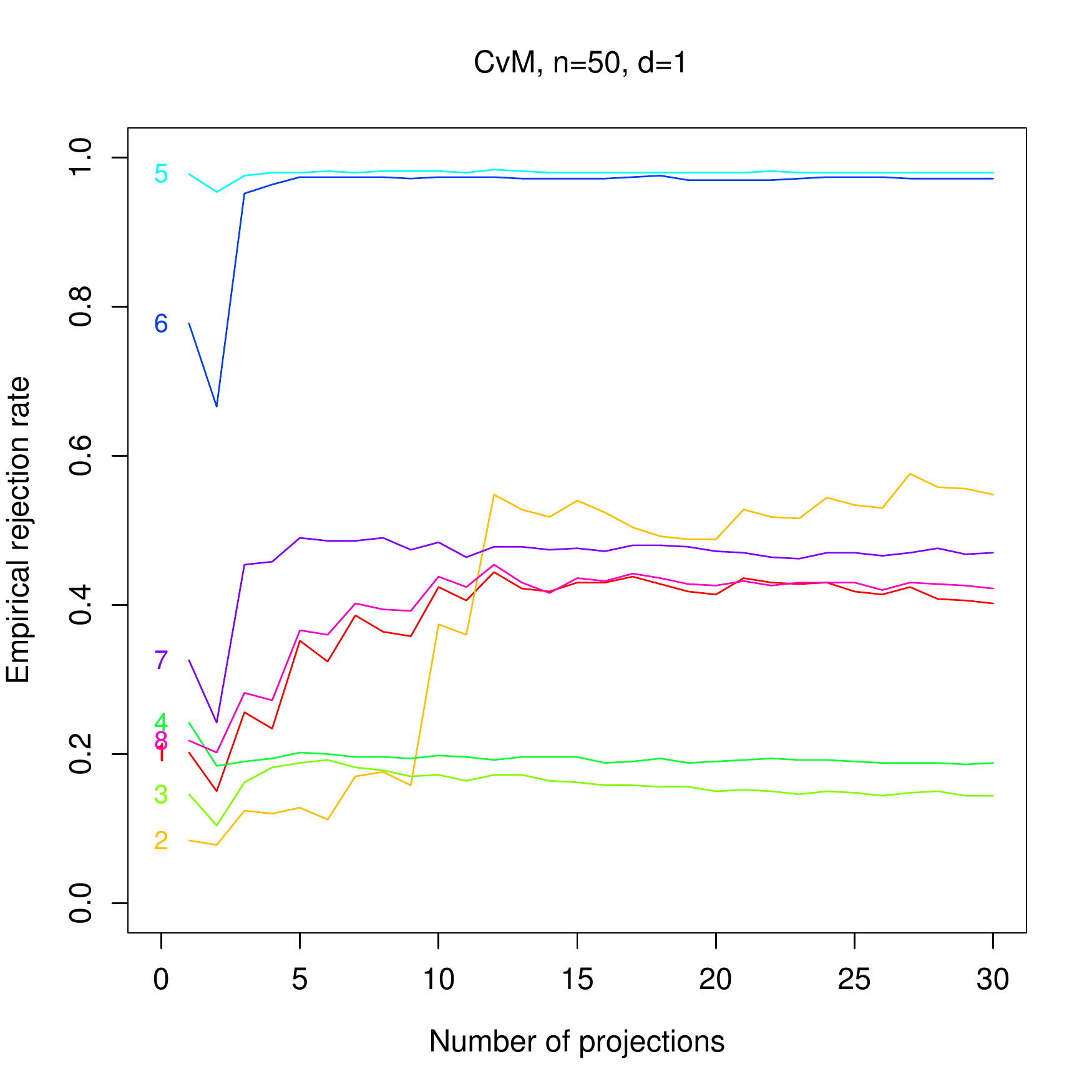}
\caption*{}
\end{minipage}
\hfill
\begin{minipage}[t]{0.3\linewidth}
\centering
\includegraphics[width=2.2in,height=2in]{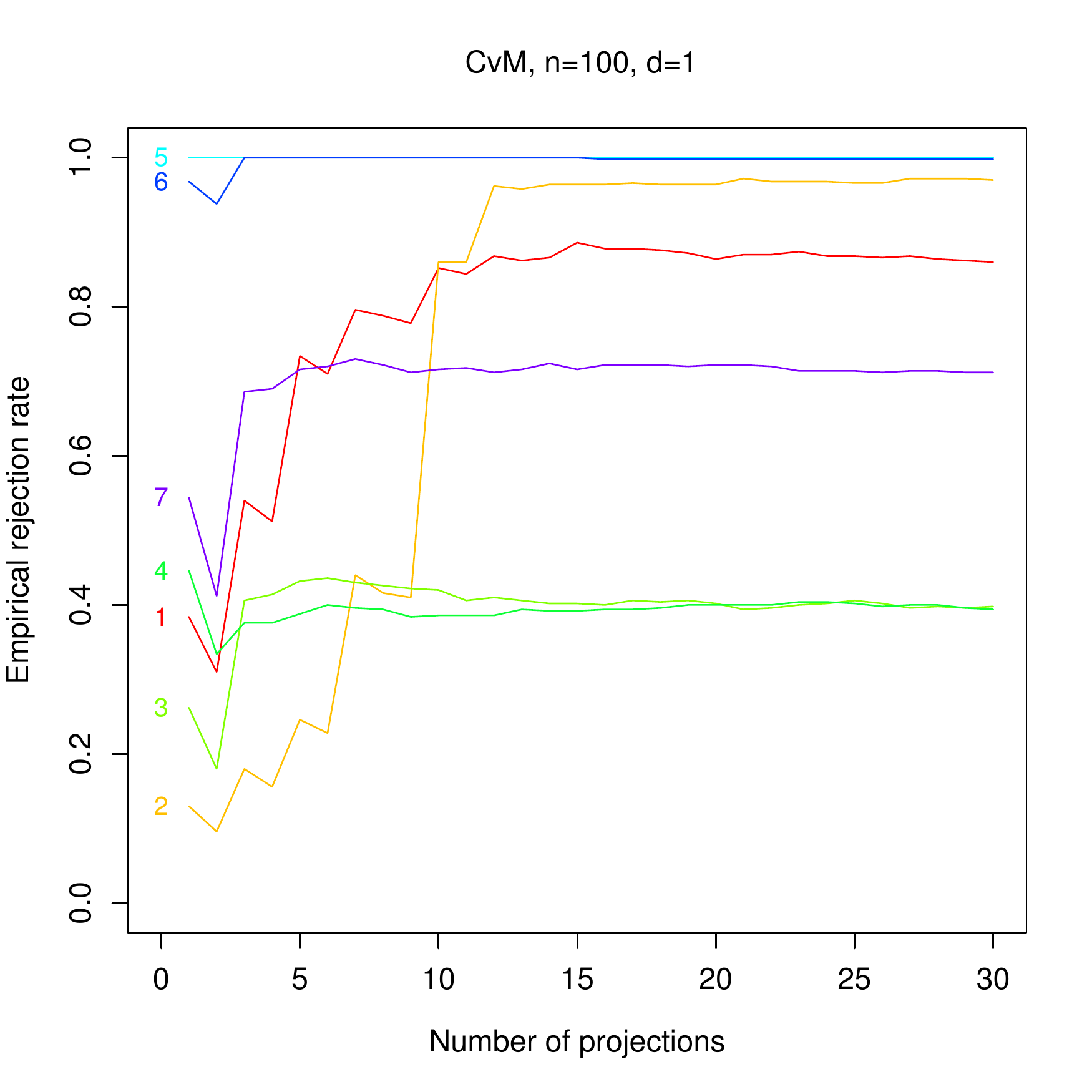}
\caption*{}
\end{minipage}
\hfill
\begin{minipage}[t]{0.3\linewidth}
\centering
\includegraphics[width=2.2in,height=2in]{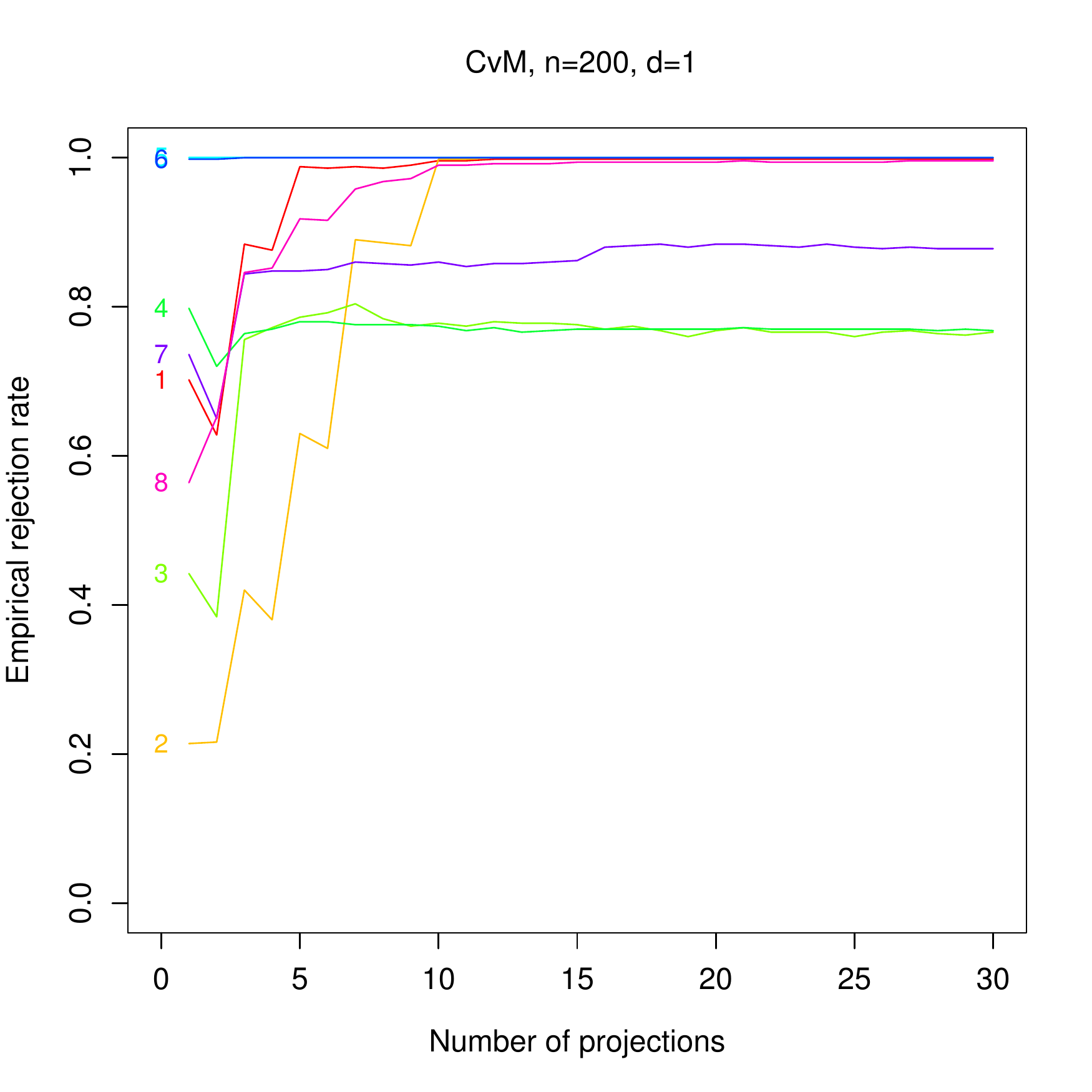}
\caption*{}
\end{minipage}
\begin{minipage}[t]{0.3\linewidth}
\centering
\includegraphics[width=2.2in,height=2in]{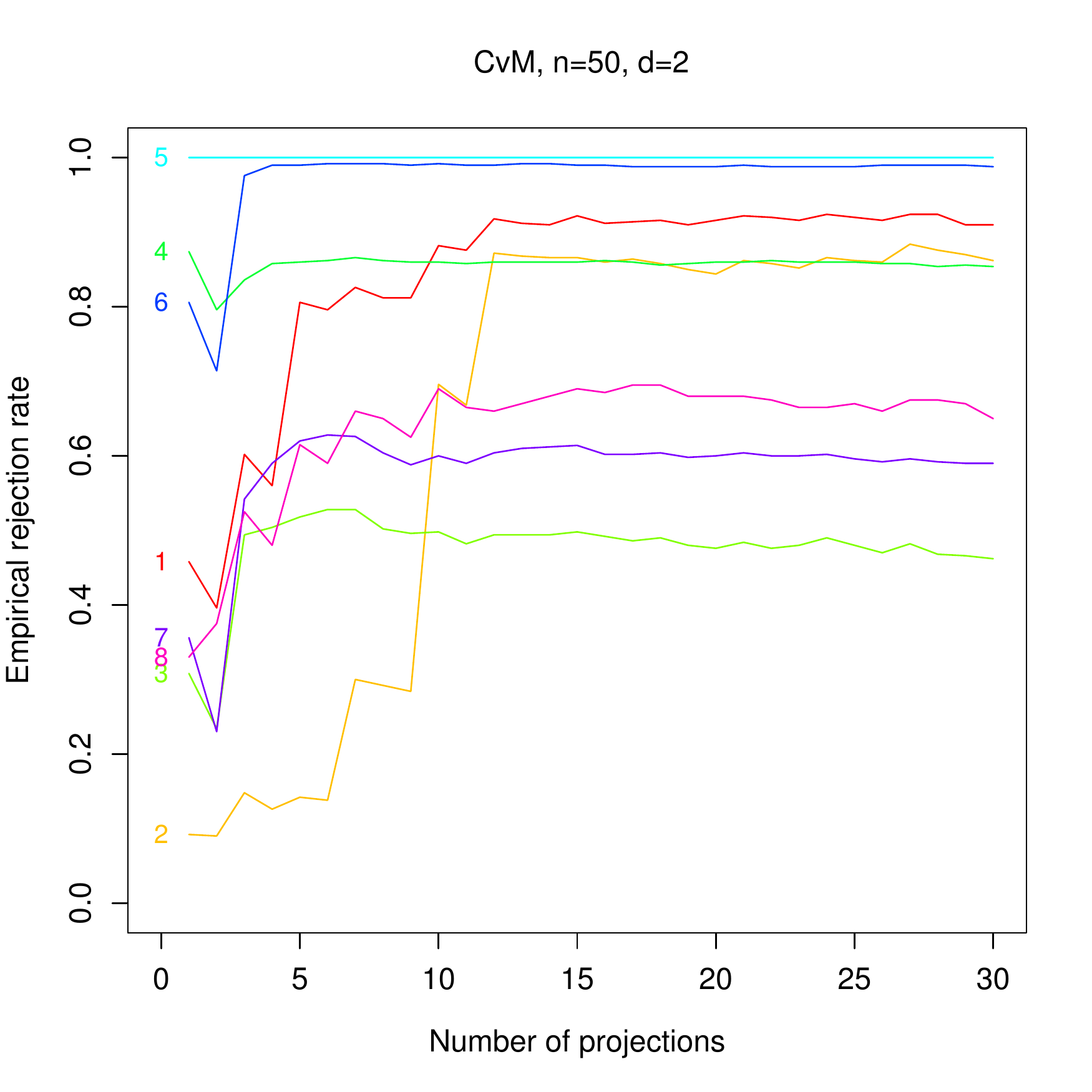}
\caption*{}
\end{minipage}
\hfill
\begin{minipage}[t]{0.3\linewidth}
\centering
\includegraphics[width=2.2in,height=2in]{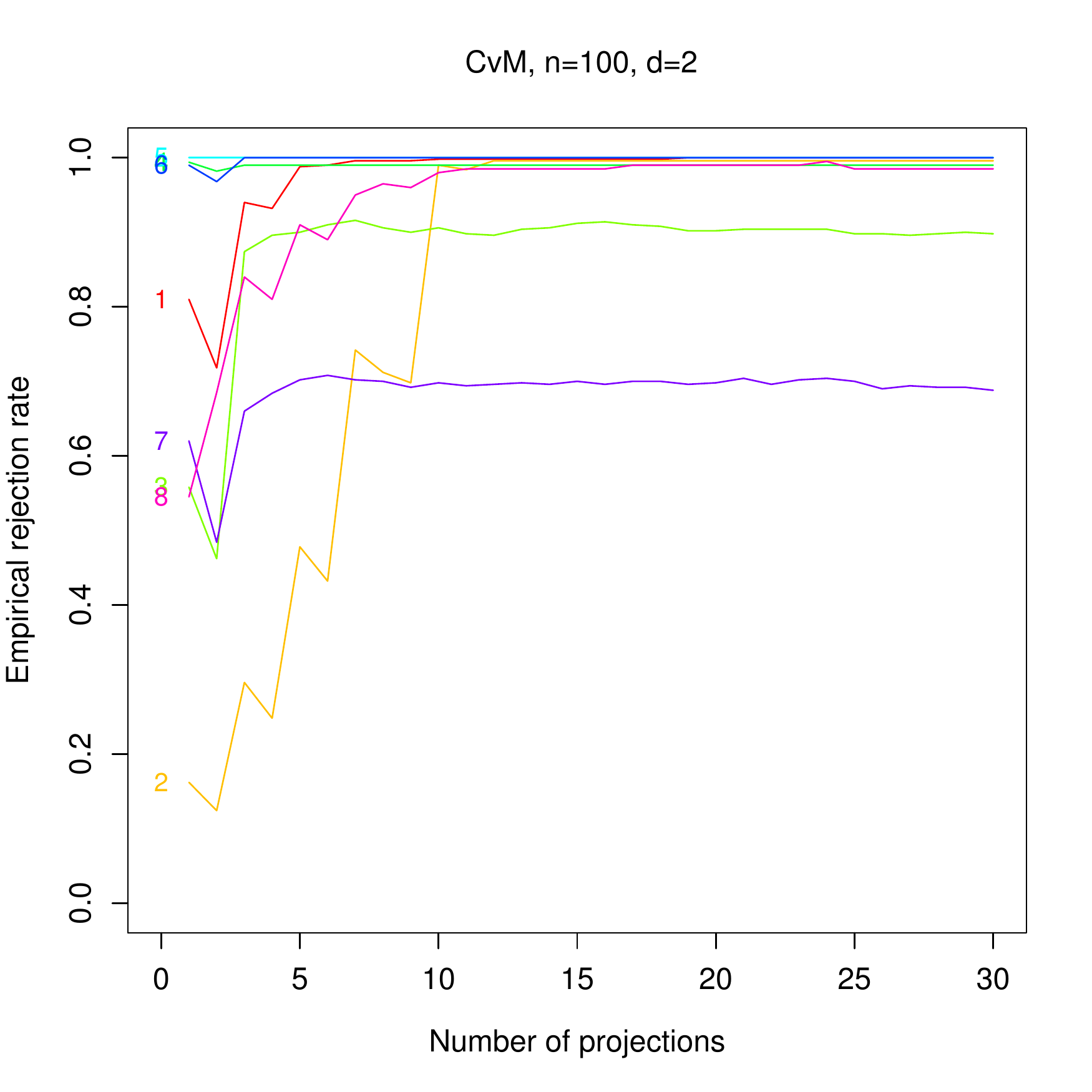}
\caption*{}
\end{minipage}
\hfill
\begin{minipage}[t]{0.3\linewidth}
\centering
\includegraphics[width=2.2in,height=2in]{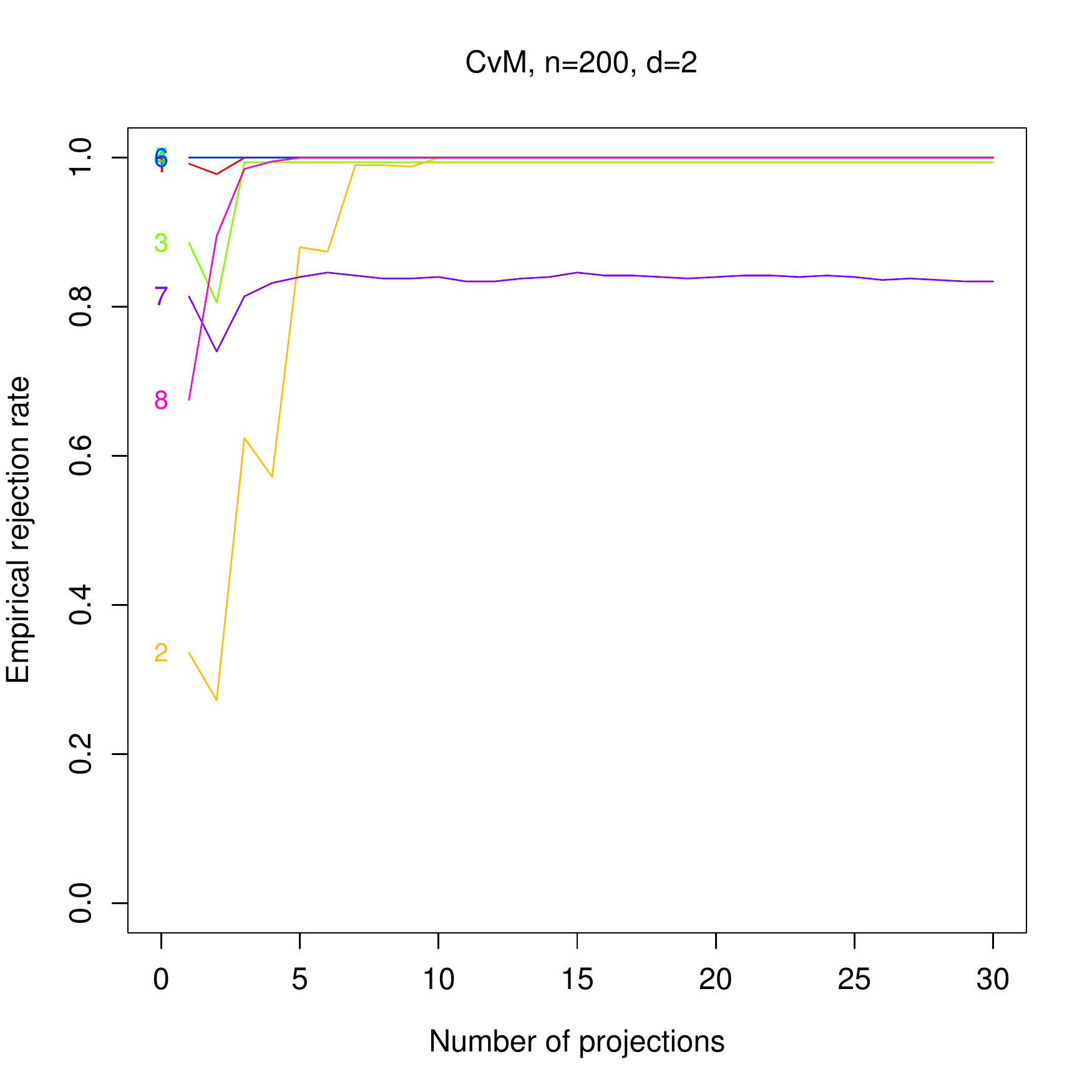}
\caption*{}
\end{minipage}
\vspace{-1cm}
\caption{Empirical powers of the CvM and KS tests with the number of projections from 1 to 30  in scenario S$k$, $k=1,\cdots,8$, based on
the significance level $\alpha=0.05$ and  the  sample sizes
$n=50$, $100$, and $200$, respectively. The first and third row correspond to the alternative $H_{k,1}$, while the second and fourth row for $H_{k,2}$, $k=1,\cdots,8$. }
\label{Kpower}
\end{figure}

\subsubsection{Dependence on the bandwidth }
\label{band}
A bandwidth $b$ is involved in the estimation of the parameter $\bbeta$, then affects the construction
of the test.  We set the bandwidth $b = cn^{-1/5}$ with $c$ being a
tuning constant.  In order to  explore the influence of $c$, we check
empirical sizes and powers using different bandwidths, with $c=2, 3, 4$ respectively, based on  $M=500$ Monte Carlo experiments and  $B=10000$ bootstrap samples   in each scenario.  
The results shown in Table  \ref{band size power}  imply  that the choice of $b$ has little  influence either on the size  or power of the test since
all three bandwidths have decent performance. In most scenarios, the bandwidth with  $c=3$ outperforms the other two in terms of  power,  except 
$H_{8,1}$ with $n=100$. As for empirical sizes, the bandwidth  with $c=2$ and $c=4$ provide values closer to the predetermined level for the sample 
size $n=50$.  As $n$ increases,  the size of the bandwidth  with $c=3$ improves dramatically, slightly better than others, especially for 
$H_{7,0}$ with $n=200$. In practice, we suggest $c=3$ as the default value.

Table \ref{size power} documents the empirical sizes and powers based on the number of projections $K=7$ and the bandwidth $b=3n^{-1/5}$.
It is shown that in all scenarios, the level is respected under the null hypothesis, with the approximation being better for a larger sample size.
%the empirical sizes of both CvM and KS tests become closer to the nominal significance level as the sample size increases. 
The power is increasing with sample size and tends towards 1 fast as the deviation increases, 
 which reveals a positive confirmation of the effectiveness of the bootstrap correction.

\begin{sidewaystable}[h!]
\caption{Empirical sizes and powers of the CvM and KS  tests based on different  bandwidth $b=cn^{-1/5}$: left of the parentheses  $c=2$, right of the parentheses    $c=4$; inside the parentheses $c=3$.   The significant level $\alpha$ is $0.05$, with sample size $n=50, 100, 200$ and the  number of projections $K=7$.}
	\label{band size power}
	\begin{center}
		\begin{tabular}{ccccccc}
			\hline
			& \multicolumn{2}{c}{$n =50$} & \multicolumn{2}{c}{$n=100$} & \multicolumn{2}{c}{$n=200$} \\
			\cmidrule(lr){2-3} \cmidrule(lr){4-5} \cmidrule(lr){6-7}
			$H_{k,\delta}$ & CvM & KS &           CvM& KS     &  CvM    & KS \\ \hline
			$H_{1,0}$                     &0.066 (0.056) 0.056       & 0.058 (0.056) 0.064       & 0.050 (0.044) 0.050        & 0.068 (0.062) 0.044       & 0.038 (0.047) 0.066       & 0.068 (0.050) 0.082       \\ 
			$H_{2,0}$                     & 0.054 (0.097) 0.034      & 0.064 (0.090) 0.050      & 0.052 (0.062) 0.074        & 0.060 (0.068) 0.082      & 0.058 (0.060) 0.068        & 0.056 (0.056) 0.072      \\ 
			$H_{3,0}$                     & 0.050 (0.045) 0.046       & 0.050 (0.044) 0.054     & 0.042 (0.048) 0.064       & 0.044 (0.065) 0.080       & 0.030 (0.049) 0.040       & 0.050 (0.054) 0.038      \\ 
			$H_{4,0}$                     & 0.038 (0.050) 0.042       & 0.044 (0.054) 0.048     & 0.030 (0.041) 0.036       & 0.040 (0.039) 0.036      & 0.036 (0.038) 0.048         & 0.036 (0.037) 0.048       \\ 
						$H_{5,0}$                     & 0.038 (0.032) 0.036       & 0.042 (0.040) 0.032      & 0.030 (0.026) 0.028       & 0.032 (0.029) 0.034       & 0.042 (0.043) 0.040        & 0.042 (0.050) 0.030        \\ 
			$H_{6,0}$                     & 0.040 (0.064) 0.058       & 0.042 (0.070) 0.068      & 0.040 (0.049) 0.042       & 0.046 (0.058) 0.054      & 0.046 (0.049) 0.054       & 0.050 (0.051) 0.046       \\
			$H_{7,0}$                     & 0.046 (0.036) 0.038       & 0.052 (0.046) 0.046      & 0.016 (0.051) 0.032       & 0.034 (0.060) 0.046       & 0.028 (0.047) 0.038       & 0.026 (0.054) 0.058      \\ 
			$H_{8,0}$                     & 0.078 (0.087) 0.070       & 0.066 (0.095) 0.080      & 0.082 (0.072) 0.064        & 0.078 (0.076) 0.070      & 0.074 (0.055) 0.066       & 0.068 (0.065) 0.074       \\ 
			&        &     &       &      &        &  \\
			$H_{1,1}$                     & 0.284 (0.386) 0.210      & 0.194 (0.260) 0.178      & 0.536 (0.796) 0.516       &  0.412 (0.606) 0.400      & 0.926 (0.988) 0.932       & 0.818 (0.960) 0.804       \\ 
			$H_{2,1}$                     & 0.136 (0.170) 0.120       & 0.120 (0.138) 0.106     & 0.314 (0.440) 0.286        & 0.260 (0.380) 0.252       & 0.688 (0.890) 0.650        & 0.572 (0.812) 0.528      \\ 
			$H_{3,1}$                     & 0.176 (0.182) 0.178      & 0.136 (0.128) 0.142     & 0.370 (0.430) 0.366        & 0.276 (0.302) 0.264      & 0.716 (0.804) 0.718       & 0.580 (0.646) 0.560      \\ 
			$H_{4,1}$                     & 0.230 (0.196) 0.020      & 0.172 (0.148) 0.158     & 0.438 (0.396) 0.444       & 0.312 (0.274) 0.302      & 0.754 (0.776) 0.788        & 0.586 (0.626) 0.598      \\ 
			$H_{5,1}$                     & 0.968 (0.980) 0.986        & 0.918 (0.946) 0.936     & 1.000 (1.000) 1.000           & 1.000 (1.000) 1.000          & 1.000 (1.000) 1.000           & 1.000 (1.000) 1.000          \\ 
			$H_{6,1}$                     & 0.966 (0.974) 0.952       & 0.926 (0.934) 0.890     & 1.000 (1.000) 0.996           & 0.998 (0.996) 0.992      & 1.000 (1.000) 1.000           & 1.000 (1.000) 1.000          \\ 
			$H_{7,1}$                     & 0.446 (0.486) 0.422      & 0.286 (0.316) 0.254      & 0.660 (0.730) 0.678        & 0.450 (0.500) 0.478        & 0.840 (0.860) 0.852        & 0.690 (0.684) 0.678      \\ 
			$H_{8,1}$                     & 0.212 (0.402) 0.204      & 0.204 (0.348) 0.198      & 0.376 (0.564) 0.384       & 0.620 (0.492) 0.352      & 0.668 (0.958) 0.646       & 0.636 (0.946) 0.610      \\ 
			&        &     &       &      &        &  \\
			$H_{1,2}$                     & 0.626 (0.826) 0.552       & 0.478 (0.670) 0.394      & 0.970 (0.996) 0.954       & 0.896 (0.976) 0.896      & 1.000 (1.000) 1.000           & 1.000 (1.000) 0.998          \\
			$H_{2,2}$                     & 0.180 (0.300) 0.166        & 0.172 (0.226) 0.144     & 0.476 (0.742) 0.490       & 0.406 (0.600) 0.354        & 0.866 (0.990) 0.874        & 0.770 (0.966) 0.774      \\ 
			$H_{3,2}$                     & 0.530 (0.528) 0.526      & 0.378 (0.410) 0.396     & 0.870 (0.916) 0.886       & 0.756 (0.830) 0.748       & 0.998 (0.994) 0.996       & 0.992 (0.980) 0.994       \\ 
			$H_{4,2}$                     & 0.816 (0.866) 0.844      & 0.720 (0.736) 0.702     & 0.986 (0.990) 0.988        & 0.970 (0.958) 0.966      & 1.000 (1.000) 0.998           & 1.000 (1.000) 0.998          \\ 
			$H_{5,2}$                     & 1.000 (1.000) 1.000          & 0.996 (0.996) 1.000     & 1.000 (1.000) 1.000           & 1.000 (1.000) 1.000          & 1.000 (1.000) 1.000           & 1.000 (1.000) 1.000           \\ 
			$H_{6,2}$                     & 0.968 (0.992) 0.972      & 0.922 (0.932) 0.940     & 0.996 (1.000) 0.998           & 0.994 (0.998) 0.998      & 1.000 (1.000) 1.000           & 1.000 (1.000) 1.000          \\ 
			$H_{7,2}$                     & 0.570 (0.626) 0.564      & 0.262 (0.322) 0.304     & 0.742 (0.702) 0.740       & 0.422 (0.422) 0.460      & 0.820 (0.842) 0.836       & 0.610 (0.666) 0.608      \\ 
			$H_{8,2}$                     & 0.468 (0.660) 0.466       & 0.430 (0.585) 0.430     & 0.712 (0.950) 0.716        & 0.672 (0.920) 0.694      & 0.898 (1.000) 0.906           & 0.882 (0.995) 0.892      \\ \hline 
		\end{tabular}
	\end{center}
\end{sidewaystable}

\begin{table}[h!]
\caption{Empirical sizes and powers of the CvM and KS  tests based on the significant level $\alpha=0.05$, with sample sizes $n=50, 100, 200$, the  number of projections $K=7$ and the bandwidth $b=3n^{-1/5}$.}
		\label{size power}
	\begin{center}
		\begin{tabular}{ccccccc}
		\hline
		& \multicolumn{2}{c}{$n =50$} & \multicolumn{2}{c}{$n=100$} & \multicolumn{2}{c}{$n=200$} \\
\cmidrule(lr){2-3} \cmidrule(lr){4-5} \cmidrule(lr){6-7}
			$H_{k,\delta}$ & CvM & KS &           CvM& KS     &  CvM    & KS \\ \hline
			$H_{1,0}$                     & 0.056       & 0.056      & 0.044        & 0.062       & 0.047        & 0.050        \\ 
			$H_{2,0}$                     & 0.097       & 0.090       & 0.062        & 0.068       & 0.060         & 0.056       \\ 
			$H_{3,0}$                     & 0.045       & 0.044      & 0.048        & 0.065       & 0.049        & 0.054       \\ 
			$H_{4,0}$                     & 0.050        & 0.054      & 0.041        & 0.039       & 0.038        & 0.037       \\ 
			$H_{5,0}$                     & 0.032       & 0.040       & 0.026        & 0.029       & 0.043        & 0.050        \\ 
			$H_{6,0}$                     & 0.064       & 0.070       & 0.049        & 0.058       & 0.049        & 0.051       \\
			$H_{7,0}$                     & 0.036       & 0.046      & 0.051        & 0.060        & 0.047        & 0.054       \\ 
			$H_{8,0}$                     & 0.087       & 0.095      & 0.072        & 0.076       & 0.055        & 0.065       \\ 
			                     &        &     &       &      &        &  \\
			$H_{1,1}$                     & 0.386       & 0.260       & 0.796        & 0.606       & 0.988        & 0.960        \\ 
			$H_{2,1}$                     & 0.170        & 0.138      & 0.440         & 0.380        & 0.890         & 0.812       \\ 
			$H_{3,1}$                     & 0.182       & 0.128      & 0.430         & 0.302       & 0.804        & 0.646       \\ 
			$H_{4,1}$                     & 0.196       & 0.148      & 0.396        & 0.274       & 0.776        & 0.626       \\ 
			$H_{5,1}$                     & 0.980        & 0.946      & 1.000            & 1.000           & 1.000            & 1.000           \\ 
			$H_{6,1}$                     & 0.974       & 0.934      & 1.000            & 0.996       & 1.000            & 1.000           \\ 
			$H_{7,1}$                     & 0.486       & 0.316      & 0.730         & 0.500         & 0.860         & 0.684       \\ 
			$H_{8,1}$                     & 0.402       & 0.348      & 0.564        & 0.492       & 0.958        & 0.946       \\ 
			   &        &     &       &      &        &  \\
			$H_{1,2}$                     & 0.826       & 0.670       & 0.996        & 0.976       & 1.000            & 1.000           \\
			$H_{2,2}$                     & 0.300         & 0.226      & 0.742        & 0.600         & 0.990         & 0.966       \\ 
			$H_{3,2}$                     & 0.528       & 0.410      & 0.916        & 0.830        & 0.994        & 0.980        \\ 
			$H_{4,2}$                     & 0.866       & 0.736      & 0.990         & 0.958       & 1.000            & 1.000           \\ 
			$H_{5,2}$                     & 1.000           & 0.996      & 1.000            & 1.000           & 1.000            & 1.000           \\ 
			$H_{6,2}$                     & 0.992       & 0.932      & 1.000            & 0.998       & 1.000            & 1.000           \\ 
			$H_{7,2}$                     & 0.626       & 0.322      & 0.702        & 0.422       & 0.842        & 0.666       \\ 
			$H_{8,2}$                     & 0.660        & 0.585      & 0.950         & 0.920       & 1.000            & 0.995       \\ \hline 
		\end{tabular}
		\end{center}
	\end{table}

\subsubsection{Power against the local alternatives}
Instead of taking a fixed deviation, we let deviation change w.r.t. sample size 
$n$ for local alternative detection. Scenarios $1, 2, 3, 8$ in the previous subsection are considered,   with a  series of new deviation coefficients   $\gamma_{n} = n^{-1/2}\sqrt{50}\delta_1$, with $\delta_1$ for S1, S2, S3, and S8 being $2/5$, $5/2$, $-1$ and $5/2$, respectively. For   scenario S$k$, $k=1,2,3,8$, the
	data is generated from
	\begin{equation*}
		Y=\bZ_k^{\top}\bbeta_k+\left\langle \bX, \brho_k\right\rangle+\gamma_n \triangle_{{\btheta}_k}\left(\bX\right)+\varepsilon.
	\end{equation*}

Figure \ref{local} gives the empirical power for a Pitman local alternative that departs from the null at a rate of $n^{-1/2}$
at the level $\alpha=0.05$ with the sample size $n=50, 100, 200, 300, 500$. We observe that the tests have moderate power, and the power improves or decreases  slightly  as $n$ increases,  but is always large than 0.2,  
which is consistent with Theorem  \ref{theorem3}.

\begin{figure}[h]
\begin{minipage}[t]{0.5\linewidth}
\centering
\includegraphics[width=3in,height=3in]{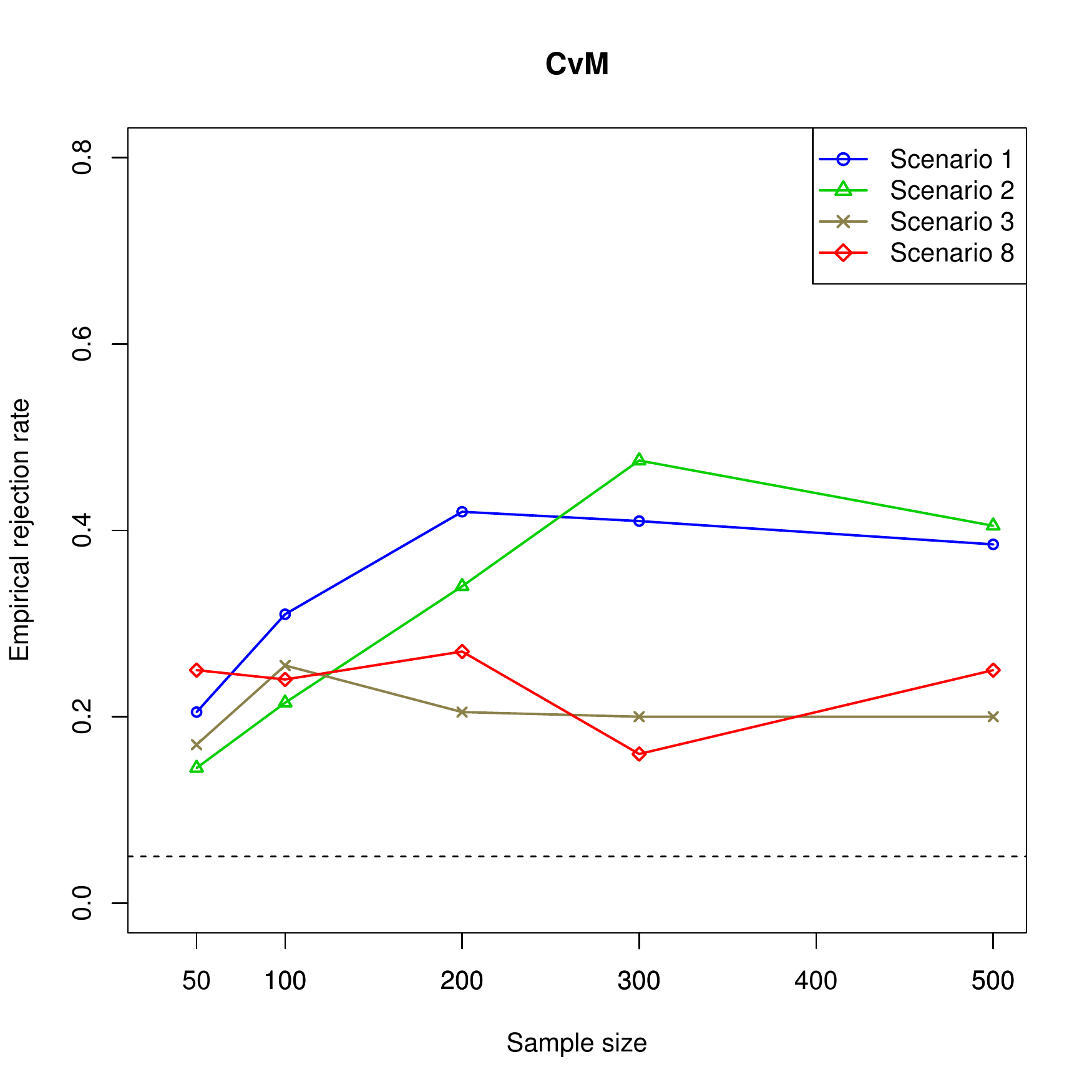}
\caption*{}
\end{minipage}
\hfill
\begin{minipage}[t]{0.5\linewidth}
\centering
\includegraphics[width=3in,height=3in]{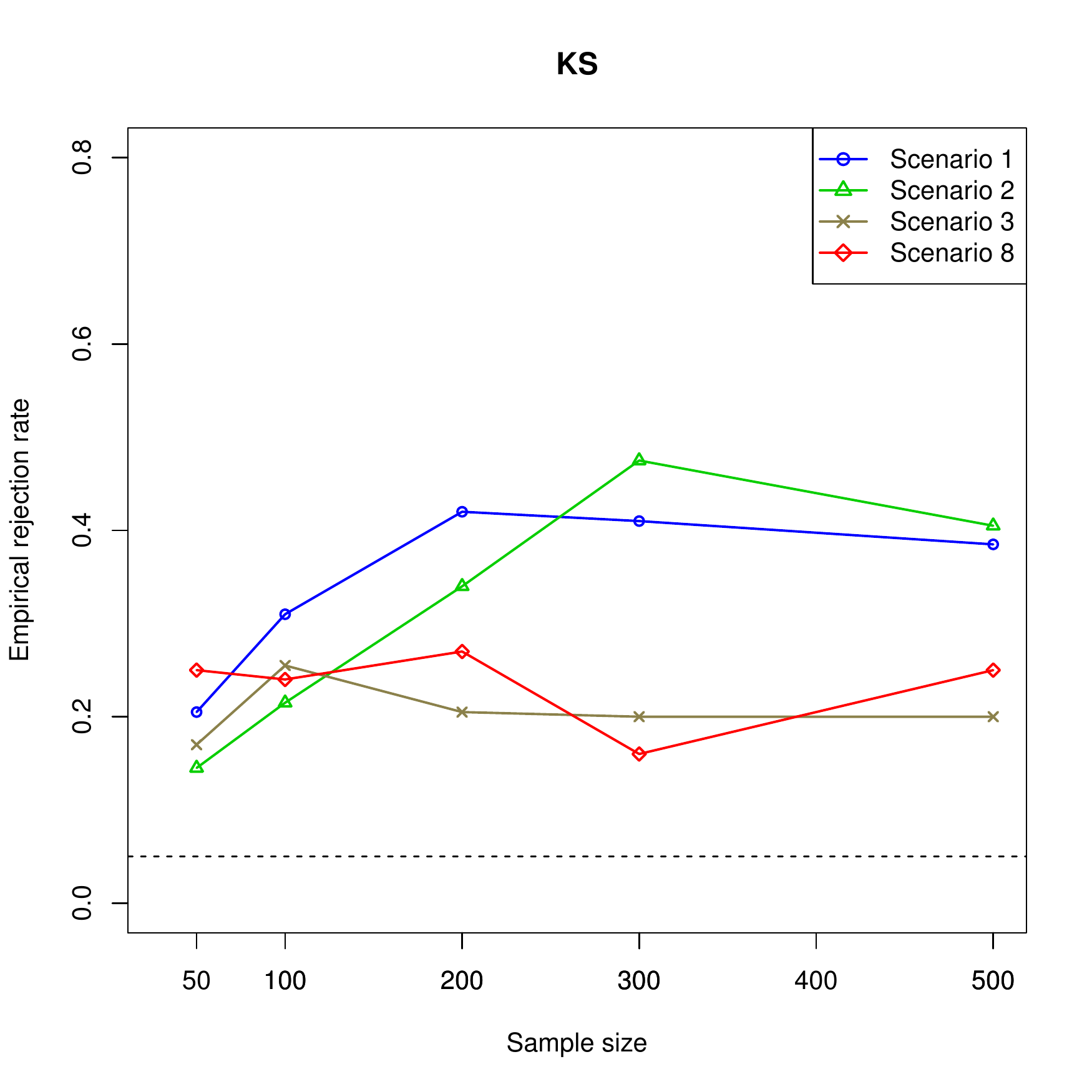}
\caption*{}
\end{minipage}
\caption{ Empirical powers for a Pitman local alternative in scenario S$k$, $k=1,2,3,8$, 
at the level $\alpha=0.05$ (dashed line) with the sample sizes $n=50, 100, 200, 300, 500$.}
\label{local}
\end{figure}

%Next,  the effectiveness of the bootstrap correction
	
	\subsection{Real data analysis}
	The proposed tests are further illustrated with two real  data sets  containing  mixed-type covariates.
	The first one is the classical Tecator data set with $215$ finely chopped pure meat samples, which is available in the R package \texttt{fda.usc}.
	Each  sample includes the fat, water, and protein content of the meat, as well as  a curve, consisting of $100$ channel points of spectrum absorbance in the wavelength range of 850nm-1050nm. These spectra  can naturally be regarded as functional data since they are densely recorded  at $100$ channel points and seem to be quite smooth, see Figure \ref{data} (a).
	The aim of the study is to predict the fat content $Y$ of a meat  sample based on its water content $Z_1$,  protein content $Z_2$ and the  near-infrared absorbance spectrum $\bX$, hence the SPFLM is chosen as a candidate model. The null hypothesis of interest is:
	\begin{equation*} 
		H_0: \quad Y  = \beta_1Z_1+ \beta_2Z_2 +\left\langle  \bX,\brho \right\rangle +\varepsilon, \quad \text{for some}~  \brho\in \mathcal{H}.	
	\end{equation*}
	We apply our proposed  method to the above test with  $B=10000$   bootstrap replications.
	%The $p$-value is 0.494 for CvM and  0.378 for 
	Table \ref{tab:tector_test} presents the $p$-values at  the different numbers of projections.
	We conclude from the table that the null hypothesis 
	should be retained under the significance level $\alpha=0.05$. This suggests strongly that there is a
	significant linear relationship between the fat content and spectrum absorbance curves, contrary to the conclusion in \citet{cuesta2019goodness}.
	The reason is that our goodness-of-fit tests  are based on the SFPLR with two  added scalar covariates, namely, water and  protein content  of the meat, while theirs are on the FLM. The difference  leads to inverse results,  as well as reveals the great value  of the test of SPFLR, which is more comprehensive and enjoys broader applications.

\begin{figure}[h]
\begin{minipage}[t]{0.5\linewidth}
\centering
\includegraphics[width=3in,height=3in]{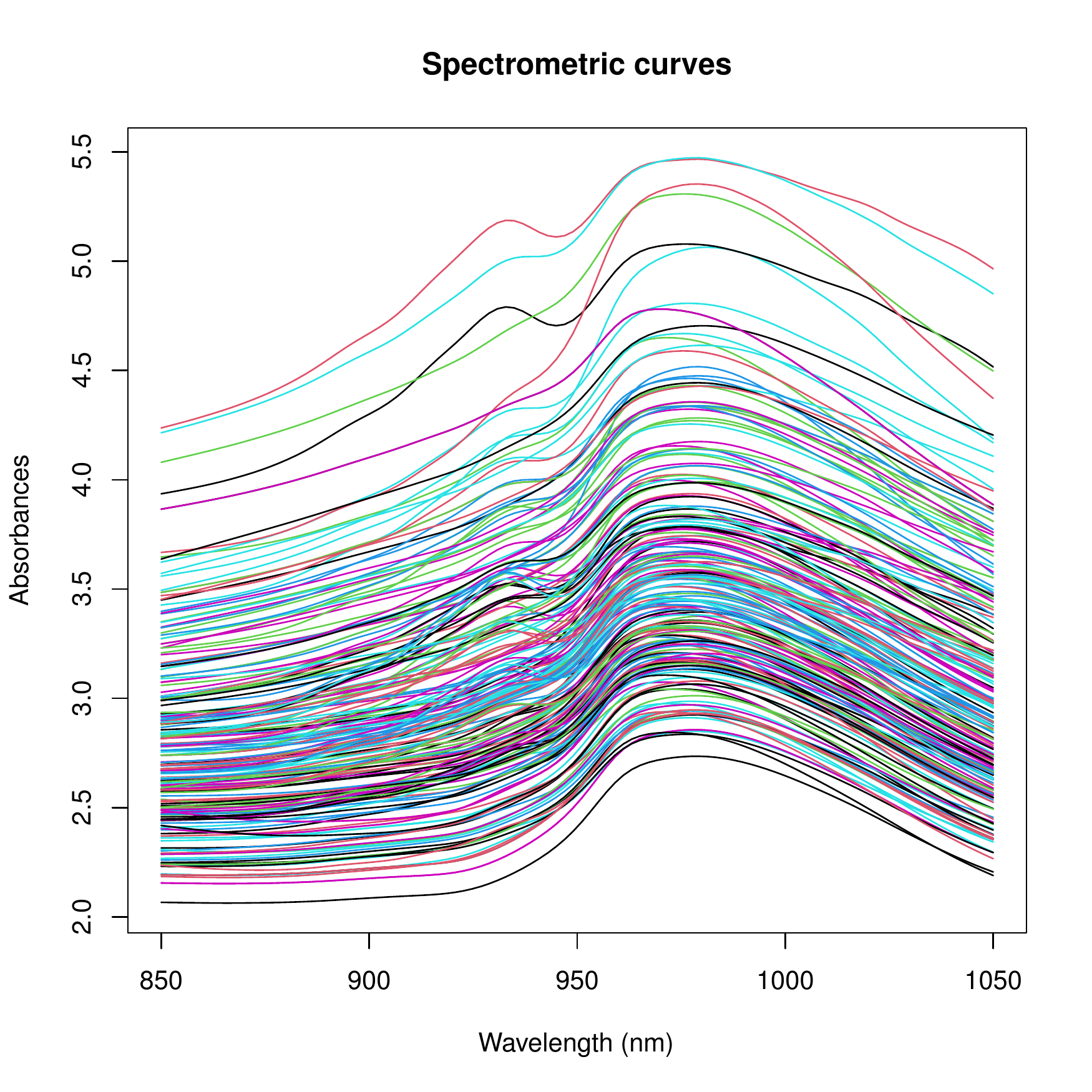}
\caption*{(a)}
\end{minipage}
\hfill
\begin{minipage}[t]{0.5\linewidth}
\centering
\includegraphics[width=3in,height=3in]{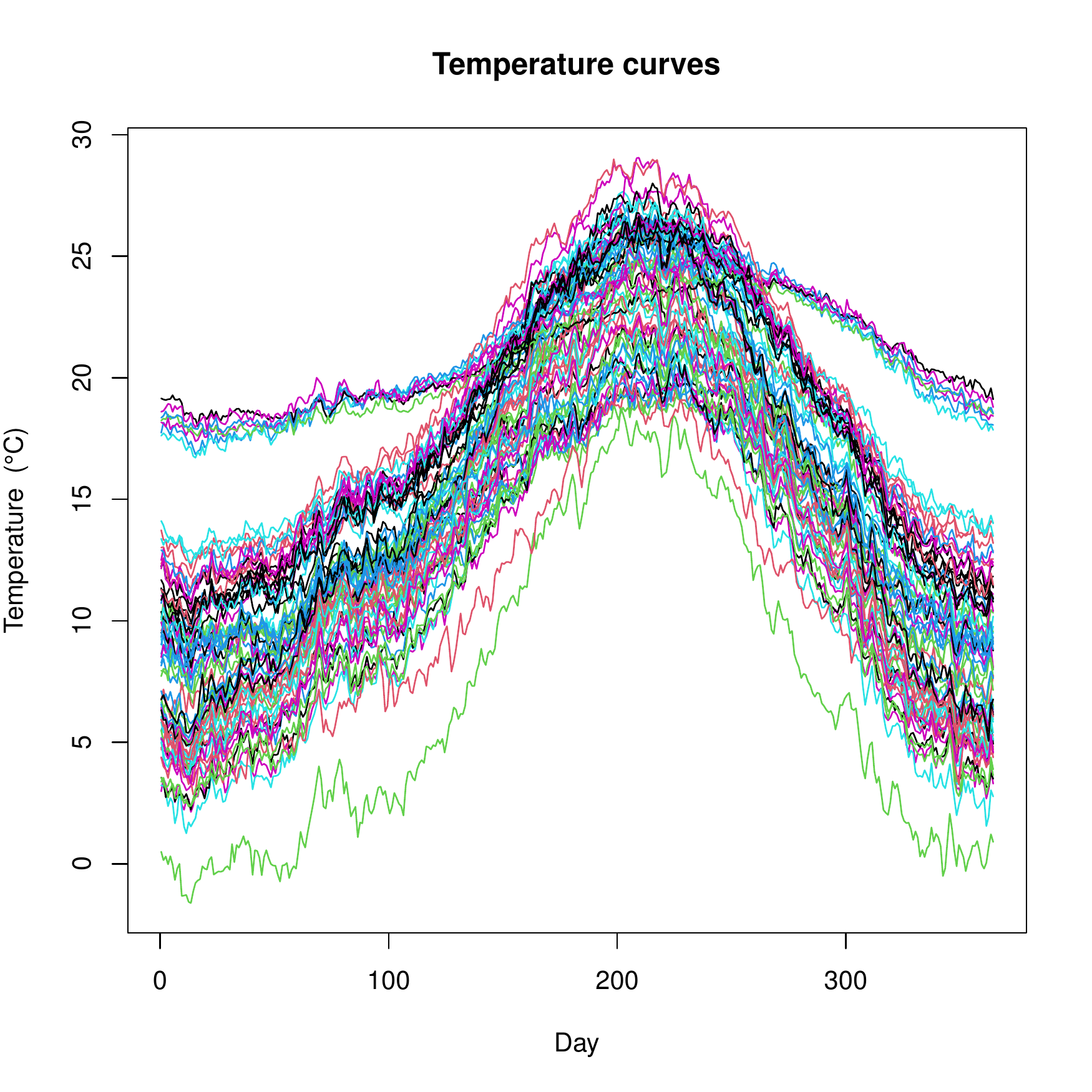}
\caption*{(b)}
\end{minipage}
\caption{(a) Plot of  spectrometric curves in the Tecator data set; (b)Daily mean temperature for the 73 Spanish weather stations.}
\label{data}
\end{figure}

\begin{table}[htbp]
  \centering
  \caption{The $p$-value of  CvM and KS  tests  with different numbers of projections in Tecator data set.}
    \begin{tabular}{cccccc}
    \hline
          & $K=2$   & $K=4$   & $K=7$   & $K=10$  & $K=12$ \\
        \hline   
    CvM   & $0.505$  & $0.626$  & $0.589$  & $0.561$  & $0.551$  \\
    KS    & $0.312$  & $0.406$  & $0.401$  & $0.394$  & $0.387$  \\
     \hline
    \end{tabular}%
  \label{tab:tector_test}%
\end{table}%	

The second example is the AEMET data set  in the R package \texttt{fda.usc},
which consists of  a daily temperature of $73$ Spanish weather stations 
during the period 1980-2009 and some other meteorological variables. 
The right plot of Figure \ref{data}  displays the functional observations of the daily temperature.	
In our study,  the goal is  to explain the daily wind speed $Y$
(averaged over 1980-2009) through  the daily temperature in each 
weather station (functional covariate $\bX$), and  the altitude of each station (scalar covariate $Z$).
The null hypothesis  is:
	\begin{equation*} 
		H_0: \quad Y  = \beta Z+\left\langle  \bX,\brho \right\rangle +\varepsilon, \quad \text{for some}~  \brho\in \mathcal{H}.	
	\end{equation*}
With the same procedures as before, Table \ref{tab:temp_test} reports the $p$-values at  different numbers of projections.
%the CvM and KS tests lead to $p$-values 0.015 and  0.023 respectively.
Thus we reject the null hypothesis and there is no evidence that the effect of the  daily temperature on the  wind speed
 is linear at level $\alpha=0.05$.

\begin{table}[htbp]
  \centering
  \caption{The $p$-values of the CvM and KS  tests  with different numbers of projections in the AEMET data set.}
    \begin{tabular}{cccccc}
    \hline
          & $K=2$   & $K=4$   & $K=7$   & $K=10$  & $K=12$ \\
          \hline
    CvM   & $0.008$  & $0.016$  & $0.018$  & $0.019$  & $0.018$  \\
    KS    & $0.020$  & $0.029$  & $0.031$  & $0.022$  & $0.026$  \\
    \hline
    \end{tabular}%
  \label{tab:temp_test}%
\end{table}%

	%\section{Extension}
	%\label{ext}

	%\section{Extension}
	%\label{ext}
	
	\section{Concluding remarks}
	\label{con}
In our paper, we consider testing functional linearity with the existence of mixed-type covariates  in the SFPLR model.  %by applying a feasible finite number of r
A robust two-step parameter estimation procedure is raised in the construction of the test statistics from the projected residual marked empirical process. Theoretical results present that our test statistics converge to a Gaussian process under the null, and are able to detect a series of local alternatives at the parametric rate. The calibration for the critical values of the test statistics is implemented by a wild bootstrap on the residuals. In practice, $p$-values computed from $K=7$ random directions are merged with  the FDR method to lower the effect of random projection choice. 
	
	 To complete the goodness-of-fit tests in the presence of both functional covariates and scalar predictors in the future, we list a few potential  directions:
	
	(a) Testing functional linearity in semi-functional partial linear quantile regression [\cite{ding2018semi}], as a promising extension of our current work.
	
	(b) Testing  linearity of scalar covariates in the functional partial linear model  [\cite{lian2011functional}].

	\section*{Acknowledgements}
		This research is supported by the National Natural Science Foundation of China (with grant numbers 71973005, 11771240, and 12026242).

	%\section{Extension}
	%\label{ext}

	%\section{Extension}
	%\label{ext}

	\begin{spacing}{1.0}
		\bibliography{refnew}

\begin{thebibliography}{}

\bibitem[Aneiros et~al., 2018]{aneiros2018bootstrap}
Aneiros, G., Ra{\~n}a, P., Vieu, P., and Vilar, J. (2018).
\newblock Bootstrap in semi-functional partial linear regression under
  dependence.
\newblock {\em Test}, 27(3):659--679.

\bibitem[Aneiros-P{\'e}rez and Vieu, 2006]{aneiros2006semi}
Aneiros-P{\'e}rez, G. and Vieu, P. (2006).
\newblock Semi-functional partial linear regression.
\newblock {\em Statistics \& Probability Letters}, 76(11):1102--1110.

\bibitem[Aneiros-Perez and Vieu, 2008]{aneiros2008nonparametric}
Aneiros-Perez, G. and Vieu, P. (2008).
\newblock Nonparametric time series prediction: A semi-functional partial
  linear modeling.
\newblock {\em Journal of Multivariate Analysis}, 99(5):834--857.

\bibitem[Benjamini and Yekutieli, 2001]{benjamini2001control}
Benjamini, Y. and Yekutieli, D. (2001).
\newblock The control of the false discovery rate in multiple testing under
  dependency.
\newblock {\em Annals of statistics}, 29(4):1165--1188.

\bibitem[Bhattacharya and Zhao, 1997]{bhattacharya1997semiparametric}
Bhattacharya, P. and Zhao, P.-L. (1997).
\newblock Semiparametric inference in a partial linear model.
\newblock {\em Annals of statistics}, 25(1):244--262.

\bibitem[Bickel and Rosenblatt, 1973]{bickel1973some}
Bickel, P.~J. and Rosenblatt, M. (1973).
\newblock On some global measures of the deviations of density function
  estimates.
\newblock {\em Annals of Statistics}, 3(6):1071--1095.

\bibitem[Billingsley, 1999]{billingsley1999convergence}
Billingsley, P. (1999).
\newblock {\em Convergence of Probability Measures}.
\newblock Wiley, New York.

\bibitem[Boente and Vahnovan, 2017]{boente2017robust}
Boente, G. and Vahnovan, A. (2017).
\newblock Robust estimators in semi-functional partial linear regression
  models.
\newblock {\em Journal of Multivariate Analysis}, 154:59--84.

\bibitem[Cai and Hall, 2006]{cai2006prediction}
Cai, T.~T. and Hall, P. (2006).
\newblock Prediction in functional linear regression.
\newblock {\em Annals of Statistics}, 34(5):2159--2179.

\bibitem[Cardot et~al., 2003]{cardot2003testing}
Cardot, H., Ferraty, F., Mas, A., and Sarda, P. (2003).
\newblock Testing hypotheses in the functional linear model.
\newblock {\em Scandinavian Journal of Statistics}, 30(1):241--255.

\bibitem[Cardot et~al., 2007]{cardot2007clt}
Cardot, H., Mas, A., and Sarda, P. (2007).
\newblock Clt in functional linear regression models.
\newblock {\em Probability Theory and Related Fields}, 138(3-4):325--361.

\bibitem[Chen, 1988]{chen1988convergence}
Chen, H. (1988).
\newblock Convergence rates for parametric components in a partly linear model.
\newblock {\em Annals of Statistics}, 16(1):136--146.

\bibitem[Cuesta-Albertos et~al., 2019]{cuesta2019goodness}
Cuesta-Albertos, J.~A., Garc{\'\i}a-Portugu{\'e}s, E., Febrero-Bande, M., and
  Gonz{\'a}lez-Manteiga, W. (2019).
\newblock Goodness-of-fit tests for the functional linear model based on
  randomly projected empirical processes.
\newblock {\em Annals of Statistics}, 47(1):439--467.

\bibitem[Delsol et~al., 2011]{delsol2011structural}
Delsol, L., Ferraty, F., and Vieu, P. (2011).
\newblock Structural test in regression on functional variables.
\newblock {\em Journal of Multivariate Analysis}, 102(3):422--447.

\bibitem[Ding et~al., 2018]{ding2018semi}
Ding, H., Lu, Z., Zhang, J., and Zhang, R. (2018).
\newblock Semi-functional partial linear quantile regression.
\newblock {\em Statistics \& Probability Letters}, 142:92--101.

\bibitem[Durbin, 1973]{durbin1973weak}
Durbin, J. (1973).
\newblock Weak convergence of the sample distribution function when parameters
  are estimated.
\newblock {\em Annals of Statistics}, 1:279--290.

\bibitem[Escanciano, 2006]{escanciano2006consistent}
Escanciano, J.~C. (2006).
\newblock A consistent diagnostic test for regression models using projections.
\newblock {\em Econometric Theory}, 22(6):1030--1051.

\bibitem[Fan and Li, 1996]{Fan1996}
Fan, Y. and Li, Q. (1996).
\newblock {Consistent model specification tests: omitted variables and
  semiparametric functional forms}.
\newblock {\em Econometrica}, 64(4):865--890.

\bibitem[Ferraty and Vieu, 2006]{ferraty2006nonparametric}
Ferraty, F. and Vieu, P. (2006).
\newblock {\em Nonparametric functional data analysis: theory and practice}.
\newblock Springer Science \& Business Media, New York.

\bibitem[Garc{\'\i}a-Portugu{\'e}s et~al., 2014]{garcia2014goodness}
Garc{\'\i}a-Portugu{\'e}s, E., Gonz{\'a}lez-Manteiga, W., and Febrero-Bande, M.
  (2014).
\newblock A goodness-of-fit test for the functional linear model with scalar
  response.
\newblock {\em Journal of Computational and Graphical Statistics},
  23(3):761--778.

\bibitem[Gonz{\'a}lez-Manteiga and Crujeiras, 2013]{gonzalez2013updated}
Gonz{\'a}lez-Manteiga, W. and Crujeiras, R.~M. (2013).
\newblock An updated review of goodness-of-fit tests for regression models.
\newblock {\em Test}, 22(3):361--411.

\bibitem[Hardle and Mammen, 1993]{hardle1993comparing}
Hardle, W. and Mammen, E. (1993).
\newblock Comparing nonparametric versus parametric regression fits.
\newblock {\em Annals of Statistics}, 21(4):1926--1947.

\bibitem[Hilgert et~al., 2013]{hilgert2013minimax}
Hilgert, N., Mas, A., and Verzelen, N. (2013).
\newblock Minimax adaptive tests for the functional linear model.
\newblock {\em Annals of Statistics}, 41(2):838--869.

\bibitem[Hoffmann-J{\o}rgensen and Pisier, 1976]{hoffmann1976law}
Hoffmann-J{\o}rgensen, J. and Pisier, G. (1976).
\newblock The law of large numbers and the central limit theorem in banach
  spaces.
\newblock {\em The Annals of Probability}, pages 587--599.

\bibitem[Horv{\'a}th and Kokoszka, 2012]{horvath2012inference}
Horv{\'a}th, L. and Kokoszka, P. (2012).
\newblock {\em Inference for functional data with applications}, volume 200.
\newblock Springer Science \& Business Media, New York.

\bibitem[Kong et~al., 2016]{kong2016partially}
Kong, D., Xue, K., Yao, F., and Zhang, H.~H. (2016).
\newblock Partially functional linear regression in high dimensions.
\newblock {\em Biometrika}, 103(1):147--159.

\bibitem[Li and Zhu, 2020]{li2020inference}
Li, T. and Zhu, Z. (2020).
\newblock Inference for generalized partial functional linear regression.
\newblock {\em Statistica Sinica}, 30(3):1379--1397.

\bibitem[Lian, 2011]{lian2011functional}
Lian, H. (2011).
\newblock Functional partial linear model.
\newblock {\em Journal of Nonparametric Statistics}, 23(1):115--128.

\bibitem[Liang, 2000]{liang2000asymptotic}
Liang, H. (2000).
\newblock Asymptotic normality of parametric part in partially linear models
  with measurement error in the nonparametric part.
\newblock {\em Journal of Statistical Planning and Inference}, 86(1):51--62.

\bibitem[Patilea et~al., 2012]{patilea2012projection}
Patilea, V., S{\'a}nchez-Sellero, C., and Saumard, M. (2012).
\newblock Projection-based nonparametric goodness-of-fit testing with
  functional covariates.
\newblock {\em arXiv preprint arXiv:1205.5578}.

\bibitem[Ramsay and Silverman, 2005]{RS05}
Ramsay, J.~O. and Silverman, B.~W. (2005).
\newblock {\em Functional Data Analysis}.
\newblock Springer, New York.

\bibitem[Resnick, 2014]{resnick2014random}
Resnick, S.~I. (2014).
\newblock {\em Random Variables, Elements, and Measurable Maps}.
\newblock Springer, New York.

\bibitem[Shin, 2009]{shin2009partial}
Shin, H. (2009).
\newblock Partial functional linear regression.
\newblock {\em Journal of Statistical Planning and Inference},
  139(10):3405--3418.

\bibitem[Stute, 1997]{stute1997nonparametric}
Stute, W. (1997).
\newblock Nonparametric model checks for regression.
\newblock {\em Annals of Statistics}, 25(2):613--641.

\bibitem[Stute et~al., 1998]{stute1998model}
Stute, W., Thies, S., and Zhu, L.-X. (1998).
\newblock Model checks for regression: an innovation process approach.
\newblock {\em Annals of Statistics}, 26(5):1916--1934.

\bibitem[Vaart and Wellner, 1996]{van1996weak}
Vaart, V.~D. and Wellner, J. (1996).
\newblock {\em Weak convergence and empirical processes: with applications to
  statistics}.
\newblock Springer Science \& Business Media, New York.

\bibitem[Wang et~al., 2016]{wang2016functional}
Wang, J.-L., Chiou, J.-M., and M{\"u}ller, H.-G. (2016).
\newblock Functional data analysis.
\newblock {\em Annual Review of Statistics and Its Application}, 3:257--295.

\bibitem[Yao et~al., 2005]{yao2005functional}
Yao, F., M{\"u}ller, H.-G., and Wang, J.-L. (2005).
\newblock Functional linear regression analysis for longitudinal data.
\newblock {\em Annals of Statistics}, 33(6):2873--2903.

\bibitem[Zheng, 1996]{zheng1996consistent}
Zheng, J.~X. (1996).
\newblock A consistent test of functional form via nonparametric estimation
  techniques.
\newblock {\em Journal of Econometrics}, 75(2):263--289.

\end{thebibliography}
	\end{spacing}

	\setcounter{footnote}{0}
	\setcounter{section}{0} 
	\renewcommand\thesection{\Alph{section}}
	
\newpage

	\begin{center}
	\hypertarget{app}{\Large\textbf{Appendix}}
\end{center}

\section{Auxiliary lemmas} 

We first decompose $T_{n,\bh}(x)$  into the following five
convenient terms:
\begin{align} \label{decom_T}
	T_{n,\bh} (x)&= n^{-1/2} \sum_{i=1}^{n} \mathds{1}_{\left\{\bX_{i}^{\bh} \le x\right\}}
	\left(Y_i -  \bZ_i^{\top}\tilde{\bbeta} - \bX_{i}^{\hat{\brho}}\right) \notag \\
	&= n^{-1/2}\left\{T_{n,\bh}^1 (x)-T_{n,\bh}^2 (x)-T_{n,\bh}^3 (x)-T_{n,\bh}^4 (x)-T_{n,\bh}^5 (x)\right\} ,
\end{align}
where %$a_n\to 0$ is a normalizing positive sequence to be determined later and
\begin{align*}
	T_{n,\bh}^1 (x):&=n^{-1/2}\sum_{i=1}^n \mathds{1}_{\left\{\bX_i^{\bh}\le x\right\}} \left(Y_i-\bZ_i^{\top}{\bbeta}-\bX_{i}^{\brho}
	\right),\\
	T_{n,\bh}^2 (x):&=n^{-1/2}\sum_{i=1}^n
	\left\langle\mathds{1}_{\left\{\bX_i^{\bh}\le x\right\}}\bX_i-\mathbb E \left[{\mathds{1}_{\left\{\bX^{\bh}\le x\right\}}\bX}\right],\hat{\brho}-\brho\right\rangle,
	\\
	T_{n,\bh}^3 (x):&=n^{1/2}\left\langle\mathbb E \left[{\mathds{1}_{\left\{\bX^{\bh}\le x\right\}}\bX}\right],\hat\brho-\brho\right\rangle,\\
	T_{n,\bh}^4 (x):&=n^{-1/2}\sum_{i=1}^n\left(\mathds{1}_{\left\{\bX_i^{\bh}\le x\right\}} \bZ_i^{\top}-\mathbb E \left[\mathds{1}_{\left\{\bX_i^{\bh}\le x\right\}} \bZ_i^{\top}\right]\right)\left(\tilde\bbeta-\bbeta\right),\\
	T_{n,\bh}^5 (x):&=n^{1/2} \mathbb E \left[\mathds{1}_{\left\{\bX^{\bh}\le x\right\}} \bZ^{\top}\right]\left(\tilde\bbeta-\bbeta\right).
\end{align*}

Throughout this section, we employ the following notations for simplicity:
\begin{align*}
	\bar{\bX}_{x,\bh} &:= \frac{1}{n} \sum_{i=1}^{n} \indi \bX_i , \quad 
	\bE_{x,\bh}^{\bX} := \E \left(\ind \bX_1\right) = \E\left(\bar{\bX}_{x,\bh}\right), \\
	\bar{\bZ}_{x,\bh} &:= \frac{1}{n} \sum_{i=1}^{n}\indi \bZ_i^T,  \quad
	\bE_{x,\bh}^{\bZ}:= \E \left(\ind \bZ_1^T\right) = \E\left(\bar{\bZ}_{x,\bh}\right). 
\end{align*}

We now reproduce two lemmas that are helpful to our proofs.

\begin{lemma} 
	\label{lemma1} (Theorem 7.5, \cite{billingsley1999convergence})
	Let $(\Omega, \cF,\mP)$ be a probability space and let $\bX$ map $\Omega$ into $\cC[0,1]$. For $F \in C[0,1]$, denote $\omega(F,h) = \sup_{x,x' \in [0,1], \vert x-x'\vert \le h} \vert F(x) - F(x')\vert $ as the modules of continuity. Suppose that $\bX,\bX^{1},\bX^{2},\cdots$ are random functions. If $\left(\bX_{t_1}^n, \cdots, \bX_{t_k}^n\right) \dto \left(\bX_{t_1},\cdots,\bX_{t_n} \right)$ holds for all $t_1,\cdots,t_k$, and if 
	\begin{equation}
		\lim_{\delta \to 0} \limsup_{n\to \infty} \mP[\omega(\bX^n,\delta) \ge \epsilon] = 0,
	\end{equation}
	for each positive $\epsilon$, then $\bX^n \dto \bX$.
\end{lemma}

\begin{lemma} (Theorem 1,  \cite{aneiros2006semi})
	\label{lemma2}
	Under Assumptions (A2) and (B1)-(B4), %there exists 
	\begin{align}
		n^{1/2}(\tilde{\bbeta} - \bbeta) \dto N(0,\sigma_{\varepsilon}^2 \bB^{-1}).
	\end{align}
\end{lemma}

\

%\begin{lemma} \label{lemma2}
%	Under assumptions \ref{A16} and \ref{A19}, we have 
%	\begin{align*}
%	k_n^3(\log k_n)^2 &= o(n^{1/2}) \\
%	\nu(d_n = k_n) &\to 1 \, as \, n \to \infty
%	\end{align*}
%\end{lemma}

Denote ${D}_i = Y_i - \bZ_i^{\top} \bbeta$ and recall $\tilde{D}_i = Y_i - \bZ_i^{\top} \tilde{\bbeta}$ in Section 2.2, we then decompose $\hat{\brho} - \brho$  into two parts:
\begin{equation}\label{rho_decom}
	\hat{\brho} - \brho = (\hat{\brho} - \widetilde{\brho}) + (\widetilde{\brho} - \brho),
\end{equation}
where $\widetilde{\brho}$ is the estimated functional coefficient based on $\left\{(\bX_i,{D}_i)\right\}_{i=1}^{n}$.	Define $\bU_n = \frac{1}{n} \sum_{i=1}^{n} \bX_i \otimes \varepsilon_i$ and  the definition of $\widetilde{\brho}$ leads to 
\begin{equation}\label{rho}
	\widetilde{\brho} = \Gamma_n^{\dagger} \Gamma_n \brho + \left(\Gamma_n^{\dagger} - \Gamma^{\dagger}\right) \bU_n + \Gamma^{\dagger} \bU_n.
\end{equation} 
Then we derive the following decomposition:
\begin{equation} \label{decom_rho}
	\widetilde{\brho} - \brho = \bL_n + \bY_n + \bS_n + \bR_n + \bT_n,
\end{equation}
where $\bL_n := -\sum_{j=k_n+1}^{\infty} \langle \brho,\be_j\rangle \be_j$, $\bY_n := \sum_{j=1}^{k_n}\left(\langle \brho,\hat{\be}_j \rangle \hat{\be}_j - \langle \brho,\be_j \rangle \be_j\right)$, $\bS_n := (\Gamma_n^{\dagger} - \Gamma^{\dagger})\bU_n$, $\bR_n := \Gamma^{\dagger}\bU_n$, and $\bT_n := \Gamma_n^{\dagger} \Gamma_n \brho - \sum_{j=1}^{k_n} \langle\brho, \hat{\be}_j \rangle \hat{\be}_j$. Note that $\bT_n = 0$ by the construction of $\Gamma_n^{\dagger}$.

Equation  (20) in \cite{aneiros2006semi} implies that
\begin{equation} 
	\label{decom_beta}
	n^{1/2}\left(\tilde{\bbeta} - \bbeta\right) = \left(n^{-1} \widetilde{\bZ_b}^{\top} \widetilde{\bZ_b}\right)^{-1} n^{-1/2} \left(S_{n1} - S_{n2} + S_{n3}\right),
\end{equation}
where $S_{n1} = \sum_{i=1}^{n} \widetilde{\bZ}_i \widetilde{m}_b(\bX_i)$, $S_{n2} = \sum_{i=1}^{n} \widetilde{\bZ}_i \left(\sum_{l=1}^{n} w_{n,b}(\bX_i,\bX_l)\varepsilon_l\right)$, and $S_{n3} = \sum_{i=1}^{n} \widetilde{\bZ}_i\varepsilon_i$. Here, we denote by $\widetilde{m}_b\left(\bX_i\right) = m(\bX_i) - \sum_{j=1}^{n} w_{n,b}\left(\bX_i,\bX_j\right)m\left(\bX_j\right)$.

\begin{lemma} \label{lemma3}
	Under Assumptions (A2) and (B1)-(B4), %there exists
	\begin{align*}
		n^{-1}\widetilde{\bZ_b}^T\widetilde{\bZ_b} \asto \bB, &\quad \quad
		n^{-1/2} S_{n1} = o_{a.s.}(1), \\
		n^{-1/2} S_{n2} = o_{a.s.}(1), &\quad \quad
		n^{-1/2} S_{n3} = n^{-1/2} \sum_{i=1}^{n} \etab_i \varepsilon_i + o_{\mP}(1).
	\end{align*}
\end{lemma}
\begin{proof}%[Proof of Lemma \ref{lemma8}]
	See Lemma 7 and proof of Theorem 1 in \cite{aneiros2006semi}.
\end{proof}

\

In light of \eqref{rho_decom}, the term $T_{n,\bh}^2(x)$ can be expressed as 
\begin{equation*}
	T_{n,\bh}^2(x) = n^{1/2} \left\langle \bar{\bX}_{x,\bh} - \bE_{x,\bh}^{\bX}, \hat{\brho} - \widetilde{\brho} \right\rangle + n^{1/2} \left\langle \bar{\bX}_{x,\bh} - \bE_{x,\bh}^{\bX}, \widetilde{\brho} - \brho \right\rangle.
\end{equation*}

The following two lemmas then yield that $T_{n,\bh}^2(x) = o_{\mP}(1)$.

\begin{lemma} \label{lemma4}
	Under Assumptions (A2),  (C1)-(C6) and (C8)-(C9), %we have
	\begin{equation*}
		 n^{1/2} \left\langle \bar{\bX}_{x,\bh} - \bE_{x,\bh}^{\bX}, \widetilde{\brho} - \brho \right\rangle  = o_{\mP}(1).
	\end{equation*}
\end{lemma}
\begin{proof}%[Proof of Lemma \ref{lemma4}]
	See Lemmas A.3 to A.6 in \cite{cuesta2019goodness}.
	%Recalling the decomposition of $\widetilde{\brho} - \brho$ in (\ref{decom_rho}), note that the results in Lemma A.3 to A.6 in \cite{cuesta2019goodness} still hold after  taking supreme  absolute value over  $x$ in $\mathbb{R}$, we can arrive at the conclusion.
\end{proof}

\begin{lemma} 
	\label{lemma5}
	Under Assumptions (A2), (B2)-(B4) and (C3), %we have
	\begin{equation*}
		 n^{1/2} \left\langle \bar{\bX}_{x,\bh} - \bE_{x,\bh}^{\bX}, \hat{\brho} 
		- \widetilde{\brho} \right\rangle  = o_{\mP}(1).
	\end{equation*} 
\end{lemma}
\begin{proof}%[Proof of Lemma \ref{lemma3}]
	By construction of the estimator of $\brho$ and Assumption (A2), one has
	\begin{align*}
		\hat{\brho} - \widetilde{\brho} &= \Gamma_n^{\dagger}\left[n^{-1} \sum_{i=1}^{n} \bX_i \otimes \left(\widetilde{D}_i - D_i\right)\right]   \\
		&= \Gamma_n^{\dagger} \left[ n^{-1}\sum_{i=1}^{n} \bX_i \otimes \bZ_i^{\top} \left(\bbeta - \tilde\bbeta \right)\right].
	\end{align*}
	Hence,
	\begin{equation*}
		n^{1/2}\left \langle \bar{\bX}_{x,\bh} - \bE_{x,\bh}^{\bX}, \hat{\brho} - \widetilde{\brho} \right\rangle = \Gamma_n^{\dagger} \left[n^{-1} \sum_{i=1}^{n} \langle \bX_i, \bar{\bX}_{x,\bh} - \bE_{x,\bh}^{\bX} \rangle \bZ_i^{\top} n^{1/2}\left(\bbeta - \tilde\bbeta \right) \right].
	\end{equation*}
	Note that $\left\{\bX_i \right\}_{i=1}^{n}$ are i.i.d. and by the Cauchy--Schwarz inequality, %there exists
	\begin{align*}
		\E \left(\left\langle \bX_i,\bar{\bX}_{x,\bh} - \bE_{x,\bh}^{\bX} \right\rangle^2\right) =& \frac{1}{n^2} \E \left\langle \bX_1, \ind \bX_1 - \E \left(\ind \bX_1\right) \right \rangle^2\\
		& +\frac{n-1}{n^2} \E \left\langle \bX_2, \ind \bX_1 - \E \left(\ind \bX_1\right) \right \rangle^2\\
		=& \frac{1}{n^2} \E \left[ \left\langle \bX_1, \ind \bX_1 \right\rangle - \left\langle \bX_1, \E \left(\ind \bX_1\right) \right\rangle \right]^2\\
		&+ \frac{n-1}{n^2} \E \left[ \left\langle \bX_2, \ind \bX_1 \right\rangle - \left\langle \bX_2, \E \left(\ind \bX_1\right) \right\rangle \right]^2 \\
		\le& \frac{2}{n^2} \left[ \E \left\langle\bX_1,\ind \bX_1 \right\rangle^2 + \E \left\langle \bX_1, \E\left(\ind \bX_1 \right) \right\rangle^2\right] \\
		&+ \frac{2n-2}{n^2} \left[ \E \left\langle\bX_2,\ind \bX_1 \right\rangle^2 + \E \left\langle \bX_2, \E\left(\ind \bX_1 \right) \right\rangle^2\right]\\
		\le& \frac{2}{n^2} \left[ \E \Vert\bX_1\Vert^2 \left\Vert \ind \bX_1\right\Vert^2 + \E\Vert\bX_1\Vert^2 \E \left\Vert\ind \bX_1 \right\Vert^2 \right] \\
		&+\frac{2n-2}{n^2} \left[ \E \Vert\bX_2\Vert^2 \E \left\Vert \ind \bX_1\right\Vert^2 + \E\Vert\bX_2\Vert^2 \E \left\Vert\ind \bX_1 \right\Vert^2 \right] \\
		\le& \frac{2}{n^2} \E \Vert\bX_1\Vert^4 + \frac{4n-2}{n^2} \left(\E \Vert\bX_1\Vert^2\right)^{2} .
	\end{align*}

Thus,
\begin{align*}
&\E \left\vert n^{-1}\sum_{i=1}^n\left\langle \bX_i,\bar{\bX}_{x,\bh} - \bE_{x,\bh}^{\bX} \right\rangle\bZ_i^{\top}\right\vert
\le  \E \left\vert \left\langle \bX_1,\bar{\bX}_{x,\bh} - \bE_{x,\bh}^{\bX} \right\rangle\bZ_1^{\top}\right\vert\\
\le & \E \left(\left\langle \bX_i,\bar{\bX}_{x,\bh} - \bE_{x,\bh}^{\bX} \right\rangle^2\right) \E \left\Vert \bZ_1^{\top}\right\Vert^2
\le \left\{\frac{2}{n^2} \E \Vert\bX_1\Vert^4 + \frac{4n-2}{n^2} \left\{\E \Vert\bX_1\Vert^2\right\}^{2}\right\}\E \left\Vert \bZ_1^{\top}\right\Vert^2.
\end{align*}
	
	Together with Assumption (C3), one  concludes that $n^{-1}\sum_{i=1}^n\left\langle \bX_i,\bar{\bX}_{x,\bh} - \bE_{x,\bh}^{\bX} \right\rangle\bZ_i^{\top}  = o_{\mP}(1)$.  Combining with  Lemma \ref{lemma2},  it follows that   
	\begin{equation*}
		n^{-1} \sum_{i=1}^{n} \left\langle \bX_i, \bar{\bX}_{x,\bh} - E_{x,\bh}^{\bX} \right\rangle \bZ_i^T n^{1/2}\left(\bbeta- \tilde\bbeta \right)  = o_{\mP}(1).
	\end{equation*}
	In addition, since $\Gamma_n^{\dagger}$ is a finite-rank operator, it is compact.  Thus, the proof is completed by Theorem 6.3.1 in \cite{resnick2014random}. 
\end{proof}

%\begin{lemma} \label{lemma7}
%	$\left\{t_{n,\bE_{x,\bh}}\right\}$ has asymptotic order between $\bO(1)$ and $\bO(k_n^{1/2})$. In addition, if $\bX$ is Gaussian and satisfies Assumption \ref{B3}, then $\sigma_{\bh}^2 := \Var(\bX^{\bh}) < \infty$ and $\lim_{n} t_{n,\bE_{x,\bh}} = \phi(x/\sigma_{\bh})$.
%\end{lemma}

\begin{lemma} \label{lemma6}
	Under Assumptions (C3), (C4), (C6) and (C7), one has
	\begin{equation*}
		 n^{1/2} \left\langle \bE_{x,\bh}^{\bX},\widetilde\brho-\brho \right\rangle   =n^{-1/2}\sum_{i=1}^{n}\left\langle\E_{x,\bh}^{\bX}, \Gamma^{\dagger}\bX_i \right\rangle \varepsilon_i+ o_{\mP}(1).
	\end{equation*}
\end{lemma}
\begin{proof}%[Proof of Lemma \ref{lemma7}]
	See the proof of Lemma A.7 in \cite{cuesta2019goodness}.
\end{proof}

\begin{lemma} \label{lemma7}
	Under condition (i), Assumptions (A2), (B1)-(B4), (C3) and (C9), 
	\begin{equation*}
		n^{1/2} \left\langle \bE_{x,\bh}^{\bX}, \hat{\brho} - \widetilde{\brho} \right\rangle  =\E_{x,\bh}^{\bX,\bZ} \bB^{-1} n^{-1/2} \sum_{i=1}^{n} \etab_i \varepsilon_i+ o_{\mP}(1),
	\end{equation*} 
with $\bE_{x,\bh}^{\bX,\bZ}$ defined as $-\E \left[\left\langle \Gamma^{-1} \bX,\E_{x,\bh}^{\bX} \right\rangle \bZ^{\top} \right]$.
\end{lemma}
\begin{proof}%[Proof of Lemma \ref{lemma6}]
	Similar to the proof of Lemma \ref{lemma5}, by replacing$\left(\bar{\bX}_{x,\bh} - \bE_{x,\bh}^{\bX}\right)$ with $\bE_{x,\bh}^{\bX}$, one obtains 
	\begin{align*}
		\left\langle n^{1/2}\bE_{x,\bh}^{\bX}, \hat{\brho} - \widetilde{\brho} \right\rangle &= \Gamma_n^{\dagger} \left[n^{-1} \sum_{i=1}^{n} \langle \bX_i,  \bE_{x,\bh}^{\bX} \rangle \bZ_i^{\top} n^{1/2}\left(\bbeta - \tilde\bbeta \right) \right]\\
		&= \Gamma^{-1} \left[n^{-1} \sum_{i=1}^{n} \langle \bX_i,  \bE_{x,\bh}^{\bX} \rangle \bZ_i^{\top} n^{1/2}\left(\bbeta - \tilde\bbeta \right) \right]+o_{\mP}(1)\\
		&= -n^{-1} \sum_{i=1}^{n} \left\langle \Gamma^{-1} \bX_i,\bE_{x,\bh}^{\bX} \right\rangle \bZ_i^{\top} n^{1/2} \left(\tilde\bbeta-\bbeta\right) + o_{\mP}(1)\\
		&= \bE_{x,\bh}^{\bX,\bZ} \bB^{-1} n^{-1/2} \sum_{i=1}^{n} \etab_i \varepsilon_i+ o_{\mP}(1),
	\end{align*}
	where the second equation follows from $\left\Vert\Gamma_n^{\dagger} - \Gamma^{\dagger} \right\Vert_{\infty} \to 0$ as well as the construction of $\Gamma^{\dagger}$ in Section 2.2. Note that the weak law of large numbers together with Assumption (C3) leads to $-n^{-1} \sum_{i=1}^{n} \left\langle \Gamma^{-1} \bX_i,\bE_{x,\bh}^{\bX} \right\rangle \bZ_i^{\top} \pto \bE_{x,\bh}^{\bX,\bZ} $. Then we can obtain the last equation using Lemma \ref{lemma3} and the continuous mapping theorem. 
\end{proof}

\begin{lemma} \label{lemma8}
	Under Assumption (C3),
	\begin{equation*}
		\sup_{x \in \mathbb{R}} \left\vert \bar{\bZ}_{x,\bh} - \bE_{x,\bh}^{\bZ}\right\vert = o_{\mP}(1).
	\end{equation*}
\end{lemma}
\begin{proof}%[Proof of Lemma \ref{lemma5}]
	It is a straightforward consequence of the weak law of large numbers in $\cH$ [see, e.g., \cite{hoffmann1976law}], the continuous mapping theorem, and Assumption (C3).
\end{proof}

Based on the above lemmas,  $T_{n,\bh}^{4}(x)= o_{\mP}(1)$. Together with the convergence mode of $T_{n,\bh}^{2}(x)$, we conclude that $T_{n,\bh}^{1}(x)$, $T_{n,\bh}^{3}(x)$ and $T_{n,\bh}^{5}(x)$ are the dominating terms in (\ref{decom_T}). More specifically, $T_{n,\bh}^{3}(x)$ and $T_{n,\bh}^{5}(x)$ represent the parameter estimation effect that arises due to the estimation of $\brho$ and $\bbeta$, respectively.

%\begin{lemma} \label{lemma9}
	%Under Assumptions (A2) and (B1)-(B4),
	%\begin{equation*}
		%n^{1/2} \bE_{x,\bh}^{\bZ} \left(\tilde\bbeta - \bbeta\right) = \bE_{x,\bh}^{\bZ} \left(n^{-1} \widetilde\bZ_b^{\top} \widetilde \bZ_b\right)^{-1} n^{-1/2} \sum_{i=1}^{n} \widetilde \bZ_i \varepsilon_i + o_{\mP}(1).
	%\end{equation*}
%\end{lemma}
%\begin{proof}
   % It is a direct conclusion from decomposition (\ref{decom_beta}) and Lemma \ref{lemma3}.
%\end{proof}

\section{Proof of theorems}	
\begin{proof}[Proof of Theorem \ref{theorem1}]
	 For any fixed $x \in \mathbb{R}$,  Lemmas \ref{lemma3} to \ref{lemma8}, together with the decompositions (\ref{KL}), (\ref{decom_rho}) and (\ref{decom_beta}), entail that 
	\begin{align*} 
		T_{n,\bh}(x) =& T_{n,\bh}^{1}(x) - T_{n,\bh}^{3}(x) - T_{n,\bh}^{5}(x) + o_{\mP}(1)   \\
		=& n^{-1/2} \sum_{i=1}^{n} \indi \varepsilon_i - n^{-1/2} \sum_{i=1}^{n} \langle \bE_{x,\bh}^{\bX},\Gamma^{\dagger} \bX_i \rangle \varepsilon_i \\
  &- \left(\bE_{x,\bh}^{\bZ}+\bE_{x,\bh}^{\bX,\bZ}\right)\bB^{-1} n^{-1/2} \sum_{i=1}^{n} \etab_i \varepsilon_i+ o_{\mP}(1)\\
         =& n^{-1/2} \sum_{i=1}^{n} \left\{A_x^i-B_x^i-C_x^i\right\} + o_{\mP}(1)  \\
		\equiv& T_{n0,\bh}(x) + o_{\mP}(1) ,
	\end{align*}
        where $A_x^i = \indi \varepsilon_i$, $B_x^{i} = \langle \bE_{x,\bh}^{\bX},\Gamma^{\dagger} \bX_i \rangle \varepsilon_i$ and $C_x^i = \left(\bE_{x,\bh}^{\bZ}+\bE_{x,\bh}^{\bX,\bZ}\right) \bB^{-1} \etab_i \varepsilon_i$. 
	Using Assumption (A2), the central limit theorem, and Slutsky's theorem, $T_{n,\bh}(x)$ is  asymptotically   normal.
	The joint asymptotic normality of $\left(T_{n,\bh}(x_1), \cdots, T_{n,\bh}(x_k)\right)$ for $\left(x_1,\cdots,x_k\right) \in \mathbb{R}^k$ follows by the Cram\'er--Wold device.
	
	%Note that $\Gamma^{\dagger} \to \Gamma^{-1}$ in the operator norm $\Vert \cdot \Vert_{\infty}$%
       %By the construction of $\Gamma^{\dagger}$ in Section 2.2, one can substitute $\Gamma^{\dagger}\bX_i$ by $\Gamma^{-1}\bX_i$, together with Lemma \ref{lemma3}, one obtains that
	Since $\bX_i$'s, $\bZ_i$'s and $\varepsilon_i$'s are i.i.d. and $\E\left[\varepsilon \vert \bX, \bZ \right] = 0$ $a.s.$,
	\begin{align*}
		& \Cov\left[n^{-1/2} \sum_{i=1}^{n} \left\{A_s^i-B_s^i-C_s^i\right\} , n^{-1/2} \sum_{i'=1}^{n} \left\{A_t^{i'}-B_t^{i'}-C_t^{i'}\right\} \right] \\
		=& \E\left[A_s^1 A_t^1\right] - \E\left[A_s^1 B_t^1\right] - \E\left[A_s^1 C_t^1\right] - \E\left[B_s^1 A_t^1\right] + \E\left[B_s^1 B_t^1\right] + \E\left[B_s^1 C_t^1\right] \\
		&- \E\left[C_s^1 A_t^1\right] + \E\left[C_s^1 B_t^1\right] + \E\left[C_s^1 C_t^1\right].
	\end{align*}
 %&\equiv C_1(s,t)-C_2(s,t)-C_3(s,t)-C_2(t,s)+C_4(s,t)+C_5(s,t)-C_3(t,s)+C_5(t,s)+C_6(s,t) .
	Applying the tower property with conditioning variables $\bX$ and $\bZ$, it follows that
	\begin{align*}
		\E\left[A_s^1 A_t^1\right] &= \int_{\left\{(\bx,\bz): \bx^{\bh} \le s\wedge t \right\}} \Var\left[Y \vert \bX=\bx,\bZ = \bz \right] dP_{(\bX,\bZ)}(\bx,\bz), \\
		\E\left[A_s^1 B_t^1\right]  &= \int_{\left\{(\bx,\bz): \bx^{\bh} \le s \right\}} \Var\left[Y \vert \bX = \bx, \bZ=\bz \right] \langle \bE_{t,\bh},\Gamma^{\dagger} \bx\rangle dP_{(\bX,\bZ)}(\bx,\bz), \\
		\E\left[A_s^1 C_t^1\right]  &= \int_{\left\{(\bx,\bz): \bx^{\bh} \le s \right\}}  \Var\left[Y \vert \bX=\bx,\bZ = \bz \right] \left(\bE_{t,\bh}^{\bZ}+\bE_{t,\bh}^{\bX,\bZ}\right) \bB^{-1} \etab_1 dP_{(\bX,\bZ)}(\bx,\bz),\\
		\E\left[B_s^1 B_t^1\right]  &= \int \Var\left[Y \vert \bX=\bx,\bZ = \bz \right] \langle \bE_{s,\bh},\Gamma^{\dagger} \bx \rangle  \langle \bE_{t,\bh},\Gamma^{\dagger} \bx \rangle dP_{(\bX,\bZ)}(\bx,\bz), \\
		\E\left[B_s^1 C_t^1\right]  &= \int \Var\left[Y \vert \bX=\bx,\bZ = \bz \right] \langle \bE_{s,\bh},\Gamma^{\dagger} \bx \rangle \left(\bE_{t,\bh}^{\bZ}+\bE_{t,\bh}^{\bX,\bZ}\right) \bB^{-1} \etab_1 dP_{(\bX,\bZ)}(\bx,\bz),\\
		\E\left[C_s^1 C_t^1\right]  &= \int \Var\left[Y \vert \bX=\bx,\bZ = \bz \right] \left(\bE_{s,\bh}^{\bZ}+\bE_{s,\bh}^{\bX,\bZ}\right) \bB^{-1} \left(\bE_{t,\bh}^{\bZ}+\bE_{t,\bh}^{\bX,\bZ}\right)^{\top} dP_{(\bX,\bZ)}(\bx,\bz).
	\end{align*}
        The construction of $\Gamma^{\dagger}$ and condition (i) imply that $\left\Vert \Gamma^{\dagger}x - \Gamma^{-1}x\right\Vert \pto 0$, then Cauchy--Schwarz inequality entails that $\E\left[A_s^1 B_t^1\right] - C_2(s,t)$, $\E\left[B_s^1 B_t^1\right] - C_4(s,t)$ and $\E\left[B_s^1 C_t^1\right] - C_5(s,t)$ converge to zero. Applying Slutsky's theorem, we obtain the finite-dimensional convergence of $T_{n,\bh}$.
	
	%We now prove (b). 
	The tightness of $T_{n,\bh}^{1}$ has been proven in Theorem 1.1 of \cite{stute1997nonparametric}. For $T_{n,\bh}^{2}$, by Cauchy--Schwarz inequality, 
	\begin{equation*}
		\sup_{x\in \mR} \left\vert T_{n,\bh}^2(x)\right\vert \le \sup_{x\in \mR} \left\Vert \bar{\bX}_{x,\bh} - \bE_{x,\bh}^{\bX}\right\Vert n^{1/2} \left\Vert \hat\brho- \brho\right\Vert,
	\end{equation*}
	 Assumption $\E \left[\left\Vert\hat\brho-\brho\right\Vert^4\right] = \bO(n^{-2})$ implies that $n^{1/2} \left\Vert \hat\brho -\brho\right\Vert \pto 0$, while
	the weak law of large numbers in $\cH$ together with Assumption (C3) leads to $\bar{\bX}_{x,\bh} - \bE_{x,\bh}^{\bX} \pto 0$ in $\cH$. We finally obtain $	\sup_{x\in \mR} \left\vert T_{n,\bh}^2(x)\right\vert \pto 0 $ by the continuous mapping theorem. 
	
	For the tightness of $T_{n,\bh}^{3}$, define 
	\begin{equation*}
		\bar{T}_{n,\bh}^{3}(u) \equiv n^{1/2}\left\langle \E \left[\indu \bX\right], \hat\brho - \brho \right\rangle,
	\end{equation*}
	with $\bU_{\bh} = F_{\bh}\left(\bX^{\bh}\right)$. Note that
	\begin{equation*}
		T_{n,\bh}^{3}(x) = \bar{T}_{n,\bh}^{3}\left(F_{\bh}(x)\right).
	\end{equation*}
	
	For $0 \le u_1 <u< u_2 \le 1$, consider
	\begin{align*}
		\bar{T}_{n,\bh}^{3}(u) - \bar{T}_{n,\bh}^{3}(u_1) &= n^{1/2} \left\langle \E\left[\induul \bX\right],\hat\brho-\brho \right\rangle, \\
		\bar{T}_{n,\bh}^{3}(u_2) - \bar{T}_{n,\bh}^{3}(u) &= n^{1/2} \left\langle \E\left[\induur \bX\right],\hat\brho-\brho \right\rangle.
	\end{align*}
	Then, applying Cauchy--Schwarz and Jensen inequalities,
	\begin{align*}
		&\E \left[ \left\vert\bar{T}_{n,\bh}^{3}(u) - \bar{T}_{n,\bh}^{3}(u_1) \right\vert^2 \left\vert\bar{T}_{n,\bh}^{3}(u_2) - \bar{T}_{n,\bh}^{3}(u) \right\vert^2 \right]\\
		\le& n^2\E \left[\left\Vert\E\left[\induul \bX\right] \right\Vert^2 \left\Vert \E\left[\induur \bX\right]\right\Vert^2 \left\Vert \hat\brho-\brho\right\Vert^4\right]\\
		=&n^2 \E \left[\left\Vert \hat\brho-\brho\right\Vert^4\right] \int\E\left[\induul \bX(t)\right]^2 dt \int\E\left[\induur \bX(t)\right]^2 dt\\
		\le& n^2 \E \left[\left\Vert \hat\brho-\brho\right\Vert^4\right] \int\E\left[\induul \bX(t)^2\right] dt \int\E\left[\induur \bX(t)^2\right] dt\\
		=& n^2 \E \left[\left\Vert \hat\brho-\brho\right\Vert^4\right] \left[F(u)-F(u_1)\right] \left[F(u_2)-F(u)\right]\\
		\le& n^2 \E \left[\left\Vert \hat\brho-\brho\right\Vert^4\right] \left[F(u_2)-F(u_1)\right]^2\\
		\le& \left[G(u_2)-G(u_1)\right]^2,
	\end{align*}
	where $F(u) = \int\E\left[\indu \bX(t)^2\right] dt$ and $G(u) = \sup_n \left\{n^2 \E \left[\left\Vert \hat\brho-\brho\right\Vert^4\right]\right\} F(u)$ are non-decreasing and continuous functions in $[0,1]$. The weak convergence of $\bar{T}_{n,\bh}^{3}$ in $D([0,1])$ is obtained by employing Theorem 13.5 in \cite{billingsley1999convergence} as $\gamma = 2$ and $\alpha = 1$. Then, $T_{n,\bh}^{3}$ converges in $\cH$ as a result of the continuous mapping theorem.
	
	Lemma \ref{lemma2} and Lemma \ref{lemma8} imply that $T_{n,\bh}^{4}(x) \pto 0$ uniformly in $x \in \mR$. As for $T_{n,\bh}^{5}$, by Assumption (C3), $\bE_{x,\bh}^{\bZ}$ can be regarded as a bounded nonrandom function of $x \in \mR$.  Together with Lemma \ref{lemma2}, we derive that $T_{n,\bh}^{5}$ weakly converge in $\cH$.
	
	As a consequence of all the above proof, together with Lemma \ref{lemma1} and  Slutsky's theorem, under $H_0^{\bh}$, $T_{n,\bh}$ weakly converges to a Gaussian process with zero mean and covariance function $K(s,t)$. % which is given by
	%\begin{equation} \label{cov_fun}
		%K(s,t)  = C_1(s,t)-C_2(s,t)-C_3(s,t)-C_2(t,s)+C_4(s,t)+C_5(s,t)-C_3(t,s)+C_5(t,s)+C_6(s,t).
	%\end{equation}	     
\end{proof}

\bigskip

\begin{proof}[Proof of Corollary \ref{Corollary1}]
The weak convergence of the empirical process $T_{n,\bh}(x)$ and  the continuous mapping theorem directly lead to that $\left\Vert T_{n,\bh} \right\Vert_{KS} \dto \Vert \cG \Vert_{KS}$.

For the CvM norm, we will prove that $\int_{\mR} T_{n,\bh}(x)^2 dF_{n,\bh}(x) \dto \int_{\mR} \cG(x)^2 dF_{\bh}(x)$.
The weak law of large numbers in $\cH$ and continuous mapping theorem yield that 
\begin{equation}
\label{F}
\sup_{x\in\mR}\left\vert F_{n,\bh}(x)-F_{\bh}(x)\right\vert\stackrel{a.s.}{\longrightarrow} 0.  
\end{equation}
It is clear that 
$\sup_{x\in\mR}\left\vert T_{n,\bh}(x)-\cG(x)\right\vert\stackrel{a.s.}{\longrightarrow} 0$. Note that
\begin{align*}
\left\vert\int_{\mR} T_{n,\bh}(x)^2 dF_{n,\bh}(x) - \int_{\mR} \cG(x)^2 dF_{\bh}(x) \right\vert\le& \left\vert\int_{\mR} \left\{T_{n,\bh}(x)^2-\cG(x)^2\right\} dF_{n,\bh}(x)\right\vert\\
&+\left\vert\int_{\mR} \cG(x)^2 \left\{dF_{n,\bh}(x)-dF_{\bh}(x)\right\}\right\vert
\end{align*}
The first term of the right-hand side of the above inequality is $o_{a.s}\left(1\right)$. The trajectories of the limiting process $\cG(x)$  are bounded and continuous almost surely.  Applying Helly--Bray Theorem  of these trajectories and taking into account (\ref{F}), one can get 
$\left\vert\int_{\mR} \cG(x)^2 \left\{dF_{n,\bh}(x)-dF_{\bh}(x)\right\}\right\vert\stackrel{a.s.}{\longrightarrow} 0$. This concludes the
proof of Corollary 1.
\end{proof}

\begin{proof}[Proof of Theorem \ref{theorem2}]
	Under the fixed alternative $H_1^{\bh}$ in (\ref{H1_h}), one has  that uniformly in $x \in \mathbb{R}$, 
	\begin{align*}
		n^{-1/2} T_{n,\bh}(x) =& n^{-1} \sum_{i=1}^{n} \left(Y_i - \bX_i^{\brho^{\ast}} - \bZ_i^{\top} \bbeta\right) \indi + n^{-1} \sum_{i=1}^{n} \left\langle \bX_i, \brho^{\ast} - \hat{\brho}\right\rangle \indi \\
		&+ n^{-1} \sum_{i=1}^{n} \bZ_i^{\top}\left(\bbeta - \tilde{\bbeta}\right) \indi \\
		=& \cG_1(x) + o_{\mP}(1),
	\end{align*}
	The first and third terms in the second equation are both Donsker by Corollary 2.10.13 in \cite{van1996weak}, thus  converging uniformly via the Glivenko--Cantelli theorem. The strong law of large numbers leads to the pointwise convergence of the second term in the second equation for $x \in \mR$, while the uniform convergence follows from the Cram\'er--Wold device and tightness, which can be proved through similar steps in the proof of Theorem \ref{theorem1}.
\end{proof}

\begin{proof}[Proof of Theorem \ref{theorem3}]
	Under the sequence of local alternatives $H_{1n}^{\bh}$ in (\ref{H1n}), one has that uniformly in $x \in \mathbb{R}$, 
	\begin{align*}
		T_{n,\bh}(x) &= n^{-1/2} \sum_{i=1}^{n} \left\{\hat{\varepsilon}_i - n^{-1/2}r\left(\bX_i\right)\right\} \indi + \frac{1}{n} \sum_{i=1}^{n}  r\left(\bX_i\right) \indi\\
		&= n^{-1/2} \sum_{i=1}^{n} \left(Y_i - n^{-1/2}r\left(\bX_i\right) - \bX_i^{\hat \brho} - \bZ_i^{\top} \tilde\bbeta \right) \indi + \frac{1}{n} \sum_{i=1}^{n}  r\left(\bX_i\right) \indi\\
		&= n^{-1/2} \sum_{i=1}^{n} \left(\varepsilon_i^{1n} + \bX_i^{\brho_0} + \bZ_i^{\top} \bbeta- \bX_i^{\hat \brho} - \bZ_i^{\top} \tilde\bbeta \right) \indi + \frac{1}{n} \sum_{i=1}^{n}  r\left(\bX_i\right) \indi\\
		&= n^{-1} \sum_{i=1}^{n} \left(\bar{A}_{x}^{i} - \bar{B}_{x}^{i} - \bar{C}_{x}^{i}\right) + \frac{1}{n} \sum_{i=1}^{n} r\left(\bX_i \right) \indi + o_{\mP}(1)\\
		&= n^{-1} \sum_{i=1}^{n} \left(\bar{A}_{x}^{i} - \bar{B}_{x}^{i} - \bar{C}_{x}^{i}\right) + \E\left[r\left(\bX\right) \indx\right] + o_{\mP}(1) \\
		&\dto \mathcal G_0(x) + \Delta_r(x),
	\end{align*}
	where $\left\{\varepsilon_i^{1n}\right\}$ denoting $\left\{Y_i - n^{-1/2}r(\bX_i) - \bX_i^{\brho} - \bZ_i^{\top}\bbeta\right\}$ are i.i.d. random sequence with mean zero under $H_{1n}^{\bh}$, thus enjoy the same properties as $\left\{\varepsilon_i\right\}$ under $H_{0}^{\bh}$. $\bar{A}_{x}^{i}, \, \bar{B}_{x}^{i},\, \bar{C}_{x}^{i}$ resemble  $A_x^i,\, B_x^i,\, C_x^i$ defined in the proof of Theorem \ref{theorem1}, except that $\brho$ under $H_0$ is replaced by $\brho_0$ under $H_{1n}^{\bh}$.  The proof is completed following similar arguments  of the proof of  Theorem \ref{theorem1}.
\end{proof}

\begin{proof}[Proof of Theorem \ref{theorem4}]
	From the simulation process of wild bootstrap, we use 
	\begin{equation*}
		\hat{\brho}^{*} = \Gamma_n^{\dagger}\Gamma_n \hat{\brho} + \left(\Gamma_n^{\dagger} - \Gamma^{\dagger}\right)\bU_{n}^{*} + \Gamma^{\dagger}\bU_{n}^{*}
	\end{equation*}
	to mimic (\ref{rho}). It follows that
	\begin{equation*}
		\hat{\brho}^{*} - \hat{\brho} = \left(\Gamma_n^{\dagger} - \Gamma^{\dagger}\right)\bU_{n}^{*} + \Gamma^{\dagger}\bU_{n}^{*},
	\end{equation*}
	where $\bU_{n}^{*} = \frac{1}{n} \sum_{i=1}^{n} \bX_i \otimes U_i^{*}$.
	
	Similarly, there exists
	\begin{equation*}
		n^{1/2} \left(\tilde{\bbeta}^{*} - \tilde{\bbeta}\right) = \left(n^{-1} \widetilde{\bZ}_b^T \widetilde{\bZ}_b\right)^{-1} n^{-1/2} \left(S_{n1}-S_{n2}^{*}+S_{n3}^{*}\right),
	\end{equation*}
	where $S_{n2}^{*}$ and $S_{n3}^{*}$ are same as $S_{n2}$ and $S_{n3}$ with $\varepsilon_i$ replaced by $U_i^{*}$.
	
	Note that $U_i^{*}$'s are independent conditional on $\bX_i, \, \bZ_i$ and have identical first and second order  moments with $\varepsilon_i$'s, hence the previous proof still holds with  $\varepsilon_i$'s replaced by $U_i^{*}$'s. Thus  uniformly in $x \in \mathbb{R}$,
	\begin{align*}
		T_{n,\bh}^{*}(x) &= n^{-1/2} \sum_{i=1}^{n} \indi \left\{Y_i^{*} - \bX_i^{\hat{\brho}} - \bZ^{\top}\tilde{\bbeta} \right\} - n^{-1/2} \sum_{i=1}^{n} \indi \left\{\bX_i^{\hat{\brho}^{*} - \hat{\brho}} - \bZ^{\top} \left(\tilde{\bbeta}^{*}\ - \tilde{\bbeta} \right)\right\} \\
		&= n^{-1/2} \sum_{i=1}^{n} \indi U_i^{*} - n^{-1/2} \sum_{i=1}^{n} \langle \bE_{x,\bh}^{\bX},\Gamma^{\dagger} \bX_i \rangle U_i^{*} \\
		&- \left(\bE_{x,\bh}^{\bZ}+\bE_{x,\bh}^{\bX,\bZ}\right) \bB^{-1} n^{-1/2} \sum_{i=1}^{n} \etab_i U_i^{*} + o_{\mP}(1)\\
		&= T_{n0,\bh}^{*}(x)+o_{\mP}(1),
	\end{align*}
	where $T_{n0,\bh}^{*}$ is the wild bootstrap version of $T_{n0,\bh}(x)$ with $\varepsilon_i$'s replaced by $U_i^{*}$'s. The proof is finished by the  similar procedure used in the proof of Theorem \ref{theorem1}.
\end{proof}

	%\section{Mathematical proofs}

	%\section{Additional tables}

\end{document}